\documentclass[a4paper,11pt,reqno,twoside]{amsart}
\usepackage{graphicx}

\usepackage{subfigure}
\usepackage{color}
\usepackage{ifpdf}
\usepackage{array}
\usepackage{slashbox}
\usepackage[T1]{fontenc}
\usepackage[latin1]{inputenc}
\usepackage{mathtools}
\usepackage{amsthm,amssymb}
\usepackage{url}

\newcommand*{\A}{\mathbb{A}}
\newcommand*{\C}{\mathbb{C}}
\newcommand*{\R}{\mathbb{R}}
\newcommand*{\Q}{\mathbb{Q}}
\newcommand*{\Z}{\mathbb{Z}}
\newcommand*{\h}{\mathrm{ht}\,}

\renewcommand*{\geq}{\geqslant}
\renewcommand*{\leq}{\leqslant}
\newtheoremstyle{erdfn}
  {}
  {}
  {\itshape}
  {}
  {\bfseries}
  {}
  { }
  {}
\newtheoremstyle{erthm}
  {}
  {}
  {\itshape}
  {}
  {\bfseries}
  {}
  { }
  {}
\newtheoremstyle{errem}
  {}
  {}
  {}
  {}
  {\bfseries}
  {}
  { }
  {}
\theoremstyle{erthm}
\newtheorem{theorem}{Theorem}[section]
\newtheorem{corollary}[theorem]{Corollary}
\newtheorem{proposition}[theorem]{Proposition}
\newtheorem{lemma}[theorem]{Lemma}

\theoremstyle{erdfn}
\newtheorem{definition}[theorem]{Definition}

\theoremstyle{errem}

\newtheorem{remark}{Remark}

\numberwithin{equation}{section}
\title[On the zeros of Weng zeta functions for Chevalley groups]%
      {On the zeros of Weng zeta functions \\ for Chevalley groups}
\author[H. Ki]{Haseo Ki}
\address{Department of mathematics, Yonsei University, Seoul, 120-749, Korea\\
and Korea Institute for Advanced Study, Seoul, Korea}
\email{haseo@yonsei.ac.kr}
\author[Y. Komori]{Yasushi Komori}
\address{
Department of Mathematics, Rikkyo University, Nishi-Ikebukuro,
Toshima-ku, Tokyo 171-8501, Japan} \email{komori@rikkyo.ac.jp}
\author[M. Suzuki]{Masatoshi Suzuki}
\address{
Department of Mathematics, Tokyo Institute of Technology,
2-12-1 Ookayama, Meguro-ku, Tokyo 152-8551, Japan}
\email{msuzuki@math.titech.ac.jp}

\AtBeginDocument{%
\mathtoolsset{showonlyrefs,mathic = true}
\begin{abstract}
We prove that all but finitely many zeros of Weng's zeta function
for a Chevalley group defined over $\Q$ are simple and on the
critical line under some reasonable geometric hypothesis.
\end{abstract}
\maketitle
}
\begin{document}

\section{Introduction}

Weng zeta functions are new objects in the theory of zeta functions 
arose from the theory of (periods of) Eisenstein series of reductive
algebraic groups $G$ defined over $\Q$. They are defined for every
standard maximal parabolic subgroup $P$ of $G$. Roughly Weng zeta
functions are linear combinations of products of the complete
Riemann zeta-function with rational function coefficients, but they
have amazing structures that may come from their origin. We start
from the most simple example.
\medskip

Let $\zeta(s)$ be the Riemann zeta function, and let
\begin{align}
\hat{\zeta}(s) &=\pi^{-s/2} \Gamma(s/2) \zeta(s).
\end{align}
Let $G=SL(2)$,
$P=B=P_{1,1}=\{\bigl( \begin{smallmatrix}\ast&\ast\\0&\ast\end{smallmatrix} \bigr)\} \subset G$ and $K=SO(2)$.
(In this case, the Borel subgroup $B$ coincides with the maximal parabolic subgroup $P$.)
Let $\mathcal H$ be the upper half plane.
Usual non-holomorphic Eisenstein series
attached to this pair $(G,P)$ is defined by the series
\[
E(s,z)=\sum_{\gamma \in P({\Z})\backslash G({\Z})} \Im\,(\gamma z)^s \quad
\]
for $z \in {\mathcal H}$, ${\rm Re}(s)>1$, and then it is continued meromorphically to the whole complex $\C$.

To understand the integral
\[
\int_{\mathcal F} E(s,z) \, d\mu(z)
\]
in a suitable sense is a fundamental problem
for the theory of harmonic analysis on $SL(2,{\Z})\backslash\mathcal H$,
where $\mathcal F$ is a standard fundamental domain of
the action of $G({\Z})=SL(2,{\Z})$ on $\mathcal H$.
Usually the above integral is studied by using the analytic truncation
\[
\Lambda^\tau E(s,z)
=
\begin{cases}
E(s,z), & y \leq e^\tau,\,z \in {\mathcal F}, \\
E(s,z)-(y^s+\frac{\hat{\zeta}(2s-1)}{\hat{\zeta}(2s)}\,y^{1-s}), & y > e^\tau,\,z \in {\mathcal F}.
\end{cases}
\]
For every $\tau \geq 0$ we have
\[
\int_{\mathcal F} \Lambda^\tau E(s,z) \, d\mu(z) =
\frac{e^{\tau(s-1)}}{s-1}-\frac{e^{-\tau s}}{s}\frac{\hat{\zeta}(2s-1)}{\hat{\zeta}(2s)}.
\]
(Therefore $\int_{\mathcal F} E(s,z) \, d\mu(z)=0$ if $0<\Re(s)<1$.)
Multiplying by $s(s-1)\hat{\zeta}(2s)$ on both sides, we define
\[
\aligned
Z(s,\tau)& = s(s-1) \cdot \hat{\zeta}(2s) \cdot
\int_{\mathcal F} \Lambda^\tau E(s,z) \, d\mu(z)  \\
& = e^{\tau(s-1)} \, s \, \hat{\zeta}(2s) -  e^{-\tau s} \, (s-1) \,\hat{\zeta}(2s-1).
\endaligned
\]
for $\tau \geq 0$. The functional equation of $E(z,s)$ derives $Z(s,\tau)=Z(1-s,\tau)$ automatically,
but we mention that the functional equation of $E(z,s)$
coming from the symmetry of the Weyl group of $SL(2)$.

A remarkable fact is that all zeros of $Z(s,\tau)$ lie on the line $\Re(s)=1/2$
and simple for every fixed $\tau \geq 0$ (see \cite{MR2142130, MR2220265}).
Weng zeta functions are of a generalization of the specialization of $Z(s,\tau)$ for $\tau=0$.
\medskip

Let $G$ be a connected semisimple algebraic group defined and split over $\Q$
and $T$ a ($\Q$-split) maximal torus of $G$.
Let $\Phi$ be the root system with respect to $(G,T)$.
Fix a Borel subgroup $B$ of $G$ containing $T$.
Then it determines the fundamental system $\Delta$ of $\Phi$.
Let $X^\ast(T)$ be the group of characters of $T$ defined over $\Q$
that is a free module of rank  $r = {\rm dim}\, T$.
Let ${\frak a}_0^\ast = X^\ast(T) \otimes {\R}$ and ${\frak a}_0 = {\rm Hom}(X^\ast(T), {\R})$.
Then ${\frak a}_0$ and ${\frak a}_0^\ast$ are real vector spaces of dimension $r$.
The root system $\Phi$ is a finite subset of $X^\ast(T)$,
therefore it is canonically embedded in ${\frak a}_0^\ast$.
For every simple root $\alpha \in \Delta$,
we have the coroot $\alpha^\vee \in {\frak a}_0$.

Let $E(\lambda,g) = E^{G/B}({\bf 1}, \lambda,g)$ be Langlands'
Eisenstein series for $\lambda \in {\frak a}_0^\ast \otimes \C
\simeq {\C}^r$ and $g \in G({\A})$.
Let $\tau \in {\frak a}_0$ be a sufficiently regular point in the sense of \cite{MR0558260},
that is, $\alpha(\tau)$ is positive and large enough for all $\alpha \in \Phi$.
Then, by using Arthur's truncation
operators $\Lambda^\tau$ acting on a space of continuous functions on $G({\Q})\backslash G({\A})^{1}$
(see \cite{MR0518111, MR0558260}),
we obtain the well-defined integral
\[
\int_{G({\Q}) \backslash G({\A})^{1}} \Lambda^\tau E(\lambda,g) \, dg
\]
which is often called the period of $E(\lambda,g)$.
The integral can be calculated explicitly
by using the root system $\Phi$ and its Weyl group $W$ of $\Phi$:
\begin{equation} \label{0427_1}
\aligned
& \int_{G({\Q}) \backslash G({\A})^{1}} \Lambda^\tau E(\lambda,g) \, dg \\
& \quad = v \sum_{w \in W} e^{\langle w\lambda - \rho, \tau \rangle}
\prod_{\alpha \in \Delta}
\frac{1}{\langle w \lambda - \rho, \alpha^\vee \rangle}
M(w,\lambda),
\endaligned
\end{equation}
where $v$ is the volume of $\{\sum_{\alpha \in \Delta} a_\alpha \alpha^\vee \,|\, 0 \leq a_\alpha <1 \}$,
$\rho$ is the half sum of all positive roots and
$M(w,\lambda)$ are usual (unnormalized) intertwining operators
$M(w,\lambda): I_B(\lambda) \to I_B(w \lambda)$ for $\lambda \in {\frak a}_0^\ast \otimes {\C}$ and $w \in W$.
If $G$ is a split semisimple group, we have
\[
M(w,\lambda)=\prod_{\alpha \in \Phi_w}
\frac{\hat{\zeta}(\langle \lambda, \alpha^\vee \rangle )}{\hat{\zeta}(\langle \lambda, \alpha^\vee \rangle +1)},
\]
where $\Phi_w=\Phi^+ \cap w^{-1}\Phi^-$
(see Chapter 5 of Langlands ~\cite{MR0419366}, Jacquet-Lapid-Rogawski ~\cite{MR1625060},
and Weng ~\cite{We3}).
\medskip

Note that we assumed that $G$ is split.
Standing on the above formula for the period of $E(\lambda,g)$, Weng introduced the ``period''
$\omega_{\Q, B}^{(G,T)}:{\frak a}_{0}^\ast \otimes \C \simeq {\C}^r \to {\C}$ by
\[
\aligned
\omega_{\Q, B}^{(G,T)}(\lambda)
=\sum_{w \in W}
\prod_{\alpha \in \Delta}
\frac{1}{\langle w \lambda - \rho, \alpha^\vee \rangle}
\prod_{\alpha \in \Phi_w}
\frac{\hat{\zeta}(\langle \lambda, \alpha^\vee \rangle)}{\hat{\zeta}(\langle \lambda, \alpha^\vee \rangle +1)}.
\endaligned
\]
In this definition, we ignored the volume $v$ and took $\tau=0$
in the right hand side of \eqref{0427_1}. 
We do not know $\tau=0$ is sufficiently regular or not in general.
Therefore we should note that $\omega_{\Q, B}^{(G,T)}$ may not be a period of Eisenstein series in a rigorous meaning.
However, we will find that $\omega_{\Q, B}^{(G,T)}$ itself is an interesting object apart from its origin.
For instance, a priori, there is no obvious reason for the restriction $\tau=0$.
However it is known that the point $\tau=0$ has a quite special meaning
for the distribution of zeros of Weng zeta functions
by studying $\tau$-version $\omega_{\Q, B}^{(SL(3),T)}(\lambda;\tau)$.
More precisely, a Weng zeta function corresponding to $\omega_{\Q, B}^{(SL(3),T)}(\lambda;\tau)$ also satisfies a functional equation
for a well-chosen $\tau$ including $\tau=0$, but it may not satisfy the Riemann hypothesis unless $\tau=0$ (see {\bf A.4} of \cite{We3} for details).
Probably, the most intrinsic reason for $\tau=0$ is the notion of stability which
makes a connection between $\omega_{\Q, B}^{(G,T)}(\lambda)$ and the theory of semi-stable principal $G$-lattices.
One of the consequences of such connection is formulated as the parabolic reduction conjecture
(\cite[Conjecture 10 and 11]{We7} or \cite[section 3]{We6})
which consists of two basic relations between total volumes of fundamental domains of reductive algebraic groups
and semi-stable volumes corresponding to parabolic subgroups.
It is a natural and deep conjecture for structures of algebraic groups and is proved for $G=SL(n)$
by Siegel, Kontsevich--Soibelman and Weng~\cite[Section 2.1]{We6}.
In particular, \eqref{0427_1} is justified for $\tau=0$ by the parabolic reduction conjecture
if $G=SL(n)$ (\cite[Section 2.1]{We7}).

The second part of the parabolic reduction conjecture claims that
the semi-stable volume can be written as ``an alternating sum'' of
total volumes associated with simple Levi factors of parabolic subgroups,
and it brings about the volume hypothesis \eqref{VF} in Section \ref{section_06}
which is a keystone of the main result Theorem \ref{thm101} below.
The first part of the parabolic reduction conjecture
is a formula of total volume by semi-stable volumes and
it is a generalization of classical results of Siegel and Langlands.
\medskip

>From the point of view for the theory of periods of automorphic functions,
it is natural to consider more general ``period''
by replacing $\hat{\zeta}$ by (completed) $L$-functions
attached to Dirichlet characters or automorphic representations.
However we restrict our interest only on the above $\omega_{\Q, B}^{(G,T)}$
for the simplicity of arguments.
\medskip

By definition, the period $\omega_{\Q, B}^{(G,T)}(\lambda)$ has a large symmetry with respect to the Weyl group,
but it is an $r$-variable function in general ($r = {\rm dim}\, T$).
A priori there is no obvious way to derive a single-variable function
from $\omega_{\Q, B}^{(G,T)}(\lambda)$ preserving the symmetry of the Weyl group.
A discovery of Weng was that
there is a canonical way to derive a (conjecturally) nice single-variable function
from $\omega_{\Q, B}^{(G,T)}(\lambda)$
if we use a standard maximal parabolic subgroup of $G$.
The key fact is the bijection
\[
\aligned
\Delta  = \{\alpha_1, \cdots ,\alpha_\ell \}
\quad & \to \quad \{\text{maximal parabolic subgroups $P$}\} \\
 \alpha_j \qquad \quad & \mapsto \hspace{3cm}  P_j
\endaligned
\]
We denote by $\alpha_P$
the simple root corresponding to the standard maximal parabolic subgroup $P$.
Following Weng, we define
\[
\omega_{\Q,P/B}^{(G,T)}(s)
= \underset{\lambda}{\rm Res}~\omega_{\Q, B}^{(G,T)}(\lambda)
\quad (s = \langle \lambda-\rho,\alpha_P^\vee \rangle),
\]
where $\underset{\lambda}{\rm Res}$ means taking residues along $(r -1)$ many hyperplanes
\[
\langle \lambda - \rho, \beta^\vee \rangle =0, \quad \beta \in \Delta\setminus\{\alpha_P \}.
\]
Then zeta functions for $((G,T);P/B)/{\Q}$ are defined by
\[
\hat{\zeta}_{P}(s):=\hat{\zeta}_{\Q, P/B}^{(G,T)}(s) := \prod_{k,h} \hat{\zeta}(ks+h) \cdot \omega_{\Q,P/B}^{(G,T)}(s),
\]
where $\prod_{k,h}\hat{\zeta}(ks+h)$ is the ``minimal'' finite product of $\hat{\zeta}(ks+h)$
in the sense that $\hat{\zeta}_{P}(s)$ has no $\hat{\zeta}(as+b)$ or $\hat{\zeta}(c)$ in its denominators.

\begin{remark}
The functions
$\omega_{\Q,P/B}^{(G,T)}(\lambda)$, $\omega_{\Q,P/B}^{(G,T)}(s)$, $\hat{\zeta}_{P}(s)$ are written as
$\omega_{\Q}^{G}(\lambda)$, $\omega_{\Q}^{G/P}(s)$, $\xi_{{\Q},\, o}^{G/P}(s)$ respectively in \cite{We3} and \cite{Ko}.
\end{remark}

Here we mention several examples of $\hat{\zeta}_P(s)$.
Let $P_{n_1,n_2} \subset SL(n)$ be the maximal parabolic subgroup attached to $n=n_1+n_2$.
Then we find that
\[
\hat{\zeta}_{\Q, P_{n-1,1}}^{SL(n)}(s) = \hat{\zeta}_{\Q, P_{1,n-1}}^{SL(n)}(s)
 = \sum_{h=1}^n \frac{Q_{h}(s)}{P_h(s)}\cdot \hat{\zeta}(s+h),
\]
where $P_h$, $Q_h$ are polynomials satisfying ${\rm deg}\,P_h \leq n-1$, ${\rm deg}\,Q_h < {\rm deg}\,P_h$
(see also section \ref{section_04}).
For example,
\[
\hat{\zeta}_{\Q, P_{1,1}}^{SL(2)}(s) = \frac{\hat{\zeta}(s+2)}{s}-\frac{\hat{\zeta}(s+1)}{s+2},
\]
\[
\aligned
 \hat{\zeta}_{\Q, P_{2,1}}^{SL(3)}(s) 
& = \left( \frac{\hat{\zeta}(2)}{s}-\frac{1}{2(s+1)}\right)\hat{\zeta}(s+3)
-\frac{\hat{\zeta}(s+2) }{s(s+3)}
-\left( \frac{\hat{\zeta}(2)}{s+3}-\frac{1}{2(s+2)}\right)\hat{\zeta}(s+1).
\endaligned
\]
In the following examples, each term of $\hat{\zeta}_P(s)$
consists of the product of several $\hat{\zeta}(s)$:
\[
\aligned
\hat{\zeta}_{\Q,\, P_{2,2}}^{SL(4)}&(s) =
R_1(s)\,\hat{\zeta}(s+3)\hat{\zeta}(s+4) + R_2(s)\,\hat{\zeta}(s+3)^2 \\
+ &\, R_3(s)\,\hat{\zeta}(s+1)\hat{\zeta}(s+2)
+R_4(s)\,\hat{\zeta}(s+2)^2 + R_5(s)\,\hat{\zeta}(s+2)\hat{\zeta}(s+3), 
\endaligned
\]
\[
\aligned
\hat{\zeta}_{\Q,\, P_{\rm long}}^{Sp(4)}(s) & =
R_1(s)\,\hat{\zeta}(s+3)\hat{\zeta}(2s+4) 
+ R_2(s)\,\hat{\zeta}(s+2)\hat{\zeta}(2s+4) \\
& \quad + R_3(s)\,\hat{\zeta}(s+2)\hat{\zeta}(2s+3)  
+ R_4(s)\,\hat{\zeta}(s+1)\hat{\zeta}(2s+3), 
\endaligned
\]
\[
\aligned
\hat{\zeta}_{\Q,\,P_{\rm long}}^{G_2}&(s) =
R_1(s)\,\hat{\zeta}(s+3)\hat{\zeta}(2s+4)\hat{\zeta}(3s+6)
+ R_2(s)\,\hat{\zeta}(s+1)\hat{\zeta}(2s+3)\hat{\zeta}(3s+4) \\
& + R_3(s)\,\hat{\zeta}(s+2)\hat{\zeta}(2s+4)\hat{\zeta}(3s+5)
+ R_4(s)\,\hat{\zeta}(s+2)\hat{\zeta}(2s+3)\hat{\zeta}(3s+5) \\
&+ R_5(s)\,\hat{\zeta}(s+2)\hat{\zeta}(2s+3)\hat{\zeta}(3s+4)
+ R_6(s)\,\hat{\zeta}(s+2)\hat{\zeta}(2s+4)\hat{\zeta}(3s+6),
\endaligned
\]
where $R_i$'s are certain rational functions. For more examples, see
Appendix A of \cite{We3}, 
but note that a changing of variable $s \mapsto as+b$ is necessary in order to be consistent with the notation in this paper 
and the right hand sides of (7) and (8) in \cite[\S A.2]{We3} should be exchanged.
More generally, if $G=SL(n)$ and $P=P_{n-1,1}$, under certain normalization,
the zeta function $\hat{\zeta}_{P}(s)$ is equal to the rank $n$ zeta function
over the rational number field $\Q$  which was introduced in \cite{MR2258077} as a generalization of Dedekind zeta functions
and Artin-Weil zeta functions
from a geometric point of view (see \cite[Section 1]{We6} and also section \ref{section_04}).
\medskip

By definition, $\hat{\zeta}_{P}(s)$ are meromorphic functions in $\C$ having only finitely many poles.
In addition Weng observed that $\hat{\zeta}_{P}(s)$ satisfy standard functional equations
for several examples of $((G,T);P/B)$,
and conjectured that it holds for general pairs (if $G$ is a classical semisimple group at least):
\medskip

\noindent
{\bf Conjecture 1} There exists $c=c((G,T);P/B) \in \Q$ such that
${\hat{\zeta}_{P}(-c-s)=\hat{\zeta}_{P}(s)}$.
\medskip

The conjectural functional equation derives the corresponding Riemann hypothesis:
\medskip

\noindent
{\bf Conjecture 2}
All zeros of $\hat{\zeta}_{P}(s)$ lie on the critical line ${\Re}(s)=-c/2$.
\medskip

As for Conjecture 1, initially, Weng proved it for
$SL(n)~(n=2,3,4,5)$, $Sp(4)$, $SO(8)$ and $G_2$. Successively,
H.~Kim-Weng proved the case of $(SL(n), P_{n-1,1})$ for arbitrary $n
\geq 2$ (unpublished). Finally, the second author established the
conjectural functional equations of $\hat{\zeta}_{P}(s)$ for general
$((G,T);P/B)$, and determined the value $c((G,T);P/B)$ explicitly in \cite{Ko}.
As for Conjecture 2, its validity was known for ten
examples of $((G,T);P/B)$, namely, $G=SL(n)$ $(n=2,3,4,5)$, $Sp(4)$,
$G_2$ and their arbitrary maximal parabolic subgroup $P$, 
and further, the simplicity of zeros was also known for these pairs
(\cite{MR2142130, Ki09, MR2220265, MR2310295, MR2488589,
MR2482116}). This is surprising, because a linear combination of
zeta functions has a lot of off-line zeros in general even if it has
a functional equation (e.g. \cite{MR1763856, MR2059482, MR1220477}).

Roughly, the known proof of Conjecture 2 for the above special cases
consists of two parts. The first one is to show that all but
finitely many zeros of $\hat{\zeta}_{P}(s)$ lie on the critical
line, and the second one is to remove the possibility of off-line
zeros. Methods for these two parts have different nature. 
We may say that the first part is of the problem for the zeros of higher
position, and the other part is of the problem for low-lying zeros.
The latter problem is difficult and interesting than the former one
as well as in the theory of classical zeta/$L$-functions. 
In fact, though the proof of the first part for the above ten examples were improved
and unified in ~\cite{Ki09}, unfortunately we still need a numerical
computation for the latter part. In the present paper, we prove the
following result for the first part of Conjecture 2 by generalizing
the method of ~\cite{Ki09} and by using the theory of ~\cite{Ko}:

\begin{theorem}[Weak Riemann Hypothesis]\label{thm101}
Let $G$ be a Chevalley group defined over $\Q$,
in other words, $G$ is a connected semisimple algebraic group defined over $\Q$
endowed with a maximal $(\Q$-$)$split torus $T$.
Let $B$ be a Borel subgroup of $G$ containing $T$.
Let $P$ be a maximal parabolic subgroup of $G$ defined over $\Q$ containing $B$.

Assume the volume hypothesis \eqref{VF} in Section {\rm \ref{section_06}} below if $G \not=SL(n)$.

Then all but finitely many zeros of $\hat{\zeta}_{\Q,P/B}^{(G,T)}(s)$ are simple and on the critical line
of its functional equation.
\end{theorem}

\begin{remark}
It is expected that $\hat{\zeta}_{\Q,P/B}^{(G,T)}(s)$ has no zeros
outside the critical line as well as in known cases $G=SL(n)$ $(n=2,3,4,5)$, $Sp(4)$, $G_2$
that correspond to root systems of type $A_n$ ($n=1,2,3,4$), $C_2 (\simeq B_2)$ and $G_2$.
However we have no idea how to prove it in general.
\end{remark}

\begin{remark}
Similar results hold if we replace $\hat{\zeta}(s)$ by a suitable (completed) $L$-function,
because we use only standard analytic properties of $\hat{\zeta}(s)$ in the proof.
\end{remark}

\begin{remark}
The volume hypothesis \eqref{VF} is essential for our proof of Theorem \ref{thm101}
and is not a technical hypothesis.
In fact, it is a part of the parabolic reduction conjecture as mentioned above.
Unfortunately, to the best of the authors' knowledge,
the volume hypothesis has not been yet proved except for the case $G=SL(n)$ in spite of its significance,
but we believe that it is natural and reasonable one.
Therefore, we leave it for other work.
Conversely, Theorem \ref{thm101} shows
a relation between the Riemann hypothesis
and the deep structure of algebraic group
like the parabolic reduction conjecture.
\end{remark}

We note that the known proof of Conjecture 2 about ten examples
mentioned just before Theorem \ref{thm101} depends on explicit formulas
of $\hat{\zeta}_{\Q,P/B}^{(G,T)}(s)$ presented by Weng (see \cite{We3}).
On the other hand, the method of the proof of the
functional equation in \cite{Ko} is completely abstract, and hence
it does not depend on an individual root system. In this paper, we
will prove the above main theorem in such a way that we refine the
proof of the Riemann hypothesis of $\hat{\zeta}_{\Q,P/B}^{(G,T)}(s)$
demonstrated in \cite{Ki09, MR2220265, MR2310295, MR2488589, MR2482116}
by using terms of the abstract root system as well as in the
proof of the functional equations in \cite{Ko}.
\medskip

The paper is organized as follows. In section 2, we rewrite zeta
functions $\hat{\zeta}_{\Q,P/B}^{(G,T)}(s)$ in terms of abstract
root systems, and define zeta functions more rigorously. In section
3, we sketch the outline of the proof of Theorem \ref{thm101}. In
section 4, we briefly review the story of the proof of Theorem
\ref{thm101} restricting on the most simple cases
$(SL(n),P_{n-1,1})$. In sections 5 to 7, we carry out the scheme of
section 3 without proofs of lemmas and propositions. The section
that the proof of each lemma or proposition is accomplished is
mentioned here. In section 8, we prepare further notations and
auxiliary lemmas for the proofs of results in sections 5 to 7.
Finally, in sections 9 to 15, we fill the details of proofs of
lemmas and propositions in sections 5 to 7, and complete the proof of
Theorem \ref{thm101}.

%

\subsection*{Acknowledgements}
The authors would like to thank Lin Weng for his valuable comments on the previous version of the paper.
The authors also thank referees for their comments and suggestions. 
The first named author was supported by the National Research Foundation of Korea (NRF) grant funded by the Korea government (MSIP)(No. 2014R1A2A2A01002549).
The third named author was supported by KAKENHI (Grant-in-Aid for Young Scientists (B)) No. 21740004 and No. 25800007.

%

%
\section{Definition of Weng zeta functions for $(G,P)$}
%
%

\subsection{Root system and the Weyl group}

Let $V$ be an $r$-dimensional real vector space equipped with an inner product $\langle \cdot, \cdot \rangle$.
Let $\Phi \subset V$ be a (reduced) root system of rank $r$
and $\Delta=\{\alpha_1,\cdots,\alpha_r\}$, its fundamental system.
Let $\alpha^\vee=2\alpha/\langle \alpha,\alpha \rangle$ be the coroot associated with $\alpha \in \Phi$.
Let $\Lambda=\{\lambda_1,\cdots,\lambda_r\}$ be the set of fundamental weights
satisfying $\langle \lambda_i,\alpha_j^\vee \rangle = \delta_{ij}$.
Let $\Phi^+$ be the corresponding positive system of $\Phi$
and $\Phi^-=-\Phi^+$ so that $\Phi = \Phi^+ \cup \Phi^-$.
Let
\[
\rho=\frac{1}{2}\sum_{\alpha \in \Phi^+} \alpha =\sum_{j=1}^r \lambda_j
\]
be the Weyl vector.
Let $\h\alpha^\vee=\langle \rho,\alpha^\vee \rangle$ be the height of $\alpha^\vee$.

Let $W$ be the Weyl group generated by simple reflections $\sigma_j=\sigma_{\alpha_j}:V\to V$
attached to simple roots $\alpha_j \in \Delta$.
We denote  the identity of $W$ by ${\rm id}$.
For $w \in W$, we put
\[
\Phi_w=\Phi^+ \cap w^{-1} \Phi^-
\]
and let $l(w)=\vert \Phi_w \vert$ be the length of $w$.
Let $w_0$ be the longest element of $W$.
Then we have $w_0^2={\rm id}$, $w_0 \Delta=-\Delta$ and $w_0 \Phi^+=\Phi^-$.

We fix an integer $p$ with $1 \leq p \leq r$.
Let $\Phi_p$ be the root system normal to the fundamental weight $\lambda_p$.
A fundamental system of $\Phi_p$ is given by $\Delta_p = \Delta \setminus \{ \alpha_p \}$.
Let $\Phi_p^+=\Phi^+ \cap \Phi_p ~(\subset \Phi^+)$ be the corresponding positive system of $\Phi_p$.
Then $\Phi_p = \Phi_p^+ \cup \Phi_p^-$ with $\Phi_p^-=-\Phi_p^+$.
Let
\[
\rho_p=\frac{1}{2}\sum_{\alpha \in \Phi_p^+} \alpha.
\]
Let $W_p$ be the subgroup of $W$
generated by simple reflections $\{\sigma_j=\sigma_{\alpha_j} \,|\, \alpha_j \in \Delta_p\}$.
Let $w_p$ be the longest element of $W_p$.
Then $w_p^2={\rm id}$, $w_p\Delta_p=-\Delta_p$ and $w_p \Phi_p^+=\Phi_p^-$.

\begin{definition}
Define the subset ${\frak W}_p$ of $W$ by
\[
{\mathfrak W}_p:=\{w \in W ~|~ \Delta_p \subset w^{-1}(\Delta \cup \Phi^{-})\}.
\]
\end{definition}
\noindent
Clearly ${\rm id}$, $w_0$, $w_p$ belong to ${\frak W}_p$. The condition
\begin{equation}\label{0201}
\Delta_p \subset w^{-1}(\Delta \cup \Phi^{-})
\end{equation}
plays an important role in several places of the proof of Theorem \ref{thm101}.
\medskip

To describe the functional equation,
we introduce the constants
\begin{equation}\label{0202}
c_p = 2 \, \langle \lambda_p - \rho_p, \alpha_p^\vee \rangle
\end{equation}
as well as ~\cite[section 2]{Ko}.
Then $c_p$ is a positive integer for every $1 \leq p \leq r$.

\subsection{Definition of Weng zeta functions in terms of abstract root system} \hfill \\

Let $G$ be a connected semisimple algebraic group and let $\frak g$
be its Lie algebra of $G$. Let $T$ be a maximal torus of $G$, and
let ${\frak g}_\alpha=\{X \in {\frak g} \,|\, {\rm
Ad}(t)X=\alpha(t)X \}$ for each character $\alpha \in X^\ast(T)$.
Then the set $\Phi=\Phi(G,T) = \{ \alpha \in X^\ast(T) \,|\, {\frak
g}_\alpha \not=\emptyset \}$ is finite, and makes a root system (in
the vector space $X^\ast(T) \otimes \R$). Conversely, for a given
root system $\Phi$, there exists a connected semisimple algebraic
group $G=G(\Phi)$ defined over a prime field having $\Phi$ as its
root system with respect to a split maximal torus $T$ of $G$ by the
fundamental theorem of Chevalley. The group $G(\Phi)$ is called a
Chevalley group of type $\Phi$ (or split group, since it has a
maximal torus which is split over the prime field).
\medskip

Therefore we can deal with Weng zeta functions for Chevalley groups defined over $\Q$
by using the language of abstract root systems only.
Now we define Weng zeta functions again in terms of abstract root systems.
\medskip

Let $\zeta(s)$, $\hat{\zeta}(s)$ as above, and let
\begin{align}
\xi(s) &=s(s-1)\hat{\zeta}(s) = s(s-1)\pi^{-s/2}\Gamma(s/2) \zeta(s).
\end{align}
Note that Weng \cite{We3} and Komori~\cite{Ko} use the notation $\xi(s)$
to indicate our $\hat{\zeta}(s)$
as well as Langlands~\cite{MR0419366} et al.

\begin{definition}[Periods for $(\Phi,\Delta)$] \label{def_0201}
Let $\Phi$ be an irreducible root system and
let $\Delta$ be a fundamental system of $\Phi$.
For $\lambda \in V$, we define
\[
\omega_\Delta^{\Phi}(\lambda)=
\sum_{w \in W}
\prod_{\alpha \in \Delta} \frac{1}{\langle w\lambda-\rho,\alpha^\vee \rangle}
\prod_{\alpha \in \Phi_w} \frac{\hat{\zeta}(\langle \lambda, \alpha^\vee \rangle )}{\hat{\zeta}(\langle \lambda, \alpha^\vee \rangle+1)}.
\]
Here we understand that the second product equals $1$ if $\Phi_w=\Phi^+ \cap w^{-1} \Phi^{-}=\emptyset$.
\end{definition}

\begin{definition}[Periods for $(\Phi,\Delta,p)$] \label{def_0202}
Let $\Phi$ be an irreducible root system of rank $r$
with a fundamental system $\Delta=\{\alpha_1,\cdots,\alpha_r\}$.
Let $1 \leq p \leq r$.
Take the coordinate of $V$ as
\[
\lambda = \sum_{j=1}^{r}(1+s_j)\lambda_j=\rho+\sum_{j=1}^{r} s_j \lambda_j
\quad (\lambda \in V)
\]
so that $\langle \lambda-\rho,\alpha^\vee \rangle=\sum_{j=1}^{r} a_j s_j$
for $\alpha^\vee=\sum_{j=1}^r a_j \alpha_j^\vee$,
and write $\Delta_p=\{\beta_1,\cdots,\beta_{r-1}\}$.
Then we define $\omega_p:{\C} \to{\C}$ by
\begin{equation}\label{e:omegap}
\aligned \omega_p(s)=\omega_{\Delta,p}^{\Phi}(s) & =
\underset{\langle \lambda-\rho,\beta_1^\vee \rangle=0}{\rm Res}
\cdots \underset{\langle \lambda-\rho,\beta_{r-1}^\vee
\rangle=0}{\rm Res}
 \omega_\Delta^{\Phi}(\lambda) \\
& = \underset{s_1=0}{\rm Res} \cdots \underset{s_{p-1}=0}{\rm Res}
\quad \underset{s_{p+1}=0}{\rm Res} \cdots \underset{s_r=0}{\rm
Res}~\omega_\Delta^{\Phi}(\lambda),
\endaligned
\end{equation}
where the variable $s_p$ is written as $s$.
\end{definition}
\begin{remark}
The function $\omega_p(s)$ is well defined,
since it does not depend on the ordering of the set of simple roots $\Delta_p$ by Proposition 2.2 of \cite{Ko}.
\end{remark}

\begin{remark}
Let $G=G(\Phi)$ be a Chevalley group, and let $B$ a Borel subgroup
containing the maximal split torus $T$. Let $P~(\supset B)$ be a
maximal parabolic subgroup of $G$ corresponding to the simple root
$\alpha_p$. Then functions $\omega_{\Delta}^\Phi(\lambda)$ and
$\omega_{\Delta,p}^{\Phi}(s)$ are the periods
$\omega_{\Q,B}^{(G,T)}(\lambda)$ and $\omega_{\Q,P/B}^{(G,T)}(s)$
respectively.
\end{remark}
\begin{definition}\label{def_0203}
For $w \in {\frak W}_p$ and $(k,h)\in{\Z}^2$, we define
\[\aligned
N_{p,w}(k,h) & =\sharp \{\alpha \in w^{-1}\Phi^- \,|\, \langle \lambda_p, \alpha^\vee \rangle=k,~\h\alpha^\vee=h \}, \\
N_{p}(k,h) & =\sharp \{\alpha \in \Phi \,|\, \langle \lambda_p, \alpha^\vee \rangle=k,~\h\alpha^\vee=h \}
\endaligned\]
and
\[
M_p(k,h)=\underset{w \in {\frak W}_p}{\rm max}
\bigl( N_{p,w}(k,h-1)-N_{p,w}(k,h) \bigr).
\]
\end{definition}
If $k \geq 1$ or $h \geq 1$, we have
\[\aligned
N_{p,w}(k,h)&=\sharp \{\alpha \in \Phi_w \,|\, \langle \lambda_p, \alpha^\vee \rangle=k,~\h\alpha^\vee=h \}, \\
N_{p}(k,h)&=\sharp \{\alpha \in \Phi^+ \,|\, \langle \lambda_p, \alpha^\vee \rangle=k,~\h\alpha^\vee=h \},
\endaligned\]
because a root $\alpha \in \Phi$ is either $\alpha \in \Phi^+$ or $\alpha \in \Phi^-$.

\begin{definition}[Weng zeta function for $(\Phi,\Delta, p)$] \label{def_0204}
Using the above notation, we define
\[
\hat{\zeta}_p(s)
= \hat{\zeta}_{\Delta,p}^{\,\Phi}(s)
= \omega_p(s) \prod_{k=0}^\infty \prod_{h=2}^\infty \hat{\zeta}(ks+h)^{M_p(k,h)}.
\]
\end{definition}
\begin{remark}
Note that $M_p(k,h)=0$ except for finitely many pairs of integers $(k,h)$.
The function $\hat{\zeta}_p(s)$
coincides with the zeta function $\xi_{{\Q},o}^{G/P}(s)$ of \cite{We3}
if $\Phi$ is the root system attached to $(G,T)$
and $\alpha_p$ corresponds to the maximal parabolic $P$.
Moreover, we find that
the numbers $I(G/P)$ and $J(G/P)$ defined in \cite[section 2]{We3} are given by
\[
I(G/P)=\sum_{(k,h) \in {\Z}^2,\,k\not=0} M_p(k,h),
\quad J(G/P)=\sum_{h \in {\Z}}M_p(0,h),
\]
because the product
$\prod_{k=0}^\infty \prod_{h=2}^\infty \hat{\zeta}(ks+h)^{M_p(k,h)}$
is minimal in the sense that it eliminates all $\hat{\zeta}(as+b)$ and $\hat{\zeta}(c)$
appearing in the denominators of $\omega_{p}(s)$ (\cite[Theorem 2.3, section 5]{Ko}).
\end{remark}

\begin{remark}
We will have $M_p(k,h)=N_{p,w_0}(k,h-1) - N_{p,w_0}(k,h)$ in Corollary \ref{cor_0807} below.
\end{remark}

\begin{theorem}[Komori, Theorem 2.4 of \cite{Ko}] \label{prop_0205}
Let $c_p$ be the constant defined in \eqref{0202}.
Then $\hat{\zeta}_p(s)$ satisfies the functional equation
\[
\hat{\zeta}_p(s)=\hat{\zeta}_p(-c_p-s).
\]
\end{theorem}

\begin{remark}
See Appendix 2 for table of numbers $c_p$.
The numbers $c_p$ can be interpreted geometrically
in terms of the first Chern class of tangent bundles on a homogeneous space
or an index of a homogeneous space.
See Appendix 3 for details.
\end{remark}

Now the reformulation of Theorem \ref{thm101} in terms of $\hat{\zeta}_p(s)$ is obvious,
thus we omit such a direct reformulation of Theorem \ref{thm101}.
See Theorem \ref{thm_0703} and Theorem \ref{thm_0704}
that are reformulations of Theorem \ref{thm101} in terms of entire functions $\xi_p(s)$ defined below.

%
\section{Outline of the proof} \label{section_03}
%

Define the entire function
$\xi_p(s)=Q(s)\, \hat{\zeta}_p(s)$
by multiplying a suitable  polynomial.
\medskip

\noindent
\textbf{1.} At first we construct an entire function $\varepsilon_p(s)$ satisfying
\[
(\star) \quad\quad  \xi_p(s)=\varepsilon_p(s) \pm \varepsilon_p(-c_p-s).
\]
Here the sign $\pm$ depends on the degree of $Q(s)$. The formula
$(\star)$ plays a central role on the current line of the proof.
More precisely, Theorem \ref{thm101} is reduced to a study of the
location of zeros of $\varepsilon_p(s)$ by $(\star)$, and
fortunately, it is less hard to investigate the zeros of
$\varepsilon_p(s)$ than $\hat{\zeta}_p(s)$. A kind of the formula
$(\star)$ was used in every (known) proof of the Riemann hypothesis
of $\hat{\zeta}_p(s)$ for $G=SL(n)$ ($n=2,3,4,5$), $Sp(4)$, $G_2$.
This step is described more precisely in section \ref{section_05}.
\medskip

\noindent
\textbf{2.} Successively we investigate the location of the zeros of $\varepsilon_p(s)$.
The aim of this second step is to show that
\begin{enumerate}
\item[(i)]  the number of zeros of $\varepsilon_p(s)$ in $\Re(s) \geq -c_p/2$ is finitely many,
\item[(ii)] in a left half-plane, $\varepsilon_p(s)$ has no zero in a region  $\Re(s) \leq -\kappa \log(|{\Im}(s)|+10)$.
\end{enumerate}
This step is described more precisely in section \ref{section_06}.
\medskip

\noindent
\textbf{3.} Finally we prove Theorem \ref{thm101} by using the results of the second step.
The main tool of this step is the Hadamard product formula of $\varepsilon_p(s)$.
Essentially this part is a modification of the method which
was established in \cite{Ki09} by the first author.
This step is described more precisely in section \ref{section_07}.
\medskip

The most essential part of our proof of Theorem \ref{thm101} is in
the second step. To complete the proof of the second step, we need a
detailed study of a  structure of $\Z$-grading of root systems
$\Phi$ (or ${\rm Lie}(G)$) induced from fundamental weights $\lambda_p$. 
We review the flow from the first step to the second
step in the next section by taking up the cases of
$(SL(n),P_{n-1,1})$ as an example.

The first step is established by an algebraic way via the general
theory of root systems and Weyl groups attached to maximal parabolic
subgroups $P$. The argument of the final step is achieved by a
purely analytic way. Frequently, basic analytic properties of
$\zeta(s)$ will be used in demonstrating Theorem
\ref{thm101}.

%
\section{Cases of $(SL(n),P_{n-1,1})$} \label{section_04}
%
%
Let $P=P_{n-1,1}=\Bigl\{ \Bigl( \begin{smallmatrix} A & B \\ 0 & D \end{smallmatrix} \Bigr) \in SL(n) \, |\, A \in GL(n-1),\, D \in GL(1) \Bigr\}$ be the standard maximal parabolic subgroup of $G=SL(n)$
attached to the partition $n=(n-1)+1$.
For these special cases, Theorem \ref{thm101} was established by Weng (unpublished)
after the work of the first author~\cite{Ki09}.
As a review  of the proof of general cases,
we sketch the proof of Theorem \ref{thm101} for these special cases
restricting it on the first and the second step.
\medskip

Let $B$ be the standard Borel subgroup consisting of upper triangular matrices,
and let $T$ be the standard split torus in $B$,
that is, the torus consisting of diagonal matrices.
The root system $\Phi$ associated with $T$ is of type $A_{n-1}$,
and is realized as
$\Phi^+=\{ e_{i} - e_{j} \,|\, 1 \leq i < j \leq n \}$,
where   $\{e_i \,|\, 1 \leq i \leq n\}$ is the standard orthonormal basis of ${\R}^n$.
The set of simple roots $\Delta$ associated with $B$ is
$\Delta=\{\alpha_1:=e_1-e_2, \cdots, \alpha_{n-1}:=e_{n-1}-e_{n}\}$,
and the half sum of positive roots is
\[
\rho=\frac{1}{2}\bigl( (n-1)e_1 + (n-3)e_2 + \cdots - (n-3)e_{n-1} - (n-1)e_n \bigr).
\]
The Weyl group $W$ is identified with the symmetric group on $n$ letters $S_n$
by the convention $w(e_{i}-e_{j})=e_{w(i)}-e_{w(j)}$ ($w \in S_n$).
The longest element $w_0$ of $W$ is given by the permutation $\Bigl(\begin{smallmatrix} 1 & 2 & \cdots & n \\ n & n-1 & \cdots & 1\end{smallmatrix}\Bigr)$,
and $-w_0\rho=\rho$.
The maximal parabolic subgroup $P=P_{n-1,1}$ corresponds to the simple root $\alpha_{n-1}=e_{n-1}-e_{n}$,
and has the Levi decomposition $P=MN$ with $M \simeq GL(n-1)$. We have
\[
\aligned
\Phi_P^+ &= \Phi_{n-1}^+ = \{ e_{i} - e_{j} \,|\, 1 \leq i < j \leq n-1 \}, \quad
\Delta_P = \Delta_{n-1} =\{\alpha_{1}, \cdots, \alpha_{n-2}\}, \\
\rho_P &= \rho_{n-1} = \frac{1}{2}\bigl( (n-2)e_1 + (n-4)e_2 + \cdots - (n-4)e_{n-2} - (n-2)e_{n-1} \bigr),
\endaligned
\]
and the fundamental weight corresponding to $P$ is
\[
\lambda_P=\lambda_{n-1}=
\frac{1}{n}(e_1+ \cdots +e_{n-1}-(n-1)e_n).
\]
The Weyl group $W_P=W_{n-1}$ is the subgroup of $W$
which corresponds to the symmetric group on $(n-1)$ letters $\{1,2,\cdots, n-1 \}$.
The longest element $w_P=w_{n-1}$ of $W_P$ is given by the permutation
$\Bigl(\begin{smallmatrix} 1 & 2 & \cdots & n-1 \\ n-1 & n-2 & \cdots & 1\end{smallmatrix}\Bigr)$,
and $-w_P \rho_P = \rho_P$.
\medskip

Comparing with general cases, structures of $(\Phi^+ \setminus
\Phi_P^+) \cap w^{-1}\Phi^{\pm}$ and $\Phi_P^+ \cap
w^{-1}\Phi^{\pm}$ $(w \in W)$ in the present case are rather simple,
therefore it is not hard to find that $\prod_{h=2}^{n-1}
\hat{\zeta}(h) \cdot \hat{\zeta}(s+n)$ is the minimal product of
zeta functions and zeta values which eliminates all zeta functions
and zeta values in the denominator of each term of
$\omega_{\Q,P/B}^{(G,T)}(s)$. Hence we have
\[
\hat{\zeta}_P(s) = \omega_{\Q,P/B}^{(G,T)}(s) \, \prod_{h=2}^{n-1} \hat{\zeta}(h) \cdot \hat{\zeta}(s+n).
\]
More precisely, we obtain
\[
\hat{\zeta}_P(s) = \sum_{h=1}^{n} R_h(s)\, \hat{\zeta}(s+h)
\]
with rational functions
\[
R_h(s) = \sum_{{w \in {\frak W}_P}\atop{|(\Phi^+ \setminus \Phi_P^+) \cap w^{-1} \Phi^{+}|=h-1}} C_w
\prod_{\alpha \in (w^{-1}\Delta) \setminus \Phi_P}
\frac{1}{\langle \lambda_P, \alpha^\vee \rangle s + \h\alpha^\vee -1 },
\]
where ${\frak W}_P = \{ w \in W \, |\, w\Delta_P \subset \Delta \cup \Phi^- \}$ and
\[
C_w =
\prod_{\alpha \in (w^{-1}\Delta) \cap (\Phi_P \setminus \Delta_P)} \frac{1}{\h\alpha^\vee -1 }
\quad \hat{\zeta}(2)^{|\Delta_P^+ \cap w^{-1}\Phi^{+}|} \!\!\!\!
\prod_{\alpha \in (\Phi_P^+ \setminus \Delta_P) \cap w^{-1}\Phi^{+}}
\frac{\hat{\zeta}(\h\alpha^\vee +1) }{\hat{\zeta}(\h\alpha^\vee)}.
\]
The constant $c_p$ appearing in the functional equation \eqref{0202}
is calculated as follows:
\[
c_P=2\langle \lambda_P - \rho_P, \alpha_P^\vee \rangle =  n.
\]
We have $R_h(-n-s)=R_h(s)$ for every $1 \leq h \leq n$
by considering the involution $w \mapsto w_0ww_P$ of ${\frak W}_P$ according to \cite[section 4]{Ko}.
Hence we obtain the functional equation
\[
\hat{\zeta}_P(s) = \hat{\zeta}_P(-n-s).
\]
by using the functional equation of $\hat{\zeta}(s)$.
Define
\[
\xi_P(s) = \prod_{k=0}^n (s+k) \cdot \hat{\zeta}_P(s).
\]
Then we find that $\xi_P(s)$ is an entire function,
and each
\[
P_h(s):= \frac{1}{(s+h)(s+h-1)}R_h(s)\prod_{k=0}^n (s+k) \quad (1 \leq h \leq n)
\] is a polynomial.
Set
\[
\varepsilon_P(s) =
\begin{cases}
\displaystyle{\sum_{h=n/2}^{n} P_h(s) \xi(s+h)}, & \text{$n$ is even}, \\
\displaystyle{\sum_{h=(n+3)/2}^{n} P_h(s) \xi(s+h) + \frac{1}{2}P_{\frac{n+1}{2}}(s) \xi\left(s+\frac{n+1}{2}\right)}, & \text{$n$ is odd}
\end{cases}
\]
with $\xi(s)=s(s-1)\hat{\zeta}(s)$. Then we have
\[
\xi_P(s) = \varepsilon_P(s) + (-1)^{n-1} \varepsilon_P(-n-s)
\]
by the functional equation of $\hat{\zeta}_P(s)$. This is the
formula $(\star)$ of section \ref{section_03}.

By definition, we have
\[
\varepsilon_P(s) =
\begin{cases}
P_2(s) \xi(s+2), \qquad \text{$n=2$}, & \\
\displaystyle{P_n(s) \xi(s+n)\left( 1 + \sum_{h=n/2}^{n-1} \frac{P_h(s)}{P_n(s)} \frac{\xi(s+h)}{\xi(s+n)}\right)}, & \\
\qquad \text{$n > 2$ is even}, \\
\displaystyle{P_n(s) \xi(s+n)\left(1 + \sum_{h=(n+3)/2}^{n-1} \frac{P_h(s)}{P_n(s)} \frac{\xi(s+h)}{\xi(s+n)}
+ \frac{1}{2}\frac{P_{\frac{n+1}{2}}(s)}{P_n(s)} \frac{\xi\left(s+\frac{n+1}{2}\right)}{\xi(s+n)} \right)}, & \\
\qquad \text{$n$ is odd.}
\end{cases}
\]
It is well known that $\xi(s)$ has no zeros in the right half plane $\Re(s) \geq 1$.
Thus $\xi(s+n)$ has no zeros in the right half plane $\Re(s) \geq -n/2$.
Furthermore $|\xi(s+h)/\xi(s+n)|<1$ for $\Re(s) > -(h+n-1)/2$ (see Lemma \ref{lem901} below).
Therefore, if
\begin{equation}\label{150128_01}
\deg P_n(s) \geq \deg P_h(s) + 1 \quad \text{for every $1<h<n$},
\end{equation}
we can conclude that the number of zeros of $\varepsilon_P(s)$ in
the right half plane $\Re(s) \geq -n/2$ is finitely many. This is {\bf 2} (i) of
section \ref{section_03}. It is not hard to find that $\deg P_h(s)
\leq n-3$ for every $1<h<n$ ($n \geq 3$) (see Lemma \ref{lem1001}).
Therefore, we have inequality \eqref{150128_01} if $\deg P_n(s) = n-2$.
However, it is not trivial to prove $\deg P_n(s) = n-2$ for general $n \geq 2$.
In fact, by definition of $P_n(s)$, it is equivalent to the nonvanishing of the
$\Q$-linear combination of products of special values of
$\hat{\zeta}(s)$:
\[
\sum_{{w \in {\frak W}_P}\atop{|(\Phi^+ \setminus \Phi_P^+) \cap w^{-1} \Phi^{+}|=n-1}} C_w
\prod_{{\alpha \in (w^{-1}\Delta) \setminus \Phi_P}\atop{|(w^{-1}\Delta) \setminus \Phi_P|=1}}
\frac{1}{\langle \lambda_P, \alpha^\vee \rangle} \not=0,
\]
and it is highly nontrivial.
Fortunately, in the case of $SL(n)$, this problem were solved
by using the volume formula of \cite[section 4.7]{MR2310297}
(see the volume hypothesis \eqref{VF} in section \ref{section_06} and the proof of Lemma \ref{lem1003} in section \ref{section_10}). 
In contrast with {\bf 2} (i), {\bf 2} (ii) 
is provided easily by using the Stirling formula for the gamma function.

%
%
\section{The first step of the proof of Theorem \ref{thm101}} \label{section_05}
%
%
\subsection{A modification to entire functions}

The zeta function $\hat{\zeta}_{\Delta,p}^{\Phi}(s)$ is meromorphic on $\C$
and has finitely many poles.
In order to prove Theorem \ref{thm101},
we consider the entire function
$\xi_{\Delta,p}^{\Phi}(s)$ which is a polynomial multiple of $\hat{\zeta}_{\Delta,p}^{\Phi}(s)$.
\medskip

At first, we recall formula (2.8) of ~\cite{Ko}:
\[
\aligned
\omega_{\Delta,p}^{\Phi}(s)
& = \sum_{w \in {\frak W}_p}
\prod_{\alpha \in (w^{-1}\Delta) \setminus \Delta_p}
\frac{1}{\langle \lambda_p, \alpha^\vee \rangle s + \h\alpha^\vee -1 } \\
& \times
\prod_{\alpha \in \Phi_w \setminus \Delta_p}
\hat{\zeta}(\langle \lambda_p,\alpha^\vee \rangle s + \h\alpha^\vee)
\prod_{\alpha \in (-\Phi_w)}
\hat{\zeta}(\langle \lambda_p,\alpha^\vee \rangle s + \h\alpha^\vee )^{-1}.
\endaligned
\]
Here we understand that each product equals $1$ if its index set is empty.
\begin{definition} \label{def_0501}
We define the product of zeta functions
\[
F_p(s):=\prod_{\alpha \in \Phi^{-}}
\hat{\zeta}(\langle \lambda_p, \alpha^\vee \rangle s + \h\alpha^\vee),
\]
and the meromorphic function
\[
Z_p(s):=F_p(s)\,\omega_{\Delta,p}^{\Phi}(s).
\]
\end{definition}
Then we have
\[
Z_p(s) = \sum_{w \in {\mathfrak W}_p} f_{p,w}(s) \, g_{p,w}(s),
\]
with rational functions $f_{p,w}$ and products of zeta functions $g_{p,w}$:
\begin{align*}
f_{p,w}(s) &= \prod_{\alpha \in (w^{-1} \Delta)\setminus \Delta_p}
\frac{1}{\langle \lambda_p, \alpha^\vee \rangle s + \h \alpha^\vee -1}, \\
g_{p,w}(s) &= \prod_{\alpha \in (w^{-1}\Phi^-) \setminus \Delta_p}
\hat{\zeta}(\langle \lambda_p, \alpha^\vee \rangle s + \h \alpha^\vee),
\end{align*}
since
$\Phi^- \setminus \Phi_w = \Phi^- \cap w^{-1}\Phi^{-}$
and $(\Phi_w \setminus \Delta_p) \cup (\Phi^- \setminus \Phi_w) =
(w^{-1}\Phi^-)\setminus \Delta_p$.
\medskip

We modify the above formula of $g_{p,w}$ so that all coefficients
$\langle \lambda_p, \alpha^\vee \rangle$ of $s$ in $g_{p,w}$ will be
nonnegative integers. We see that $g_{p,w}(s)$ is
\[
\aligned & \prod_{\alpha \in (w^{-1}\Phi^{-})\, \cap \Phi^{-}}
\hat{\zeta}(\langle \lambda_p, \alpha^\vee \rangle s + {\rm
ht}\,\alpha^\vee) \prod_{\alpha \in ((w^{-1}\Phi^{-})\, \cap
\Phi^{+})\setminus \Delta_p}
\hat{\zeta}(\langle \lambda_p, \alpha^\vee \rangle s + {\rm ht}\,\alpha^\vee) \\
 = &
\prod_{\alpha \in (w^{-1}\Phi^{-})\, \cap \Phi^{-}}
\hat{\zeta}(1 - \langle \lambda_p, \alpha^\vee \rangle s - {\rm ht}\,\alpha^\vee)
\prod_{\alpha \in ((w^{-1}\Phi^{-})\, \cap \Phi^{+})\setminus \Delta_p}
\hat{\zeta}(\langle \lambda_p, \alpha^\vee \rangle s + {\rm ht}\,\alpha^\vee) \\
 = &
\prod_{\alpha \in (w^{-1}\Phi^{-})\, \cap \Phi^{-}}
\hat{\zeta}(\langle \lambda_p, -\alpha^\vee \rangle s + {\rm ht}\,(-\alpha^\vee)+1)
\prod_{\alpha \in ((w^{-1}\Phi^{-})\, \cap \Phi^{+})\setminus \Delta_p}
\hat{\zeta}(\langle \lambda_p, \alpha^\vee \rangle s + {\rm ht}\,\alpha^\vee) \\
 = &
\prod_{\alpha \in (w^{-1}\Phi^{+})\, \cap \Phi^{+}}
\hat{\zeta}(\langle \lambda_p, \alpha^\vee \rangle s + {\rm ht}\, \alpha^\vee +1)
\prod_{\alpha \in ((w^{-1}\Phi^{-})\, \cap \Phi^{+})\setminus \Delta_p}
\hat{\zeta}(\langle \lambda_p, \alpha^\vee \rangle s + {\rm ht}\,\alpha^\vee)
\endaligned
\]
for each $w \in {\mathfrak W}_p$ by using the functional equation
$\hat{\zeta}(s)=\hat{\zeta}(1-s)$. Then all coefficients $\langle
\lambda_p, \alpha^\vee \rangle$ in the last line are nonnegative
integers:
\begin{equation}\label{0501}
g_{p,w}(s)  =
\prod_{\alpha \in \Phi_{w} \setminus \Delta_p}
\hat{\zeta}(\langle \lambda_p, \alpha^\vee \rangle s + {\rm ht}\,\alpha^\vee)
\prod_{\alpha \in \Phi^{+} \setminus \Phi_{w}}
\hat{\zeta}(\langle \lambda_p, \alpha^\vee \rangle s + {\rm ht}\, \alpha^\vee +1)
\end{equation}
for each $w \in {\mathfrak W}_p$.
Note that $\Phi^+ \setminus \Phi_w = \Phi^+ \cap w^{-1}\Phi^+$.
Define
\begin{equation}\label{e:delta}
\delta_{\alpha,w}=
\begin{cases}
1 & \text{if}~\alpha \in w^{-1}\Phi^{+}, \\
0 & \text{if}~\alpha \in w^{-1}\Phi^{-}.
\end{cases}
\end{equation}
Then formula \eqref{0501} is written as
\[
g_{p,w}(s) = \hat{\zeta}(2)^{|\Delta_p \cap w^{-1}\Phi^+|}
\prod_{\alpha \in \Phi^{+}\setminus \Delta_p}
\hat{\zeta}(\langle \lambda_p, \alpha^\vee \rangle s + {\rm ht}\, \alpha^\vee +\delta_{\alpha,w}).
\]
We define entire functions $\tilde{X}_{p,w}(s)$ by replacing $\hat{\zeta}(s)$ of $g_{p,w}$ by $\xi(s)$:
\begin{equation}\label{0502}
\tilde{X}_{p,w}(s) =
\xi(2)^{|\Delta_p \cap w^{-1}\Phi^+|}
\prod_{\alpha \in \Phi^{+}\setminus \Delta_p}
\xi(\langle \lambda_p, \alpha^\vee \rangle s + {\rm ht}\, \alpha^\vee +\delta_{\alpha,w}).
\end{equation}
Obviously, we have
\begin{multline}
\tilde{X}_{p,w}(s) =2^{|\Delta_p \cap w^{-1}\Phi^+|} g_{p,w}(s)\\
\times \prod_{\alpha \in \Phi^{+}\setminus \Delta_p} \Bigl(\langle
\lambda_p, \alpha^\vee \rangle s + {\rm ht}\, \alpha^\vee
+\delta_{\alpha,w} \Bigr) \Bigl(\langle \lambda_p, \alpha^\vee
\rangle s + {\rm ht}\, \alpha^\vee +\delta_{\alpha,w} -1 \Bigr).
\end{multline}
Define the polynomial $Q_p(s)$ by
\[
\aligned
Q_p(s) & =
\prod_{w \in {\mathfrak W}_p}
\left[  2^{|\Delta_p \cap w^{-1}\Phi^+|}
\prod_{\alpha \in (w^{-1} \Delta)\setminus \Delta_p}
\Bigl( \langle \lambda_p, \alpha^\vee \rangle s + \h \alpha^\vee -1 \Bigr)
\right. \\
& \qquad \qquad \times \left.
\prod_{\alpha \in \Phi^{+} \setminus \Delta_p}
\Bigl(\langle \lambda_p, \alpha^\vee \rangle s + {\rm ht}\, \alpha^\vee +\delta_{\alpha,w} \Bigr)
\Bigl(\langle \lambda_p, \alpha^\vee \rangle s + {\rm ht}\, \alpha^\vee +\delta_{\alpha,w} -1 \Bigr)
\right]. 
\endaligned
\]
\begin{definition}\label{def_0502}
Under the above notation, we define
\[
X_p(s):= Q_p(s) Z_{p}(s).
\]
\end{definition}
By definitions $X_p(s)$ is an entire function having the form
\[
X_p(s) = \sum_{w \in {\mathfrak W}_p}
\tilde{Q}_{p,w}(s) \, \tilde{X}_{p,w}(s),
\]
where $\tilde{Q}_{p,w}(s)$ are polynomials given by
\begin{equation*}
\aligned
\tilde{Q}_{p,w}(s)
& = \prod_{w\not= v \in {\mathfrak W}_p}
\left[  2^{|\Delta_p \cap v^{-1}\Phi^+|}
\prod_{\alpha \in (v^{-1} \Delta)\setminus \Delta_p}
\Bigl( \langle \lambda_p, \alpha^\vee \rangle s + \h \alpha^\vee -1 \Bigr)
\right. \\
& \qquad \qquad \times \left.
\prod_{\alpha \in \Phi^{+} \setminus \Delta_p}
\Bigl(\langle \lambda_p, \alpha^\vee \rangle s + {\rm ht}\, \alpha^\vee +\delta_{\alpha,v} \Bigr)
\Bigl(\langle \lambda_p, \alpha^\vee \rangle s + {\rm ht}\, \alpha^\vee +\delta_{\alpha,v} -1 \Bigr)
\right].
\endaligned
\end{equation*}
\begin{remark}
Note that the polynomial $Q_p(s)$ is not minimal in the sense that the polynomial of the lowest degree such that $Q(s)Z_p(s)$ has no poles.
In fact $\tilde{Q}_{p,w}(s)$ ($w\in {\mathfrak W}_p$) has a lot of common divisors
as well as $\tilde{X}_{p,w}(s)$.
\end{remark}

Now we modify $\tilde{X}_{p,w}(s)$ and $\tilde{Q}_{p,w}(s)$ a little.
Define
\begin{equation}\label{0502_1}
C_{p,w}=  \xi(2)^{|\Delta_p \cap w^{-1}\Phi^+|}
\prod_{\alpha \in \Phi_p^{+}\setminus \Delta_p}
\xi({\rm ht}\, \alpha^\vee +\delta_{\alpha,w}),
\end{equation}
\begin{equation}\label{0502_2}
X_{p,w}(s) =
\prod_{\alpha \in \Phi^{+}\setminus \Phi_p^+}
\xi(\langle \lambda_p, \alpha^\vee \rangle s + {\rm ht}\, \alpha^\vee +\delta_{\alpha,w})
\end{equation}
and
\begin{equation}\label{0502_3}
Q_{p,w}(s)=C_{p,w}\tilde{Q}_{p,w}(s).
\end{equation}
Then we have
\begin{equation*}
\tilde{X}_{p,w}(s)=C_{p,w}X_{p,w}(s)
\end{equation*}
and
\[
X_p(s) = \sum_{w \in {\frak W}_p} Q_{p,w}(s)X_{p,w}(s).
\]%
\begin{lemma}\label{lem_0503}
Let $c_p$ be the number defined in \eqref{0202}.
Let $w_0$ (resp. $w_p$) be the longest element of $W$ (resp. $W_p$).
Then the functional equations
\begin{align*}
\tilde{X}_{p,w}(-c_p-s) &=  \tilde{X}_{p,w_0ww_p}(s), \\
\tilde{Q}_{p,w}(-c_p-s) &= \epsilon_p \, \tilde{Q}_{p,w_0ww_p}(s), \\
Q_p(-c_p-s) &= \epsilon_p \, Q_p(s),
\end{align*}
hold for every $w \in {\frak W}_p$
for a suitable choice of sign $\epsilon_p \in \{\pm 1\}$
depending only on $p$.
In particular we have the functional equation
\[
X_p(-c_p-s) = \epsilon_p X_p(s).
\]
\end{lemma}
\begin{proof}
This is a consequence of Lemma 5.3 of ~\cite{Ko}.
\end{proof}
\subsection{Decomposition of $X_p(s)$}
The aim of this part is a construction of an entire function $E_p(s)$
satisfying
\[
X_p(s)=E_p(s) + \epsilon_p E_p(-c_p-s),
\]
where $\epsilon_p$ is the sign of Lemma \ref{lem_0503}.
\smallskip

\begin{definition}\label{def_0504}
For $w \in {\mathfrak W}_p$, we define the number $l_p(w)$ by
\begin{equation*}
l_p(w)
= \sum_{\alpha \in \Phi^{+} \setminus \Phi_{p}^+} (1 - \delta_{\alpha,w})
= |(\Phi^+ \setminus \Phi_{p}^+) \cap w^{-1}\Phi^{-}|.
\end{equation*}
\end{definition}
\begin{remark}
Compare $l_p(w)$ with the length
$l(w)=|\Phi^{+} \cap w^{-1}\Phi^{-}|$ of $w \in W$.
\end{remark}

\begin{lemma}\label{lem_0505}
Let $w_p$ be the longest element of $W_p$. Then we have
\[
l_p(w)+l_p(w_0 w w_p) = |\Phi^+ \setminus \Phi_p^+|
\]
for every $w \in {\frak W}_p$.
\end{lemma}
\begin{proof} We have
\[
|(\Phi^+ \setminus \Phi_p^+)\cap(w_0ww_p)^{-1} \Phi^-|
= |w_p(\Phi^+ \setminus \Phi_p^+) \cap w^{-1} \Phi^+|
= |(\Phi^+ \setminus \Phi_p^+) \cap w^{-1} \Phi^+|.
\]
Hence we obtain
\[
\aligned
l_p(w)+l_p(w_0 w w_p)
& = |(\Phi^+ \setminus \Phi_p^+)\cap w^{-1} \Phi^-| + |(\Phi^+ \setminus \Phi_p^+)\cap(w_0ww_p)^{-1} \Phi^-| \\
& = |(\Phi^+ \setminus \Phi_p^+)\cap w^{-1} \Phi^-| + |(\Phi^+ \setminus \Phi_p^+)\cap w^{-1} \Phi^+|
  = |\Phi^+ \setminus \Phi_p^+|
\endaligned
\]
for every $w \in {\frak W}_p$.
\end{proof}
\begin{definition} \label{def_0506}
We divide the subset ${\frak W}_p$ of $W$ into three parts:
\[
{\mathfrak W}_p = {\frak W}_{p}^{+} ~\amalg~  {\frak W}_{p}^{-} ~\amalg~ {\frak W}_{p}^{0},
\]
where $\amalg$ means the disjoint union, and
\begin{equation*}
\aligned
{\frak W}_{p}^{+}
&:= \{ w \in {\mathfrak W}_p ~|~ l_p(w) < |\Phi^+ \setminus \Phi_p^+|/2 \,\}, \\
{\frak W}_{p}^{-}
&:= \{ w \in {\mathfrak W}_p ~|~ l_p(w) > |\Phi^+ \setminus \Phi_p^+|/2 \,\}, \\
{\frak W}_{p}^{0}
&:= \{ w \in {\mathfrak W}_p ~|~ l_p(w) = |\Phi^+ \setminus \Phi_p^+|/2 \,\}. \\
\endaligned
\end{equation*}
\end{definition} 
Note that ${\frak W}_{p}^{0}=\emptyset$ is possible, and
\[
w_0 \,{\frak W}_{p}^{-}\, w_p = {\frak W}_{p}^{+}, \qquad w_0 \,{\frak W}_{p}^{0}\, w_p = {\frak W}_{p}^{0}.
\]
by Lemma \ref{lem_0505}.
\begin{definition}\label{def_0507} Define
\begin{equation}
E_p(s) = \sum_{w \in {\frak W}_{p}^{+}} Q_{p,w}(s)X_{p,w}(s)
+
\frac{1}{2} \sum_{w \in {\frak W}_{p}^{0}}
Q_{p,w}(s)X_{p,w}(s).
\end{equation}
Here we understand that the second sum equals zero
if the subset ${\frak W}_{p}^{0}$ is empty.
\end{definition}

\begin{proposition}\label{prop_0508}
We have
\begin{equation}
X_p(s) = E_p(s) + \epsilon_p E_p(-c_p-s),
\end{equation}
where $\epsilon_p$ is the sign of Lemma \ref{lem_0503}.
\end{proposition}
\begin{proof} This is obvious by the above definitions and lemmas. \end{proof}

The decomposition of Proposition \ref{prop_0508} is useful for our
proof of Theorem \ref{thm101} (the weak Riemann hypothesis). However
we do not know whether this choice of $E_p(s)$ is best possible for
the Riemann hypothesis for $\hat{\zeta}_P(s)$.

%
%
\subsection{Reduction of $X_p(s)$}
%
%
By definition of $Q_{p,w}(s)$ and $X_{p,w}(s)$,
component terms $Q_{p,w}(s)X_{p,w}(s)$ of $X_p(s)$
have a lot of ``common factors''.
Now we define ``the greatest common divisor'' of these terms.

\begin{definition}\label{def_0509}
Define
\[
\aligned
D_p(s) &= \prod_{k=1}^\infty \prod_{h=2}^\infty \xi(ks+h)^{N_p(k,h-1)-M_p(k,h)}, \\
R_p(s) &= {\rm g.c.d}\,\{Q_{p,w}(s) ~|~ w \in {\mathfrak W}_p\},
\endaligned
\]
where ``${\rm g.c.d}$'' means the monic polynomial of the maximal degree
which divides all polynomials $Q_{p,w}(s)$ in the polynomial ring ${\C}[s]$.
\end{definition}
\begin{remark} The above definition of $D_p(s)$
is different from the one of \cite[(5.36)]{Ko},
since we use the notation
$\xi(s)=s(s-1)\pi^{-s/2}\Gamma(s/2)\zeta(s)$
in this article.
\end{remark}
\begin{definition}\label{def_0510}
Define
\[
\aligned
\xi_p(s) & = \frac{X_p(s)}{R_p(s)D_p(s)}, \qquad
\varepsilon_p(s) & = \frac{E_p(s)}{R_p(s)D_p(s)}.
\endaligned
\]
\end{definition}
We have
\[
\xi_p(s) = \varepsilon_p(s) \pm \varepsilon_p(-c_p-s).
\]
The entire function $\xi_p(s)$ equals the zeta function
$\hat{\zeta}_p(s)$ times a polynomial. Hence Theorem \ref{thm101} is
equivalent to the statement that all but finitely many zeros of
$\xi_p(s)$ lie on the line $\Re(s)=-c_p/2$.

%
%
\section{The second step of the proof Theorem \ref{thm101}} \label{section_06}
%
%
Roughly the second step is a consequence of the result that
$X_{p, \rm id}(s)$ is dominant in $\Re(s)\geq -c_p/2$ and
$X_{p, \rm id}(s)/D_p(s)$ has no zeros in $\Re(s) \geq -c_p/2$,
since
\[
E_p(s)
= \Bigl[ \sum_{w \in {\mathfrak W}_p^{\ddagger}} Q_w(s) \Bigr] \cdot X_{p,\rm id}(s) \cdot
\Bigl[~
1 +
\!\!\!\!\!\! \sum_{{w \in {\frak W}_p \setminus {\mathfrak W}_p^{\ddagger}}\atop{2 l_p(w) \leq |\Phi^+ \setminus \Phi_p^+|}}
R_w(s) \frac{X_{p,w}(s)}{X_{p, \rm id}(s)}
~\Bigr],
\]
where ${\mathfrak W}_p^{\ddagger}$ is a subset of ${\frak W}_p$
defined below and $R_w(s)$ are some rational functions. The analytic
behavior of $E_p(s)$ in the right half-plane  $\Re(s) > -c_p/2$ is
less difficult than that of the left.

\subsection{Behavior of $E_p(s)$ in a right half-plane}

In this part we construct the dominant term of $\varepsilon_{p}(s)$
in a right half-plane via $E_p(s)$ of Definition \ref{def_0507}.

\begin{lemma} \label{lem_0601}
Let $\{v_1={\rm id}, v_2, \cdots,v_l\}$ be the left minimal coset representative of $W_p$:
\[
W = \bigsqcup_{j=1}^{l} v_jW_p.
\]
Then $l_p(w)$ depends only on the coset $v_jW_p$.
\end{lemma}
\begin{proof} 
For $w \in v_j W_p $, we have
$l_p(w)= |(\Phi^+ \setminus \Phi_p^+)\cap w^{-1} \Phi^-|
= |w (\Phi^+ \setminus \Phi_p^+) \cap \Phi^-|
= |v_j (\Phi^+ \setminus \Phi_p^+) \cap \Phi^-|
= |(\Phi^+ \setminus \Phi_p^+) \cap v_j^{-1} \Phi^-|
=l_p(v_j)$. 
\end{proof}
\begin{lemma} \label{lem_0602}
Let $w_0$ be the longest element of $W$. We have
\[
l_p({\rm id})=0, \qquad
l_p(w_0)=|\Phi^{+} \setminus \Phi_{p}^{+}|.
\]
In other words the minimum value and the maximum value of $l_p(w)$
are attained by the identity element $\rm id$ and  the longest element $w_0$ respectively. 
\end{lemma}
\begin{proof} We have
$l_p({\rm id})=|(\Phi^+ \setminus \Phi_p^+)\cap \Phi^-| =0$
and $l_p(w_0)=|(\Phi^+ \setminus \Phi_p^+) \cap \Phi^+|
= |\Phi^+ \setminus \Phi_p^+|$ by definition of $l_p(w)$.
\end{proof}
Note that from Lemmas \ref{lem_0601} and \ref{lem_0602}, $w \in W_p$ implies $l_p(w)=0$. 
\begin{lemma}\label{lem_0603}
Define
\begin{equation*}
{\frak W}_{p}^{\ddagger}
=\{
w \in {\frak W}_{p} ~|~
l_p(w)=0
\}. 
\end{equation*}
Then ${\rm id} \in {\frak W}_{p}^{\ddagger}$ and
\[
{\frak W}_{p}^{\ddagger}=\{
w \in {\frak W}_{p} ~|~
(\Phi^{+} \setminus \Phi_{p}^+) \cap w^{-1}\Phi^{+} = (\Phi^{+} \setminus \Phi_{p}^+) \}.
\]
In particular $X_{p,w}(s)=X_{p,{\rm id}}(s)$ for every $w \in {\frak W}_{p}^{\ddagger}$.
\end{lemma}
\begin{proof} This lemma immediately follows from Definition \ref{def_0504} and \eqref{e:delta}. \end{proof}
By Lemmas \ref{lem_0601}, \ref{lem_0602} and \ref{lem_0603}
\[
{\frak W}_{p}^{\ddagger} = ( W_p \amalg v_2 W_p \amalg \cdots \amalg v_m W_p )  \cap {\frak W}_{p}^{\ddagger},
\]
for some $1 \leq m < l$ ($=|W/W_p|$). 
(In fact, $m=1$ because from the proof of Lemma \ref{lem1002}, we have ${\frak W}_{p}^{\ddagger} \subset W_p$.)
\begin{definition}[the dominant term in a right half-plane] \label{def_0604}
We define
\[
X_{p}^{\ddagger}(s) := X_{p,{\rm id}}(s), \qquad
Q_{p}^{\ddagger}(s) :=
\sum_{w \in {\frak W}_{p}^{\ddagger}} Q_{p,w}(s).
\]
\end{definition}
By definition
\[
\aligned
X_{p}^{\ddagger}(s) = X_{p,{\rm id}}(s)
= \prod_{\alpha \in \Phi^{+} \setminus \Phi_p^+} \xi(\langle \lambda_p, \alpha^\vee \rangle s + \h\alpha^\vee + 1)
=
\prod_{k=1}^{\infty} \prod_{h=2}^{\infty} \xi(ks+h)^{N_p(k,h-1)}.
\endaligned
\]
The following two propositions assert that
\[
Q_{p}^{\ddagger}(s) \, X_{p}^{\ddagger}(s)
=
\Bigl( \sum_{w \in {\frak W}_{p}^{\ddagger}} Q_{p,w}(s) \Bigr)
\prod_{\alpha \in \Phi^{+} \setminus \Phi_p^+} \xi(\langle \lambda_p, \alpha^\vee \rangle s + \h\alpha^\vee + 1)
\]
is the dominant term of $E_p(s)$ in a right half-plane.
\medskip

Now we introduce an important hypothesis.
\medskip

\noindent
{\bf Volume Hypothesis.} We have
\begin{equation} \label{VF}
\underset{\lambda=\rho_p}{\rm Res}  \,
\omega_{\Delta_p}^{\Phi_p}(\lambda) >0
\quad
\text{for every $1\leq p \leq r$.}
\end{equation}

As in the name of the hypothesis,
residues $\underset{\lambda=\rho_p}{\rm Res}  \,\omega_{\Delta_p}^{\Phi_p}(\lambda)$
are expected to be volumes of certain subsets of a fundamental domain of $G(\Q) \backslash G(\A)^1$ for $G=G(\Phi)$.
The volume hypothesis is a consequence of the second part of the parabolic reduction conjecture
 \cite[Section 2.3, Conjecture 11]{We7} (or \cite[Section3, Conjecture 2]{We6})
that was proved in the case $G=SL(n)$ by using the theory of Eisenstein series (\cite[section 4.7]{MR2310297}).
This is a quite important piece of Theorem \ref{thm101}.
In fact, the following result plays an essential role in the proof of Theorem \ref{thm101},
but is not proved without \eqref{VF} at present
(see the proof of Lemma \ref{lem1003} in section \ref{section_10}).
\begin{proposition}\label{prop_0605}
Assume the volume hypothesis \eqref{VF}. Then, we have
\[
\deg Q_{p}^{\ddagger}(s) \geq \deg Q_{p,w}(s) + 1
\]
for every $w \in ({\frak W}_{p}^{+} \cup {\frak W}_{p}^{0}) \setminus {\frak W}_{p}^{\ddagger}$.
\end{proposition}
\begin{proposition}\label{prop_0606}
We have
\[
\left\vert \frac{X_{p,w}(s)}{X_{p}^{\ddagger}(s)} \right\vert < 1
\quad \text{for} \quad
\Re(s) > -\frac{c_p}{2}
\]
and
\[
\left\vert \frac{X_{p,w}(s)}{X_{p}^{\ddagger}(s)} \right\vert \leq 1
\quad \text{for} \quad
\Re(s) = -\frac{c_p}{2}
\]
for every $w \in {\frak W}_{p} \setminus {\frak W}_{p}^{\ddagger}$.
\end{proposition}
The dominant term $X_p^{\ddagger}(s)$ has the following property.
\begin{proposition}\label{prop_0607} Let $D_p(s)$ be the function defined in Definition \ref{def_0509}.
Then
\[
\frac{X_{p}^{\ddagger}(s)}{D_p(s)} \not=0
\quad \text{for} \quad
\Re(s) \geq -\frac{c_p}{2}.
\]
Furthermore, there exists a positive function $\delta(t)$ defined on the real line
satisfying $\delta(t)\log |t|\to\infty$ $(|t|\to\infty)$
such that $X_{p}^{\ddagger}(s)/D_p(s)$
has no zero in $\Re(s) \geq -c_p/2 - \delta(t)$.
\end{proposition}
\begin{remark}
By a known result ~\cite[p.135]{MR882550} obtained by Vinogradov's
method, we could take $\delta(t)$ such that
\[
\delta(t)\ll\frac{1}{(\log t)^{2/3}(\log\log t)^{1/3}}
\]
as $t \to \infty$.
\end{remark}

We prove Proposition \ref{prop_0605} in section \ref{section_10},
and Propositions \ref{prop_0606} and \ref{prop_0607} in section
\ref{section_09}. These three propositions derive the following
result.
\begin{proposition}\label{prop_0608}
The number of zeros of $\varepsilon_p(s)$ lying in right half-plane
$\Re(s) \geq -c_p/2$ is finitely many at most.
Furthermore, there exists a positive function $\delta(t)$ on the real line
satisfying $\delta(t)\log |t|\to\infty$ $(|t|\to\infty)$
such that the number of zeros of $\varepsilon_p(s)$
in $\Re(s) \geq -c_p/2 - \delta(t)$ is finitely many at most.
\end{proposition}
\begin{proof}
We have
\begin{equation}\label{0602_1}
E_p(s)
= Q_{p}^{\ddagger}(s) X_{p}^{\ddagger}(s)[ 1 + r_p(s)]
\end{equation}
with
\[r_p(s)  = \sum_{w \in {\frak W}_{p}^{+} \setminus {\frak W}_{p}^{\ddagger}}
\frac{Q_{p,w}(s)}{Q_{p}^{\ddagger}(s)} \cdot \frac{X_{p,w}(s)}{X_{p}^{\ddagger}(s)}
+ \quad \frac{1}{2} \sum_{w \in {\frak W}_{p}^{0}}
\frac{Q_{p,w}(s)}{Q_{p}^{\ddagger}(s)} \cdot \frac{X_{p,w}(s)}{X_{p}^{\ddagger}(s)}
\]
Define
\[
{\frak D}_p=\Bigl\{
 s \in {\C} ~\Bigl|~
\sum_{w \in {\frak W}_{p}^{+} \setminus {\frak W}_{p}^{\ddagger}}
\left\vert \frac{Q_{p,w}(s)}{Q_{p}^{\ddagger}(s)} \right\vert +
\frac{1}{2} \sum_{w \in {\frak W}_{p}^{0}} \left\vert
\frac{Q_{p,w}(s)}{Q_{p}^{\ddagger}(s)} \right\vert \geq 1 \Bigr\}.
\]
Then ${\frak D}_p$ is a bounded region in $\C$ by Proposition \ref{prop_0605}.
Therefore we have
\begin{equation}\label{0602}
|r_p(s)|<1 \quad \text{for} \quad
\Re(s) \geq -\frac{c_p}{2} ~\text{with}~ s \not\in {\frak D}_p,
\end{equation}
by Proposition \ref{prop_0606}.
Hence Proposition \ref{prop_0607} and  \eqref{0602} implies Proposition \ref{prop_0608}.
\end{proof}

The final proposition of the second step is about the zero-free
region of $E_p(s)$ on a left half-plane.
\begin{proposition}\label{prop_0609}
There exists a positive real number $\kappa$ such that $E_p(s)$ has no zeros
in the region $\Re(s) \leq -\kappa \log(|\Im(s)|+10)$.
\end{proposition}
This will be proved in section \ref{section_11} by using the
Stirling formula. Combining Propositions \ref{prop_0608} and
\ref{prop_0609}, we find that all but finitely many zeros of
$\varepsilon_p(s)$ lie in the region
\[
\left\{s=\sigma+it\in\C \,\left| -\kappa \log(|t|+10) < \sigma < -\frac{c_p}{2} \right.\right\}.
\]

%
%
\section{The third step of the proof Theorem \ref{thm101}} \label{section_07}
%
%

The final step of the proof of (the front half of) Theorem
\ref{thm101} consists of three parts. The first one is about the
number of zeros of $\varepsilon_p(s)$ in a given region. The second
one is the Hadamard factorization formula of $\varepsilon_p(s)$. The
third one is an application of a result of de
Bruijn~\cite[p.215]{MR0037351} (see also Lemma 3.1 of~\cite{Ki09})
which was established by the first author in ~\cite{Ki09}.

\begin{proposition}\label{prop_0701} Let $T>1$, and $\sigma>c_p/2$.
Denote by $N(T;\sigma)$
the number of zeros of ${\varepsilon}_p(s)$
in the region
\[
-\sigma < \Re(s) <-c_p/2-\delta(t), \quad 0<\Im(s)<T.
\]
Then there exist a positive number $\sigma_{\rm L}>0$ such that
\[
N(T;\sigma_{\rm L}) = C_1\, T \log T+ C_2\, T + O(\log T)
\]
for some positive real number $C_1>0$ and real number $C_2$, and
\[
N(T;+\infty)  = C_1\, T \log T+ C_3\, T + O(\log^2 T)
\]
for some real number $C_3$. 
\end{proposition}

Thus we have $N(T;+\infty)-N(T;\sigma_{\rm L})=O(T)$.
Also, we can justify this by applying and modifying the method in
~\cite[p.230]{MR882550}.
This will be proved in section \ref{section_12}.

\begin{proposition}\label{prop_0702}
Define
\[
W_p(z) = \varepsilon_p(-c_p/2+iz).
\]
Then it has the product formula
\begin{equation}\label{0701}
W_p(z)= \omega \,e^{\alpha z}V(z)W_1(z)W_2(z),
\end{equation}
where $\omega$ is a nonzero real number, $\alpha$ is a real number,
$V(z)$ is a polynomial having no zeros in $\Im(z) >0$
except for purely imaginary zeros,
\[
\aligned
W_1(z)&=\prod_{n=1}^{\infty}
\left[\left(1-\frac{z}{\rho_n}\right)\left(1+\frac{z}{\bar{\rho}_n}\right) \right], \\
W_2(z)&=\prod_{n=1}^{\infty}
\left[\left(1-\frac{z}{\eta_n}\right)\left(1+\frac{z}{\bar{\eta}_n}\right) \right]
\endaligned
\]
with $\Re(\rho_n) > 0$, $\Re(\eta_n) > 0$ and
$0<\delta(t) < \Im(\rho_n) <\sigma_{\rm L}+1< \Im(\eta_n)<\kappa \log(\Re(\eta_n)+10)$
for every $n\geq 1$.
Here $\delta(t)$ is the function of Proposition \ref{prop_0608},
$\kappa$ is the positive number of Proposition \ref{prop_0609},
and $\sigma_{\rm L}$ is the positive number of Proposition \ref{prop_0701}.
The products $W_1$ and $W_2$ converge uniformly on every compact subset in $\C$.
\end{proposition}

This will be proved in section \ref{section_13} by using Proposition \ref{prop_0701}.

\begin{proposition}[Proposition 3.1 of \cite{Ki09}]\label{prop_0703}
Let $W(z)$ be a function in $\C$.
Suppose that $W(z)$ has the product factorization
\[
W(z) = h(z) \, e^{\alpha z} \prod_{n=1}^{\infty}
\left[\left(1-\frac{z}{\rho_n}\right)\left(1+\frac{z}{\bar{\rho}_n}\right) \right],
\]
where $h(z)$ is a nonzero polynomial having $N$ many zeros counted with multiplicity
in the lower half-plane, $\alpha \in {\R}$, $\Im(\rho_n) > 0$ $(n = 1,2,\cdots)$,
and the product converges uniformly in any compact subset of $\C$.
Then, $W(z) + \overline{W(\bar{z})}$ and $W(z) - \overline{W(\bar{z})}$ has at most $N$
pair of conjugate complex zeros counted with multiplicity.
\end{proposition}

Now we achieve the following goal which is an immediate consequence
of Propositions \ref{prop_0702} and \ref{prop_0703}. Note that the
realness of the exponent $\alpha$ of Proposition \ref{prop_0702} is
crucial.

\begin{theorem}[Weak Riemann Hypothesis for $\xi_{p}$]\label{thm_0703}
Assume the volume hypothesis \eqref{VF}.
Then, there exists a bounded region ${\frak B}_p$ such that
all zeros of $\xi_{p}(s)$ outside ${\frak B}_p$
lie on the line $\Re(s)=-c_p/2$.
\end{theorem}
By studying of the behavior of the argument of $\varepsilon_p(-c_p/2 + it)$ ($t>0$),
we obtain the following additional result.
\begin{theorem}[Simple zeros of $\xi_{p}$]\label{thm_0704}
Assume the volume hypothesis \eqref{VF}.
Then, there exists a bounded region ${\frak B}_p^\prime (\supset {\frak B}_p)$
such that all zeros of $\xi_{p}(s)$ outside ${\frak B}_p^\prime$
lie on the line $\Re(s)=-c_p/2$ and simple.
\end{theorem}

This will be proved in section \ref{section_14}. Theorems
\ref{thm_0703} and \ref{thm_0704} imply the main result Theorem
\ref{thm101} because of Chevalley's fundamental theorem.

%
%
\section{Preliminaries for proof of Propositions \ref{prop_0606} and \ref{prop_0607}} \label{section_08}
%
%
In this section, we prepare several lemmas for the proof of
Propositions \ref{prop_0606} and \ref{prop_0607}. Indeed, Lemmas
\ref{lem_0801}, \ref{lem_0804} and \ref{lem_0815} will play an important role for
it. The condition \eqref{0201} in Definition \ref{0201} is essential
for Lemma \ref{lem_0801}.
\medskip

For integers $k$ and $l$, we define
\[
\aligned
\Sigma_p(k)&=\{\alpha \in \Phi \,|\, \langle \lambda_p, \alpha^\vee \rangle=k  \}, \\
\Sigma_p(k,h)&=\{\alpha \in \Phi \,|\, \langle \lambda_p, \alpha^\vee \rangle=k, ~\h \alpha^\vee = h  \}.
\endaligned
\]
They are not empty for finitely many $(k,h)$.
The set of positive roots $\Phi^+$ is decomposed into the disjoint union
\[
\Phi^+ \setminus \Phi_p^+= \bigsqcup_{k=1}^{\infty} \Sigma_p(k), \quad
\Sigma_p(k)=\bigsqcup_{h=1}^{\infty} \Sigma_p(k,h).
\]
By definition, $N_p(k,h)$ is the cardinality of $\Sigma_p(k,h)$.

\begin{lemma}\label{lem_0801}
Let $w \in {\frak W}_p$.
Let $\alpha$ be a positive root.
Suppose that $\alpha \in (\Phi^+ \setminus \Phi_p^+) \cap w^{-1} \Phi^{-}$
and $\alpha+\alpha_j \in \Phi^+ \setminus \Phi_p^+$ for some $\alpha_j \in \Delta_p$.
Then $\alpha+\alpha_j \in (\Phi^+ \setminus \Phi_p^+) \cap w^{-1} \Phi^{-}$.
\end{lemma}

\noindent
\textbf{Note.} Besides $\alpha \in (\Phi^+ \setminus \Phi_p^+) \cap w^{-1} \Phi^{-}$
and $\alpha+\alpha_p \in \Phi^+ \setminus \Phi_p^+$
does not imply
$\alpha+\alpha_p \in (\Phi^+ \setminus \Phi_p^+) \cap w^{-1} \Phi^{-}$
in general.

\begin{proof}
It suffices to show that $w(\alpha+\alpha_j) \in \Phi^-$ under the assumption.
By the assumption and the definition of ${\frak W}_p$,
we have $w\alpha \in \Phi^{-}$ and $w\alpha_j \in \Delta \cup \Phi^{-}$.
If $w\alpha_j \in \Phi^{-}$, we have $w(\alpha+\alpha_j) \in \Phi^{-}$,
since $\alpha+\alpha_j$ is a root.
If $w\alpha_j \in \Delta$, we also have $w\alpha+w\alpha_j \in \Phi^{-}$.
In fact, it is impossible that $w\alpha_j \in \Delta$ and $w\alpha+w\alpha_j \in \Phi^{+}$,
since $w\alpha \in \Phi^{-}$.
In each case, we have $w(\alpha+\alpha_j) \in \Phi^-$.
\end{proof}

\begin{lemma}\label{lem_0802}
Let $\alpha$ be a positive root in $\Phi^+ \setminus \Phi_p^+$.
Suppose that $\langle \rho_p,  \alpha^\vee \rangle <0$. Then there
exists $\alpha_j \in \Delta_p$ such that $\alpha^\vee+\alpha_j^\vee
\in (\Phi^+ \setminus \Phi_p^+)^\vee$. In particular, there exists
$\alpha_j \in \Delta_p$ such that $\alpha^\vee+\alpha_j^\vee \in
((\Phi^+ \setminus \Phi_p^+) \cap w^{-1} \Phi^{-})^\vee$, if
$\langle \rho_p, \alpha^\vee \rangle <0$ for $\alpha \in (\Phi^+
\setminus \Phi_p^+) \cap w^{-1} \Phi^{-}$.
\end{lemma}
\begin{proof}
By $\langle \rho_p, \alpha^\vee \rangle<0$,
there exists $\alpha_j \in \Delta_p$ such that
$\langle \alpha_j,\alpha^\vee \rangle<0$,
since
\[
2\rho_p
= \sum_{\beta \in \Phi_p^+} \beta
= \sum_{{j=1}\atop{j\not=p}}^{r} n_j \alpha_j \quad (n_j \in {\Z}_{>0}).
\]
Thus $\langle \alpha_j^\vee,\alpha^\vee \rangle<0$ by
$\alpha_j^\vee=2\alpha_j/\langle \alpha_j,\alpha_j\rangle$. Hence
$\alpha^\vee + \alpha_j^\vee$ is a positive root in $\Phi^\vee$,
(\cite[\S9]{MR499562}). Moreover $\alpha^\vee + \alpha_j^\vee$
is a positive root in $(\Phi^+ \setminus \Phi_p^+)^\vee$, since
$\alpha^\vee \in (\Phi^+ \setminus \Phi_{p}^+)^\vee$ and
$\alpha_p^\vee \in \Phi_p^+$. The second statement follows from the
first statement and Lemma \ref{lem_0801}.
\end{proof}

\begin{lemma}\label{lem_0803} Let $k \geq 1$.
Then
\begin{enumerate}
\item $N_p(k,h)=N_p(k,kc_p-h)$ for every $h \geq 1$, and
\item $N_p(k,h) \leq N_p(k,h+1)$ if $2h+1 \leq kc_p$.
\end{enumerate}
\end{lemma}
\begin{proof}
See Proposition 1 of \cite{Man}.
See also Lemma 4.3 (1) of~\cite{Ko} for (1).
\end{proof}

\begin{lemma}\label{lem_0804} Let $k$ and $h$ be positive integers.
Write
\[
\Sigma_p(k,h)=\{\beta_1,\cdots,\beta_N\}
\quad (N=N_p(k,h))
\]
if $\Sigma_p(k,h)$ is not empty.
Suppose that $k \geq 1$ and $2h+1 \leq kc_p$.
Then there exists simple roots
$\alpha_{j_1},\cdots,\alpha_{j_N}$ (not necessary distinct)
such that
\begin{itemize}
\item $\alpha_{j_n} \in \Delta_p ~(1\leq n \leq N)$,
\item $\beta_n^\vee + \alpha_{j_n}^\vee \in (\Phi^+ \setminus \Phi_p^+)^\vee ~(1\leq n \leq N)$, and
\item $\beta_m^\vee + \alpha_{j_m}^\vee \not= \beta_n^\vee + \alpha_{j_n}^\vee~(m\not=n)$.
\end{itemize}
\end{lemma}
\begin{proof}
See the proof of Lemma \ref{lem_0815}.
\end{proof}

\begin{lemma}\label{lem_0805}
Let $k \geq 1$ and $w \in {\frak W}_p$.
Then $N_{p,w}(k,h) \leq N_{p,w}(k,h+1)$ if $2h+1 \leq kc_p$.
In particular, $M_p(k,h)=0$ if $2h-1 \leq kc_p$
\end{lemma}
\begin{proof}
This is a consequence of Lemmas \ref{lem_0801} and \ref{lem_0804}.
\end{proof}
\begin{lemma}\label{lem_0806}
Let $k \geq 1$. Then we have
\[
M_p(k,h)=
\begin{cases}
~0 & \text{if}~2h-1 \leq kc_p, \\
~N_p(k,h-1)-N_p(k,h) & \text{if}~2h-1> kc_p.
\end{cases}
\]
\end{lemma}
\begin{proof}
We put
\[
D_p^{(1)}(s) = \prod_{k=1}^\infty \prod_{h=2}^\infty \xi(ks+h)^{N_p(k,h-1)-M_p(k,h)}
\]
(cf. Definition \ref{def_0509}). Then we have
\[
\aligned
D_p^{(1)}(s) = \prod_{k=1}^\infty
& \prod_{h \leq (kc_p+1)/2} \xi(ks+h)^{N_p(k,h-1)}\\
& \times \prod_{h> (kc_p+1)/2} \xi(ks+h)^{N_p(k,h-1)-M_p(k,h)}
\endaligned
\]
by Lemma \ref{lem_0805}. On the other hand, we have
\[
\aligned
D_p^{(1)}(-c_p-s)
& = \prod_{k=1}^\infty \prod_{h<(kc_p+1)/2} \xi(ks+h)^{N_p(k,h)-M_p(k,kc_p-h+1)} \\
& \quad \times \prod_{ h \geq (kc_p+1)/2} \xi(ks+h)^{N_p(k,h)}
\endaligned
\]
by using Lemma 4.3 (1) of ~\cite{Ko}.
Because of $D_p^{(1)}(-c_p-s)=D_p^{(1)}(s)$ by Lemma 5.5 of ~\cite{Ko},
we have $M_p(k,h) = N_p(k,h-1)- N_p(k,h)$ for $h> (kc_p+1)/2$.
\end{proof}

\begin{corollary}\label{cor_0807}
In the definition of $M_p(k,h)$, the longest element $w_0$ attains the maximum
\[
\underset{w \in {\frak W}_p}{\rm max} \left( N_{p,w}(k,h-1)-N_{p,w}(k,h) \right).
\]
\end{corollary}
\begin{proof}
This is a consequence of Lemmas \ref{lem_0803} and \ref{lem_0806},
since $N_p(k,h)=N_{p,w_0}(k,h)$.
\end{proof}

Let $\widetilde{\alpha}$ be the highest root of $\Phi^+$.
Define integers $k_1,\cdots,k_r$ by
\begin{equation}\label{0801}
\widetilde{\alpha} = \sum_{i=1}^r k_i \alpha_i \quad (k_i >0).
\end{equation}
For a positive integer $k$ with $\Sigma_p(k)\not=\emptyset$,
a lowest root $\alpha_p^{-}(k)$ for $\Sigma_p(k)$ is the root such that
$\beta - \alpha_p^{-}(k)$ is a (possibly empty) sum of simple roots
for every $\beta \in \Sigma_p(k)$.
A highest root $\alpha_p^{+}(k)$ for $\Sigma_p(k)$ is defined by a similarly way.
Lowest roots and highest roots always exist and are unique
if $\Sigma_p(k)$ is not empty (\cite[Proposition 1.4.2]{FM}).

\begin{lemma}[Lemma 1.4.5 of ~\cite{FM}]\label{lem_0808}
Let $k$ be a positive integer.
Then $\Sigma_p(k)$ is not empty if and only if $1 \leq k \leq k_{p}$,
where $k_p$ is the number defined in \eqref{0801}.
\end{lemma}

\begin{lemma}[Lemma 1.4.6 of ~\cite{FM}]\label{lem_0809}
Suppose that $1 \leq k \leq k_p$.
Let $w_p$ be the longest element of $W_p$, i.e., $w_p \Delta_p = -\Delta_p$.
Then we have $\alpha_p^{-}(1)=\alpha_p$ and $w_p \alpha_p^{-}(k) = \alpha_p^{+}(k)$.
\end{lemma}
For a positive integer $k$ with $\Sigma_p(k)\not=\emptyset$, we write
\begin{equation*}
\alpha_{p}^{\pm}(k)^\vee = k\alpha_p^\vee + \gamma_{p}^{\pm}(k)^\vee
\quad (\gamma_{p}^\pm(k)^\vee \in Q^+(\Delta_p^\vee)),
\end{equation*}
where $Q^+(\Delta_p^\vee)=\sum_{j=1}^{r} {\Z}_{\geq 0}
\alpha_j^\vee$. Then the heights $\h \gamma_{p}^{+}(k)^\vee$ and $\h
\gamma_{p}^{-}(k)^\vee$ are nonnegative, since they must be written
as a combination of simple roots $\alpha_j^\vee \in \Delta_p^\vee$
with nonnegative integer coefficients. We define
\begin{equation}\label{0802}
\hbar_k^{\pm} = \h \gamma_{p}^{\pm}(k)^\vee = \langle \rho, \gamma_{p}^{\pm}(k)^\vee \rangle
\end{equation}
for $1 \leq k \leq k_p$.
For an element $\alpha \in \Sigma_p(h)$,
the height $\h(\alpha-k\alpha_p)^\vee$ is called the ($p$-)level of $\alpha$.
In this terminology,
$\hbar_k^{+}$ (resp. $\hbar_k^{-}$) is the highest (resp. lowest) level of $\Sigma_p(k)$.
Note that $\gamma_{p}^{-}(1)=0$, since $\alpha_{p}^{-}(1)=\alpha_p$.
\begin{lemma}\label{lem_0810} For every $1 \leq k_1 < k_2 \leq k_p$,
we have $\hbar_{k_1}^+ \leq \hbar_{k_2}^+$.
\end{lemma}
\begin{proof} We have
$\alpha_p^+(k_2)^\vee - \alpha_p^+(k_1)^\vee \in \Phi^\vee$
by Corollary 1.2 of ~\cite{KS} and
\[
\alpha_p^+(k_2)^\vee - \alpha_p^+(k_1)^\vee
= (k_2-k_1)\alpha_p^\vee + (\gamma_p^+(k_2)^\vee - \gamma_p^+(k_1)^\vee).
\]
Therefore $\gamma_p^+(k_2)-\gamma_p^+(k_1) \in Q^+(\Delta_p)$,
since $k_2 - k_1>0$ and $\gamma_p^+(k) \in Q^+(\Delta_p)$.
This implies the result by definition of $\hbar_k^+$.
\end{proof}

\begin{lemma}\label{lem_0811} Let $w_p$ be the longest element of $W_p$.
Then we have
\[
\gamma_{p}^{-}(k) - w_p \gamma_p^{+}(k) =  k\gamma_{p}^{+}(1).
\]
In particular, $w_p \gamma_p^{+}(1)= - \gamma_p^{+}(1)$ and
$\gamma_p^{+}(k) - w_p \gamma_{p}^{-}(k) =  k \, \gamma_{p}^{+}(1)$.
\end{lemma}
\begin{proof}
We have $w_p \alpha_p^{-}(k)^\vee= \alpha_p^{+}(k)^\vee= k \alpha_p^\vee + \gamma_p^{+}(k)^\vee$.
Therefore
\[
\aligned \alpha_p^{-}(k)^\vee &= w_p(w_p \alpha_p^{-}(k)^\vee)
=  k w_p \alpha_p^\vee + w_p \gamma_p^{+}(k)^\vee \\
& = k \alpha_p^{+}(1)^\vee + w_p \gamma_p^{+}(k)^\vee
  = k \alpha_p^\vee + k\gamma_{p}^{+}(1)^\vee + w_p \gamma_p^{+}(k)^\vee. 
\endaligned
\]
Hence $k \alpha_p^\vee + \gamma_{p}^{-}(k)^\vee = k\alpha_p^\vee + k\gamma_{p}^{+}(1)^\vee + w_p \gamma_p^{+}(k)^\vee$.
This implies the assertion.
\end{proof}

\begin{lemma}\label{lem_0812}
Let $c_p$ be the number of \eqref{0202}.
Then we have
\[
c_p = 2 + \hbar_1^{+}= 1+ \h\alpha_p^+(1)^\vee = 1 + \h (w_p\alpha_p^\vee),
\]
where $w_p$ is the longest element of $W_p$.
\end{lemma}
\begin{proof}
We have
\[
\aligned
c_p
& = 2 \langle \lambda_p - \rho_p, \alpha_p^\vee \rangle
= 2 - \langle 2 \rho_p, \alpha_p^\vee \rangle
= 2 - \langle \rho - w_p \rho, \alpha_p^\vee \rangle \\
& = 1 + \langle w_p \rho, \alpha_p^\vee \rangle
  = 1 + \langle \rho, w_p \alpha_p^\vee \rangle
  = 1 + \langle \rho, \alpha_p^\vee + \gamma_p^{+}(1)^\vee \rangle \\
& = 2 + \langle \rho, \gamma_p^{+}(1)^\vee \rangle = 2 + \hbar_1^+.
\endaligned
\]
Then the second (resp. the third) equality of Lemma \ref{lem_0812}
follows from definition of $\hbar_1^+$ (resp. Lemma \ref{lem_0809}).
\end{proof}

\begin{lemma}\label{lem_0813}
Let $k \geq 1$.
We have $2h=kc_p+m$
if $2 \langle \rho_p, \alpha^\vee \rangle=m \in {\Z}$ for $\alpha \in \Sigma_p(k,h)$.
In particular if $\alpha \in \Sigma_p(k,h)$ for $h<kc_p/2$ (resp. $h>kc_p/2$),
$\langle \rho_p, \alpha^\vee \rangle<0$ (resp. $>0$).
\end{lemma}
\begin{proof}
Because of $c_p \lambda_p=2\rho-2\rho_p$ from the proof of Lemma 4.1 of~\cite{Ko},
we have
\[
2\h\alpha^\vee - \langle \lambda_p, \alpha^\vee\rangle \,c_p =
2\h\alpha^\vee - \langle 2\rho-2\rho_p, \alpha^\vee\rangle = 2 \,
\langle \rho_p, \alpha^\vee\rangle.
\]
This implies the assertion by definition of $\Sigma_p(k,h)$.
\end{proof}
\begin{lemma}\label{lem_0814}
Let $1 \leq k \leq k_p$ and let $\hbar_k^\pm$ be numbers in
\eqref{0802}. We have
\[
\hbar_k^+ + \hbar_k^- = k \, \hbar_1^+.
\]
Therefore we have
\[
k+1 \leq \h\alpha^\vee \leq kc_p - k -1 \quad (k \geq 2)
\quad \text{and} \quad
1 \leq \h\alpha^\vee \leq c_p - 1 \quad (k=1)
\]
for $\alpha \in \Sigma_p(k)$ if $c_p \geq 3$, and $\hbar_k^+ \leq h_{k+1}^+ \leq (k+1)\hbar_1^+ -1$.
\end{lemma}
\begin{remark}
If $c_p=2$, $\alpha_p^+(1)=\alpha_p^-(1)=\alpha_p$ by Lemma
\ref{lem_0812}. 
 Hence $\Phi$ is of type $A_1$,
and the case $k \geq 2$
does not appear.
\end{remark}
\begin{proof}
We note that
\[
\h(w_p \gamma^\vee) = - \h \gamma^\vee
\]
for every $\gamma^\vee \in Q^+(\Delta_p^\vee)$. In fact, we have
\[
\h \gamma^\vee
= \langle \rho_p, \gamma^\vee \rangle
= \langle w_p \rho_p, w_p \gamma^\vee \rangle
= - \langle \rho_p, w_p \gamma^\vee \rangle
= - \h (w_p\gamma^\vee)
\]
for $\gamma^\vee \in Q^+(\Delta_p^\vee)$,
since $\langle \rho_p,\alpha_i^\vee \rangle=1$ for $i \not=p$
from the proof of Lemma 4.1 in \cite{Ko}.
Therefore we have
\[
\aligned
k+\hbar_k^-
& =\h(\alpha_p^{-}(k))^\vee
  = \h(w_p \alpha_p^{+}(k)^\vee)
  = \h(w_p(k \alpha_p^\vee + \gamma_{p}^{+}(k)^\vee) ) \\
& = k \, \h ( w_p \alpha_p^\vee)  + \h( w_p \gamma_{p}^{+}(k)^\vee )
  = k \, \h( w_p \alpha_p^\vee)  - \h(\gamma_{p}^{+}(k))^\vee \\
& = k (1 + \hbar_1^+) - \hbar_k^+.
\endaligned
\]
Here we used Lemmas \ref{lem_0811} and \ref{lem_0812}.
This implies Lemma \ref{lem_0814}.
\end{proof}
\begin{lemma}\label{lem_0815}
Let $1 \leq k \leq k_p$.
Then $\Sigma_p(k)$ is divided into the disjoint union
\[
\Sigma_p(k)= \bigsqcup_{m=1}^{M_k} L_m(k)
\]
satisfying the following conditions:
\begin{enumerate}
\item $\alpha_{p}^{\pm}(k) \in L_1(k)$ and $|L_1(k)|= \h \alpha_{p}^{+}(k)^\vee - \h \alpha_{p}^{-}(k)^\vee+1$,
\item $|L_1(k)| > |L_2(k)|, |L_3(k)|, \ldots, |L_{M_k}(k)| \geq 1$,
\item we have
\[
L_m(k)^\vee = \{ \beta _m(k)^\vee \} \cup \{\beta_m(k)^\vee +\sum_{j=1}^{J} \alpha_{m,j}(k)^\vee ~|~  1 \leq J \leq |L_m(k)|-1\}
\]
for some $\beta_m(k) \in \Sigma_{p}(k)$ $(1 \leq m \leq M_k)$ with $\beta_1(k)=\alpha_p^{-}(k)$,
and for some $\alpha_{m,j}(k) \in \Delta_{p}$ $(1 \leq m \leq M_k,\, 1 \leq j \leq |L_m(k)|-1)$,
\item $k+\hbar_k^{-}=\h \beta_1(k)^\vee < \h \beta_2(k)^\vee \leq \h \beta_3(k)^\vee \leq \cdots \leq \h \beta_M(k)^\vee \leq kc_p/2$,
\item $\h \widetilde{\beta}_m(k)^\vee = kc_p - \h \beta_m(k)^\vee$
$(1 \leq m \leq M_k)$, where
\[
\widetilde{\beta}_m(k)^\vee=\beta_m(k)^\vee
+\sum_{j=1}^{|L_m(k)|-1}\alpha_{m,j}(k)^\vee.
\]
\end{enumerate}
\end{lemma}
\begin{proof}
At first we note that the problem is reduced to the cases of irreducible root systems
with $\Sigma_p(1)$
by using results of section 1.4 of \cite{FM}.
Then we see that the assertions of Lemmas \ref{lem_0804} and \ref{lem_0815}
for $\Sigma_p(1)$ hold by constructing explicitly the required disjoint union directly.
See Appendix 1.
\end{proof}
%
%
%
\section{Proof of Propositions \ref{prop_0606} and \ref{prop_0607}} \label{section_09}
%
%
%
%
\subsection{Proof of Proposition \ref{prop_0606}}
%
%
\begin{lemma}\label{lem901}
Let $a$, $b$ be real numbers satisfying $a \leq b$. Then
\[
\left\vert \frac{\xi(s+a)}{\xi(s+b+1)} \right\vert < 1 \quad \text{for} \quad \Re(s) > -\frac{a+b}{2}
\]
and
\[
\left\vert \frac{\xi(s+a)}{\xi(s+b+1)} \right\vert = 1 \quad \text{for} \quad \Re(s) = -\frac{a+b}{2}.
\]
\end{lemma}
\begin{proof}
This is equivalent to $|\xi(s+(a-b-1)/2)| < |\xi(s-(a-b-1)/2)|$ for
$\Re(s)>1/2$ by shifting $s+(a+b+1)/2$ to $s$, and it holds by
applying Theorem 4 of \cite{MR2220265} to $F(s)=\xi(s)$ and
$c=(a-b-1)/2$. The second statement immediately follows from the
functional equation for $\xi(s)$.
\end{proof}
Now we prove Proposition \ref{prop_0606}. We recall
\begin{equation}\label{0901}
X_{p}^{\ddagger }(s)
= X_{p,{\rm id}}(s)
= \prod_{\alpha \in \Phi_{+} \setminus \Phi_p^+} \xi(\langle \lambda_p, \alpha^\vee \rangle s + \h\alpha^\vee + 1).
\end{equation}
by Lemma \ref{lem_0603}. From \eqref{0502} and \eqref{0901}, we have
\begin{equation}\label{0902}
\frac{X_{p,w}(s)}{X_{p}^{\ddagger }(s)}
= \prod_{\alpha \in (\Phi^+ \setminus \Phi_p^+) \cap w^{-1}\Phi^-}
\frac{\xi(\langle \lambda_p, \alpha^\vee \rangle s + \h\alpha^\vee)}
{\xi(\langle \lambda_p, \alpha^\vee \rangle s + \h\alpha^\vee + 1)}.
\end{equation}
Using the notation of Lemma \ref{lem_0815}, we define
\[
L_{m}(k)_w = L_m(k) \cap w^{-1}\Phi^-
\]
and put $\Lambda_w(k)=\{ m \,|\, 1 \leq m \leq M_k, \, L_m(k)_w
\not=\emptyset \}$. In addition, we define
\[
h_m(k) =  {\rm max}\{ \h \alpha^\vee ~|~ \alpha \in L_{m}(k) \},
\quad
h_{m,w}(k) = {\rm min}\{ \h \alpha^\vee ~|~ \alpha \in L_{m}(k)_w \}.
\]
for $m \in \Lambda_w(k)$. Note that $h_m(k)$ is attained by one of $
\alpha \in L_m(k)_w$ by Lemma \ref{lem_0802}. Then the right-hand
side of \eqref{0902} equals
\[
\prod_{k=1}^{k_p} \prod_{m \in \Lambda_w(k)}\prod_{h_m=h_{m,w}(k)}^{h_{m}(k)}
\frac{\xi(ks + h_m)}
{\xi(ks + h_m + 1)} \\
 = \prod_{k=1}^{k_p} \prod_{m \in \Lambda_w(k)}
\frac{\xi(ks + h_{m,w}(k))}{\xi(ks + h_{m}(k) + 1)}
\]
by Lemma \ref{lem_0802} and  Lemma \ref{lem_0815}~(3). Now, by Lemma
\ref{lem901}, we have
\[
\left\vert \frac{\xi(ks + h_{m,w}(k))}{\xi(ks + h_{m}(k) + 1)} \right\vert < 1 \quad
\text{for} \quad
\Re(s) > - \frac{h_{m,w}(k)+h_m(k)}{2k}.
\]
for every $m \in \Lambda_w(k)$. We have
\[
\frac{h_{m,w}(k)+h_m(k)}{2k} \geq \frac{c_p}{2},
\]
for every $m \in \Lambda_w(k)$, since $h_{m,w}(k) \geq \h
\beta_m(k)^\vee = kc_p - h_m(k)$ by Lemma \ref{lem_0815}~(5). Also,
we obtain
\[
\left\vert \frac{\xi(ks + h_{m,w}(k))}{\xi(ks + h_{m}(k) + 1)} \right\vert < 1 \quad
\text{for} \quad
\Re(s) = - \frac{c_p}{2}.
\]
unless $h_{m,w}(k)+h_m(k) = kc_p$.
If $h_{m,w}(k)+h_m(k) = kc_p$,
\[
\left\vert \frac{\xi(ks + h_{m,w}(k))}{\xi(ks + h_{m}(k) + 1)} \right\vert = 1 \quad
\text{for} \quad
\Re(s) = - \frac{c_p}{2}.
\]
Now we complete the proof of Proposition \ref{prop_0606} \hfill $\Box$
\medskip

Thus, we see that Proposition \ref{prop_0606} follows from Lemma
\ref{lem_0815} whose the essential part is due to Lemmas
\ref{lem_0801} and explicit descriptions of root systems.

\subsection{Proof of Proposition \ref{prop_0607}}
We have
\[
\aligned
D_p(s) = \prod_{k=1}^\infty
& \prod_{h \leq (kc_p+1)/2} \xi(ks+h)^{N_p(k,h-1)} \prod_{h> (kc_p+1)/2} \xi(ks+h)^{N_p(k,h)}.
\endaligned
\]
by Lemma \ref{lem_0806}. Hence we have
\begin{equation}\label{0903}
\frac{X_{p}^{\ddagger}(s)}{D_p(s)}
 = \prod_{k=1}^\infty \prod_{h> (kc_p+1)/2} \xi(ks+h)^{N_p(k,h-1)-N_p(k,h)}.
\end{equation}
Note that $N_p(k,h-1)-N_p(k,h) \geq 0$ by Lemma \ref{lem_0803}.
(This non-negativity should be hold by the construction of
$M_p(k,h)$.) It is well-known that $\xi(s)\not=0$ for $\Re(s) \geq
1$. Therefore, $\xi(ks+h)\not=0$ for $\Re(s) \geq (1 - h)/k$. On the
other hand, we have $(1-h)/k \leq -c_p/2$, since $kc_p$ and $h$ are
both integers. If $kc_p$ is odd, $h> (kc_p+1)/2$ implies $(1-h)/k
\leq - c_p/2 -1/(2k) < -c_p/2$. Hence
$X_{p}^{\ddagger}(s)/D_p(s)\not=0$ for $\Re(s) \geq -c_p/2$. If
$kc_p$ is even, $h> (kc_p+1)/2$ implies $(1-h)/k \leq -c_p/2$. It
also derives $X_{p}^{\ddagger}(s)/D_p(s)\not=0$ for $\Re(s) \geq
-c_p/2$. We complete the proof of Proposition \ref{prop_0607}.
\hfill $\Box$
\medskip

Thus, Proposition \ref{prop_0607} is a simple consequence of Lemma
\ref{lem_0806}. As well as Proposition \ref{prop_0606},
the essential part of Lemma \ref{lem_0806}
is due to Lemmas \ref{lem_0801} and \ref{lem_0804}.

%
%
\section{Proof of Proposition \ref{prop_0605}} \label{section_10}
%
%

As mentioned before, the most essential part of the proof of Proposition \ref{prop_0605}
is the volume hypothesis \eqref{VF} (or {\bf Theorem} of \cite[section 4.7]{MR2310297}) 
which plays a key role in the proof of Lemma \ref{lem1003} below.

\begin{lemma}\label{lem1001}
Let $w \in {\frak W}_p$.
If $|w^{-1}\Delta \setminus \Phi_p|=1$,
the set $\Phi^+ \setminus \Phi_p^+$
equals to either
$(\Phi^+ \setminus \Phi_p^+) \cap w^{-1}\Phi^+$
or
$(\Phi^+ \setminus \Phi_p^+) \cap w^{-1}\Phi^-$.
\end{lemma}
\begin{proof}
We show that $\Phi^+ \cap w(\Phi^+ \setminus \Phi_p^+)= w(\Phi^+
\setminus \Phi_p^+)$ or $\Phi^- \cap w(\Phi^+ \setminus \Phi_p^+)=
w(\Phi^+ \setminus \Phi_p^+)$ if $|\Delta \setminus w \Phi_p|=1$.
Suppose $|\Delta \setminus w \Phi_p|=1$, and denote by $\alpha_w$
the only one simple root in $\Delta \setminus w\Phi_p$. Then
$\alpha_w$ belongs to $\Phi^+ \cap w(\Phi^+ \setminus \Phi_p^+)$ or
$\Phi^+ \cap w(\Phi^- \setminus \Phi_p^-)$, and hence $\alpha_w$
belongs to $\Phi^+ \cap w(\Phi^+ \setminus \Phi_p^+)$ or $-\alpha_w$
belongs to $\Phi^- \cap w(\Phi^+ \setminus \Phi_p^+)$.

Put $\Delta_w^+=\Delta \cap w \Phi_p^+$ and $\Delta_w^-=\Delta \cap
w \Phi_p^-$ so that \[ \Delta = \Delta_w^+ \cup \Delta_w^- \cup
\{\alpha_w\}
\]
(disjoint). Then we claim $\Delta_w^+=\Delta \cap w \Delta_p$. In
fact $\Delta_w^+ \supset \Delta \cap w \Delta_p$ is obvious. An
arbitrary $\alpha \in \Delta_w^+$ has the form
\[
\alpha = \sum_{j\not=p} n_ j w\alpha_j \quad (n_j \geq 0, w\alpha_j \in  \Phi^{-} \cup \Delta),
\]
since $w\Delta_p \subset \Phi^{-} \cup \Delta$. The right-hand side
decomposes into two parts according to $w\alpha_j \in \Phi^-$ or
$w\alpha_j \in \Delta$:
\[
\alpha = \sum_{{j\not=p}\atop{w\alpha_j \in \Phi^-}} n_ j w\alpha_j
+ \sum_{{j\not=p}\atop{w\alpha_j \in \Delta}} n_ j w\alpha_j.
\]
We note that the left-hand side is a simple root. Therefore,
$\alpha$ should be one of $\Delta \cap w \Delta_p$, and hence
$\Delta_w^+=\Delta\cap w\Delta_p$.

Now we prove the assertion of the lemma by using a different way
according to two cases $w\alpha_p \in \Phi^{+}$ or $w\alpha_p \in \Phi^{-}$.

First we deal with the case $w\alpha_p \in \Phi^{+}$.
Assume $\alpha_w \in \Phi^+ \cap w(\Phi^+ \setminus \Phi_p^+)$.
Then we will have a contradiction unless $\Phi^- \cap w(\Phi^+ \setminus \Phi_p^+) = \emptyset$.
If $\Phi^- \cap w(\Phi^+ \setminus \Phi_p^+) \not= \emptyset$,
there exists at least one simple root $\alpha_k$
such that
\begin{equation}\label{1001}
- \alpha_k \in \Phi^- \cap w(\Phi^+ \setminus \Phi_p^+).
\end{equation}
Actually we have
\[
\sum_i a_i w^{-1}(- \alpha_i) = w^{-1}(-\sum_{i} a_i \alpha_i)=b_p \alpha_p + \sum_{j\not=p} b_j \alpha_j \quad (b_p >0, \, a_i, b_j \geq 0)
\]
for $-\sum_{i}a_i \alpha_i \in \Phi^- \cap w(\Phi^+ \setminus \Phi_p^+)$.
Hence we have
\[
w^{-1}(-\alpha_k) = b_p^\prime \alpha_p + \sum_{j\not=p} b_j^\prime \alpha_j \quad (b_p^\prime >0,\,, b_j^\prime \geq 0)
\]
for at least one $\alpha_k$, since $w^{-1}(-\alpha_i) \in \Phi$.

The simple root $\alpha_k$ belongs to one of $\Delta_w^+$,
$\Delta_w^-$ or $\{\alpha_w\}$. If $\alpha_k \in \Delta_w^+$, we
have $\alpha_k \in \Delta \cap w \Delta_p \subset \Phi^+ \cap w
\Phi^+$. This contradicts \eqref{1001}. If $\alpha_k \in
\Delta_w^-$, we have $\alpha_k \in \Phi^+ \cap w \Phi_p^{-}$. This
contradicts  \eqref{1001}. If $\alpha_k = \alpha_w$, we have
$\alpha_k=\alpha_w \in \Phi^+ \cap w(\Phi^+ \setminus \Phi_p^+)$ by
assumption. This also contradicts   \eqref{1001}. Hence $\Phi^- \cap
w(\Phi^+ \setminus \Phi_p^+) = \emptyset$ which implies  $\Phi^+
\cap w(\Phi^+ \setminus \Phi_p^+) = w(\Phi^+ \setminus \Phi_p^+)$.

On the other hand, we have $\Phi^- \cap w(\Phi^+ \setminus \Phi_p^+) = w(\Phi^+ \setminus \Phi_p^+)$
if we assume that $-\alpha_w$ belongs to $\Phi^- \cap w(\Phi^+ \setminus \Phi_p^+)$
by a way similar to the above.

Finally we deal with the case $w\alpha_p \in \Phi^{-}$. We show that
$(\Phi^+ \setminus \Phi_p^+) \cap w^{-1}\Phi^{-}=(\Phi^+ \setminus
\Phi_p^+)$ by induction on $k$ for $\Sigma_p(k)$ if $w\alpha_p \in
\Phi^{-}$ (i.e. $\alpha_p \in (\Phi^+ \setminus \Phi_p^+) \cap
w^{-1}\Phi^{-}$). By Lemma \ref{lem_0801}, we have $\Sigma_p(1)
\subset (\Phi^+ \setminus \Phi_p^+) \cap w^{-1}\Phi^{-}$ if
$w\alpha_p \in \Phi^{-}$. Suppose that $\sum_p(k-1) \subset (\Phi^+
\setminus \Phi_p^+) \cap w^{-1}\Phi^{-}$ for $1 < k \leq k_p$. Let
$\alpha_p^-(k)$ be the lowest element of $\Sigma_p(k)$. Then there
exists $\beta \in \Sigma_p(k-1)$ such that
\[
\alpha_p^-(k)^\vee = \beta^\vee + \alpha_p^\vee,
\]
since $(\alpha_p^-(k)^\vee-\alpha_j^\vee)^\vee$ is not a root for any $j \not=p$ by the lowest property of $\alpha_p^-(k)^\vee$.
We have $w \alpha_p \in \Phi^-$ by the first assumption
and $w \beta \in \Phi^{-}$ by the assumption of induction.
Therefore $w \alpha_p^-(k) = (w \beta^\vee + w \alpha_p^\vee)^\vee \in \Phi^{-}$,
i.e. $\alpha_p^-(k) \in (\Phi^+ \setminus \Phi_p^+) \cap w^{-1}\Phi^{-}$.
Hence  $\Sigma_p(k) \subset (\Phi^+ \setminus \Phi_p^+) \cap w^{-1}\Phi^{-}$ by Lemma \ref{lem_0801}.
\end{proof}

\begin{lemma} \label{lem1002} We have
\[
\{ w \in {\frak W}_p^\ddagger ~|~ |\left(w^{-1}\Delta\right)
\setminus \Phi_p|=1 \} = \{ w \in W_p ~|~ \Delta_p \subset
w^{-1}(\Delta_p \cup \Phi_p^-)\}.
\]
\end{lemma}
\begin{proof}
At first we prove the left-hand side is contained in the
right-hand side. We have
\[
\{ w \in {\frak W}_p^\ddagger ~|~ |\left(w^{-1}\Delta\right)\setminus \Phi_p|=1 \} = \{ w \in {\frak W}_p ~|~ l_p(w)=0,~\,
|\left(w^{-1}\Delta\right)\setminus \Phi_p|=1\}.
\]
by definition of ${\frak W}_p^\ddagger$ in Lemma \ref{lem_0603}. 
  Since $l_p(w)=0$ implies
  \begin{equation*}
    \begin{split}
      \Phi^+\cap w^{-1}\Phi^-=
      ((\Phi^+\setminus\Phi_p^+)\cap w^{-1}\Phi^-)\cup
      (\Phi_p^+\cap w^{-1}\Phi^-)=\Phi_p^+\cap w^{-1}\Phi^-\subset\Phi_p^+,
    \end{split}
  \end{equation*}
  we have $w\in W_p$ 
  and hence ${\frak W}_p^\ddagger \subset W_p$.
  On the other hand
  for 
  $w\in W_p$, we have 
$\Delta_p \cap w^{-1} (\Delta \cup \Phi^-) = \Delta_p \cap w^{-1} (\Delta_p \cup \Phi_p^-)$
because
    $w^{-1}(\Phi \setminus \Phi_p)\cap\Phi_p=\emptyset$.
Therefore the left-hand side is contained in the right-hand side.
\medskip

Next we prove the right-hand side is contained in the
left-hand side. 
We see that
$\Delta_p \subset w^{-1} (\Delta_p \cup \Phi_p^-)$ implies $\Delta_p \subset w^{-1} (\Delta \cup \Phi^-)$.
Hence the right-hand side is contained in ${\frak W}_p$. 
>From the proof of Lemma \ref{lem_0601}, we see that $w \in W_p$ implies $l_p(w)=0$, which shows that
the right-hand side is contained in ${\frak W}_p^\ddagger$.
Furthermore for $w\in W_p$, we have
\[
  \left(w^{-1}\Delta\right)\setminus \Phi_p
  =
  (\left(w^{-1}\Delta_p\right)\setminus \Phi_p) \cup (\{w^{-1}\alpha_p\}\setminus \Phi_p)
  = \emptyset \cup \{w^{-1}\alpha_p\}
  = \{w^{-1}\alpha_p\}
\]
because $w^{-1}\Phi_p\subset \Phi_p$
and $w^{-1}\alpha_p\notin\Phi_p$.
Therefore $w \in W_p$ implies $|\left(w^{-1}\Delta\right)\setminus \Phi_p|=1$.
\end{proof}

\begin{lemma}\label{lem1003}
Assume the volume hypothesis \eqref{VF}.
Let $C_{p,w}$ be real numbers defined in \eqref{0502_1}.
Define real numbers $D_{p,w}$ by
\begin{equation}\label{1001_1}
\aligned
\frac{1}{D_{p,w}}
& = 2^{|\Delta_p \cap w^{-1}\Phi^+|}
\prod_{\alpha \in (w^{-1} \Delta) \cap (\Phi_p \setminus \Delta_p)} (\h \alpha^\vee -1)  \\
& \quad \times \prod_{\alpha \in \Phi_p^+ \setminus \Delta_p}
(\h \alpha^\vee + \delta_{\alpha,w})(\h \alpha^\vee + \delta_{\alpha,w} -1),
\endaligned
\end{equation}
where $\delta_{\alpha,w}$ is defined in \eqref{e:delta}. Then
\begin{equation}\label{1002}
\sum_{{w \in {\frak W}_p^{\ddagger}}\atop{|(w^{-1} \Delta)\setminus \Phi_p|=1}}
\frac{1}{\langle \lambda_p, \alpha_w^\vee \rangle }
C_{p,w}D_{p,w} \not=0,
\end{equation}
where $\alpha_w$ is the only one element of $(w^{-1} \Delta)\setminus \Phi_p$.
\end{lemma}
\begin{proof} We first prove
\begin{equation} \label{tem_01}
\aligned
\underset{\lambda=\rho}{\rm Res} & \,
\omega_{\Delta}^{\Phi}(\lambda) \\
&= \sum_{{w \in W}\atop{\Delta \subset w^{-1}(\Delta \cup \Phi^-)}}
\hat{\zeta}(2)^{-|\Delta \cap w^{-1} \Phi^-|}
\prod_{\alpha \in (w^{-1}\Delta) \setminus \Delta}
\frac{1}{\h\alpha^\vee -1}
\prod_{\alpha \in (\Phi^+ \setminus \Delta) \cap w^{-1}\Phi^-}
\frac{\hat{\zeta}(\h\alpha^\vee)}
{\hat{\zeta}(\h\alpha^\vee +1)}
\endaligned
\end{equation}
holds for an arbitrary reduced root system $\Phi$ (which is not
necessary irreducible) and its fundamental system $\Delta$. If
$\Phi$ is not irreducible and its irreducible decomposition is
$\Phi=\Phi_1 \amalg  \cdots \amalg \Phi_m$, we find that
$\underset{\lambda=\rho}{\rm Res} \, \omega_{\Delta}^{\Phi}(\lambda)
= \prod_{i=1}^{m} \underset{\lambda=\rho_i}{\rm Res} \,
\omega_{\Delta_i}^{\Phi_i}(\lambda)$, $\Delta_i = \Delta \cap
\Phi_i$, $\rho_i = \frac{1}{2} \sum_{\alpha \in \Phi^+ \cap \Phi_i}
\alpha$ and \eqref{tem_01} holds for each component, since the Weyl
group $W(\Phi)$ of $\Phi$ decomposes into
\[
W(\Phi) = W(\Phi_1) \times \cdots \times W(\Phi_m).
\]
We apply \eqref{tem_01} to $\Phi_p$, $\Delta_p$ and $\rho_p$ later.
\smallskip

We have
\[
\underset{\lambda=\rho}{\rm Res} \,
\omega_{\Delta}^{\Phi}(\lambda)
= \underset{s=0}{\rm Res}  \,
\omega_{\Delta,p}^{\Phi}(s)
\]
by \eqref{e:omegap}. Therefore, it is enough to show that
$\underset{s=0}{\rm Res}  \, \omega_{\Delta,p}^{\Phi}(s)$ equals to
the right-hand side in \eqref{tem_01}. By (2.8) in \cite{Ko} we have
\[
\aligned
\omega_{\Delta,p}^{\Phi}(s)
&= \sum_{w \in {\frak W}_p}
\hat{\zeta}(2)^{-|\Delta_p \cap w^{-1} \Phi^-|}
\prod_{\alpha \in (w^{-1}\Delta) \cap (\Phi_p \setminus \Delta_p)}
\frac{1}{\h\alpha^\vee -1}
\prod_{\alpha \in (\Phi_p^+ \setminus \Delta_p) \cap w^{-1}\Phi^{-}}
\frac{\hat{\zeta}(\h\alpha^\vee) }{\hat{\zeta}(\h\alpha^\vee +1)} \\
& \times
\prod_{\alpha \in (w^{-1}\Delta) \setminus \Phi_p}
\frac{1}{\langle \lambda_p, \alpha^\vee \rangle s + \h\alpha^\vee -1 }
\prod_{\alpha \in (\Phi^+ \setminus \Phi_p^+) \cap w^{-1}\Phi^-}
\frac{\hat{\zeta}(\langle \lambda_p,\alpha^\vee \rangle s + \h\alpha^\vee)}
{\hat{\zeta}(\langle \lambda_p,\alpha^\vee \rangle s + \h\alpha^\vee +1)}.
\endaligned
\]
Here we note that
\[
\aligned
(&1) \quad
\Delta \cap ((w^{-1}\Delta) \setminus \Phi_p)
= \{\alpha_p\} \cap (w^{-1}\Delta)
=
\begin{cases}
\{\alpha_p \}, &  \text{if} ~\alpha_p \in w^{-1}\Delta, \\
\,\,\,\, \emptyset, & \text{otherwise},
\end{cases} \\
\text{and} \\
(&2) \quad
\Delta \cap ((\Phi^+ \setminus \Phi_p^+) \cap w^{-1}\Phi^-)
=
\begin{cases}
\{\alpha_p \}, &  \text{if} ~\alpha_p \in w^{-1}\Phi^-, \\
\,\,\,\, \emptyset, & \text{otherwise}.
\end{cases}
\endaligned
\]
In addition, (1) and (2) do not occur simultaneously. Using these
facts together with the fact that $\langle \lambda_p,\alpha_p^\vee
\rangle=1$, we obtain that $\underset{s=0}{\rm Res}\,
\omega_{\Delta,p}^{\Phi}(s)$ is
\[
\aligned &\sum_{{w \in {\frak W}_p}\atop{\alpha_p \in w^{-1}\Delta}}
\hat{\zeta}(2)^{-|\Delta_p \cap w^{-1} \Phi^-|}\prod_{\alpha \in
(w^{-1}\Delta) \cap (\Phi_p \setminus \Delta_p)}
\frac{1}{\h\alpha^\vee -1}  \prod_{\alpha \in (\Phi_p^+ \setminus
\Delta_p) \cap w^{-1}\Phi^{-}}
\frac{\hat{\zeta}(\h\alpha^\vee) }{\hat{\zeta}(\h\alpha^\vee +1)} \\
& \times
\frac{1}{\langle \lambda_p, \alpha_p^\vee \rangle}
\prod_{\alpha \in (w^{-1}\Delta) \setminus (\Phi_p \cup \{\alpha_p \})}
\frac{1}{\h\alpha^\vee -1 }
\prod_{\alpha \in (\Phi^+ \setminus \Phi_p^+) \cap w^{-1}\Phi^-}
\frac{\hat{\zeta}(\h\alpha^\vee)}
{\hat{\zeta}(\h\alpha^\vee +1)} \\
& + \sum_{{w \in {\frak W}_p}\atop{\alpha_p \in w^{-1}\Phi^-}}
\frac{\hat{\zeta}(2)^{-|\Delta_p \cap w^{-1} \Phi^-|}}
{\hat{\zeta}(2)\langle \lambda_p,\alpha_p^\vee \rangle}\prod_{\alpha
\in (w^{-1}\Delta) \cap (\Phi_p \setminus \Delta_p)}
\frac{1}{\h\alpha^\vee -1} \prod_{\alpha \in (\Phi_p^+ \setminus
\Delta_p) \cap w^{-1}\Phi^{-}}
\frac{\hat{\zeta}(\h\alpha^\vee) }{\hat{\zeta}(\h\alpha^\vee +1)} \\
& \times \prod_{\alpha \in (w^{-1}\Delta) \setminus \Phi_p}
\frac{1}{\h\alpha^\vee -1} \prod_{\alpha \in (\Phi^+ \setminus
\Phi_p^+ \cup \{\alpha_p \}) \cap w^{-1}\Phi^-}
\frac{\hat{\zeta}(\h\alpha^\vee)}
{\hat{\zeta}(\h\alpha^\vee +1)} \\
&= \sum_{{w \in {\frak W}_p}\atop{\alpha_p \in w^{-1}\Delta}}
\hat{\zeta}(2)^{-|\Delta \cap w^{-1} \Phi^-|}\prod_{\alpha \in
(w^{-1}\Delta) \setminus \Delta} \frac{1}{\h\alpha^\vee -1}
\prod_{\alpha \in (\Phi^+ \setminus \Delta) \cap w^{-1}\Phi^-}
\frac{\hat{\zeta}(\h\alpha^\vee)} {\hat{\zeta}(\h\alpha^\vee +1)}
\\
& + \sum_{{w \in {\frak W}_p}\atop{\alpha_p \in w^{-1}\Phi^-}}
\hat{\zeta}(2)^{-|\Delta \cap w^{-1} \Phi^-|} \prod_{\alpha \in
(w^{-1}\Delta) \setminus \Delta} \frac{1}{\h\alpha^\vee -1}
\prod_{\alpha \in (\Phi^+ \setminus \Delta) \cap w^{-1}\Phi^-}
\frac{\hat{\zeta}(\h\alpha^\vee)} {\hat{\zeta}(\h\alpha^\vee +1)}.
\endaligned
\]
Hence we obtain \eqref{tem_01}.
\medskip

We recall $C_{p,w}$ in \eqref{0502_1}. We claim
\begin{equation}\label{tem_02}
\sum_{{w \in {\frak W}_p^{\ddagger}}\atop{|(w^{-1} \Delta)\setminus \Phi_p|=1}}
\frac{1}{\langle \lambda_p, \alpha_w^\vee \rangle }\,
C_{p,w}D_{p,w}
= \prod_{\alpha \in \Phi_p^+} \hat{\zeta}(\h \alpha^\vee+1) \cdot
\underset{\lambda=\rho_p}{\rm Res}\,
\omega_{\Delta_p}^{\Phi_p}(\lambda).
\end{equation}
We justify this as follows. We have
\[
C_{p,w}D_{p,w}
= \hat{\zeta}(2)^{|\Delta_p \cap w^{-1}\Phi^+|}\prod_{\alpha \in (w^{-1} \Delta) \cap (\Phi_p \setminus \Delta_p)} \frac{1}{\h \alpha^\vee -1}
\prod_{\alpha \in \Phi_p^{+}\setminus \Delta_p}
\hat{\zeta}({\rm ht}\, \alpha^\vee +\delta_{\alpha,w}) .
\]
by definitions of $C_{p,w}$ and $D_{p,w}$. Therefore
\[
\aligned
\,\sum_{{w \in {\frak W}_p^{\ddagger}}\atop{|(w^{-1} \Delta)\setminus \Phi_p|=1}} &
\frac{1}{\langle \lambda_p, \alpha_w^\vee \rangle } \,
C_{p,w}D_{p,w} \\
&  \quad = \sum_{{w \in {\frak W}_p^{\ddagger}}\atop{|(w^{-1}
\Delta)\setminus \Phi_p|=1}} \frac{\hat{\zeta}(2)^{|\Delta_p \cap
w^{-1}\Delta|} }{\langle \lambda_p, \alpha_w^\vee \rangle }
\prod_{\alpha \in (w^{-1} \Delta) \cap (\Phi_p \setminus \Delta_p)}
\frac{1}{\h \alpha^\vee -1}
\\
& \quad \times
\prod_{\alpha \in (\Phi_p^{+} \setminus \Delta_p) \cap w^{-1}\Phi^{+}}
\hat{\zeta}({\rm ht}\, \alpha^\vee + 1)
\prod_{\alpha \in (\Phi_p^{+}\setminus \Delta_p) \cap w^{-1}\Phi^{-}}
\hat{\zeta}({\rm ht}\, \alpha^\vee),
\endaligned
\]
since $|\Delta_p \cap w^{-1}\Phi^+|=|\Delta_p \cap w^{-1}\Delta|$
if $w \in {\frak W}_p$.
If $w \in W_p$, we have
\[
(w^{-1} \Delta) \cap (\Phi_p \setminus \Delta_p)
= (w^{-1} \Delta_p) \setminus \Delta_p;\qquad\Delta_p \cap w^{-1}\Delta
=\Delta_p \cap w^{-1} \Delta_p.
\]
Moreover, if  $w \in W_p$, we have
\[
\aligned
(\Phi_p^{+}\setminus \Delta_p) \cap w^{-1}\Phi^{+}
& = (\Phi_p^{+}\setminus \Delta_p) \cap w^{-1}\Phi_p^{+};  \\
(\Phi_p^{+}\setminus \Delta_p) \cap w^{-1}\Phi^{-}
& = (\Phi_p^{+}\setminus \Delta_p) \cap w^{-1}\Phi_p^{-},
\endaligned
\]
since
\[
\aligned
(\Phi_p^{+}\setminus \Delta_p) \cap w^{-1}\Phi^{+}
& =\Bigl( (\Phi_p^{+}\setminus \Delta_p) \cap w^{-1}\Phi_p^{+} \Bigr) \cup \Bigl( (\Phi_p^{+}\setminus \Delta_p) \cap w^{-1}(\Phi^{+} \setminus \Phi_p^{+}) \Bigr) \\
(\Phi_p^{+}\setminus \Delta_p) \cap w^{-1}\Phi^{-}
& =\Bigl( (\Phi_p^{+}\setminus \Delta_p) \cap w^{-1}\Phi_p^{-} \Bigr) \cup \Bigl( (\Phi_p^{+}\setminus \Delta_p) \cap w^{-1}(\Phi^{-} \setminus \Phi_p^{-}) \Bigr),
\endaligned
\]
and
\[
(\Phi_p^{+}\setminus \Delta_p) \cap w^{-1}(\Phi^{+} \setminus \Phi_p^{+})
= \emptyset, \qquad
(\Phi_p^{+}\setminus \Delta_p) \cap w^{-1}(\Phi^{-} \setminus \Phi_p^{-})
= \emptyset
\]
for $w \in W_p$.
Hence, by Lemma \ref{lem1002} and the fact that
$\langle \lambda_p, w^{-1}\alpha_p^\vee \rangle=\langle w \lambda_p, \alpha_p^\vee \rangle=1$
for $w \in W_p$, we obtain
\[
\aligned
\,& \sum_{{w \in {\frak W}_p^{\ddagger}}\atop{|(w^{-1} \Delta)\setminus \Phi_p|=1}}
\frac{1}{\langle \lambda_p, \alpha_w^\vee \rangle }
C_{p,w}D_{p,w} \\
&  \quad = \sum_{{w \in W_p}\atop{\Delta_p \subset w^{-1}(\Delta_p
\cup \Phi_p^-)}} \frac{\hat{\zeta}(2)^{|\Delta_p \cap
w^{-1}\Delta_p|}}{\langle \lambda_p, w^{-1}\alpha_p^\vee \rangle }
\prod_{\alpha \in (w^{-1} \Delta_p) \setminus \Delta_p} \frac{1}{\h
\alpha^\vee -1}
 \\
& \quad \quad \times
\prod_{\alpha \in (\Phi_p^{+} \setminus \Delta_p) \cap w^{-1}\Phi_p^{+}}
\hat{\zeta}({\rm ht}\, \alpha^\vee + 1)
\prod_{\alpha \in (\Phi_p^{+}\setminus \Delta_p) \cap w^{-1}\Phi_p^{-}}
\hat{\zeta}({\rm ht}\, \alpha^\vee) \\
&  \quad = \sum_{{w \in W_p}\atop{\Delta_p \subset w^{-1}(\Delta_p
\cup \Phi_p^-)}} \hat{\zeta}(2)^{|\Delta_p \cap
w^{-1}\Delta_p|}\prod_{\alpha \in (w^{-1} \Delta_p) \setminus
\Delta_p} \frac{1}{\h \alpha^\vee -1}
\\
& \quad \quad \times
\prod_{\alpha \in (\Phi_p^{+} \setminus \Delta_p) \cap w^{-1}\Phi_p^{+}}
\hat{\zeta}({\rm ht}\, \alpha^\vee + 1)
\prod_{\alpha \in (\Phi_p^{+}\setminus \Delta_p) \cap w^{-1}\Phi_p^{-}}
\hat{\zeta}({\rm ht}\, \alpha^\vee).
\endaligned
\]
With this formula, we obtain \eqref{tem_02} by applying
\eqref{tem_01} to $\Phi_p$, $\Delta_p$ and $\rho_p$.

By the volume hypothesis \eqref{VF}, 
the right-hand side of \eqref{tem_02} is positive, in particular, it is not equal to
zero. In fact, it should be a product of special values of the
Riemann zeta function and volumes of several (truncated) domains
corresponding to irreducible components of $\Phi_p$. Hence
\eqref{1002} follows. We complete the proof of Lemma \ref{1002}.
\end{proof}

\noindent
{\bf Proof of Proposition \ref{prop_0605}.}
Using real numbers $C_{p,w}$ and $D_{p,w}$ defined in \eqref{0502_1} and \eqref{1001_1}, respectively,
we have
\begin{equation*}
\aligned
Q_{p,w}(s)
& = C_{p,w}\, \prod_{{v \in {\mathfrak W}_p}\atop{v \not= w}}
\left[ 2^{|\Delta_p \cap v^{-1}\Phi^+|}
\prod_{\alpha \in (v^{-1} \Delta)\setminus \Delta_p}
\Bigl( \langle \lambda_p, \alpha^\vee \rangle s + \h \alpha^\vee -1 \Bigr)
\right. \\
& \qquad \qquad \times \left.
\prod_{\alpha \in \Phi^+ \setminus \Delta_p}
\Bigl(\langle \lambda_p, \alpha^\vee \rangle s + \h \alpha^\vee + \delta_{\alpha,v} \Bigr)
\Bigl(\langle \lambda_p, \alpha^\vee \rangle s + \h \alpha^\vee + \delta_{\alpha,v} -1 \Bigr)
\right] \\
& = C_{p,w}\, \prod_{{v \in {\mathfrak W}_p}\atop{v \not= w}}
\left[ D_{p,v}^{-1} \prod_{\alpha \in (v^{-1} \Delta)\setminus \Phi_p}
\Bigl( \langle \lambda_p, \alpha^\vee \rangle s + \h \alpha^\vee -1 \Bigr) \right. \\
& \qquad \qquad \times \left.
\prod_{\alpha \in \Phi^+ \setminus \Phi_p^+}
\Bigl(\langle \lambda_p, \alpha^\vee \rangle s + \h \alpha^\vee + \delta_{\alpha,v} \Bigr)
\Bigl(\langle \lambda_p, \alpha^\vee \rangle s + \h \alpha^\vee + \delta_{\alpha,v} -1 \Bigr)
\right] \\
& = Q_p^\ast(s) \left[ C_{p,w}D_{p,w}
\prod_{\alpha \in (w^{-1} \Delta) \setminus \Phi_p}
\frac{1}{\langle \lambda_p, \alpha^\vee \rangle s + \h \alpha^\vee -1} \right. \\
& \qquad \qquad \times \left.
\prod_{\alpha \in \Phi^+ \setminus \Phi_p^+}
\frac{1}{
(\langle \lambda_p, \alpha^\vee \rangle s + \h \alpha^\vee + \delta_{\alpha,w})
(\langle \lambda_p, \alpha^\vee \rangle s + \h \alpha^\vee + \delta_{\alpha,w} -1)
}
\right] ,
\endaligned
\end{equation*}
where in the second line, we used that fact that $\alpha \in \Phi_p$ implies $\langle \lambda_p, \alpha^\vee \rangle=0$, and 
$Q_p^\ast(s)$ is the polynomial
\[
\aligned Q_p^\ast(s) &= \prod_{v \in {\mathfrak W}_p} \left[
D_{p,v}^{-1} \prod_{\alpha \in (v^{-1} \Delta)\setminus \Phi_p}
\Bigl( \langle \lambda_p, \alpha^\vee \rangle s + \h \alpha^\vee -1 \Bigr)  \right. \\
& \qquad \qquad \times \left. \prod_{\alpha \in  \Phi^+ \setminus
\Phi_p^+} \Bigl(\langle \lambda_p, \alpha^\vee \rangle s + \h
\alpha^\vee + \delta_{\alpha,v} \Bigr) \Bigl(\langle \lambda_p,
\alpha^\vee \rangle s + \h \alpha^\vee + \delta_{\alpha,v} -1 \Bigr)
\right].
\endaligned
\]
In particular, we have
\[
\aligned
Q_p^{\ddagger}(s)
& = \sum_{v \in {\frak W}_p^{\ddagger}} Q_{p,v}(s) \\
& = Q_p^\ast(s) \sum_{v \in {\frak W}_p^{\ddagger}} \left[
C_{p,v}D_{p,v} \prod_{\alpha \in (v^{-1} \Delta)\setminus \Phi_p}
\frac{1}{\langle \lambda_p, \alpha^\vee \rangle s + \h \alpha^\vee -1} \right. \\
& \qquad \qquad \times \left. \prod_{\alpha \in  \Phi^+ \setminus
\Phi_p^+} \frac{1}{ (\langle \lambda_p, \alpha^\vee \rangle s + \h
\alpha^\vee + \delta_{\alpha,v}) (\langle \lambda_p, \alpha^\vee
\rangle s + \h \alpha^\vee + \delta_{\alpha,v} -1) } \right] .
\endaligned
\]
Hence we have
\[
\aligned \frac{Q_{p,w}(s)}{Q_p^{\ddagger}(s)} & =
\frac{\displaystyle C_{p,w}D_{p,w} \prod_{\alpha \in (w^{-1}
\Delta)\setminus \Phi_p} \frac{1}{\langle \lambda_p, \alpha^\vee
\rangle s + \h \alpha^\vee -1} } {\displaystyle \sum_{v \in {\frak
W}_p^{\ddagger}} \left[ C_{p,v}D_{p,v} \prod_{\alpha \in (v^{-1}
\Delta)\setminus \Phi_p} \frac{1}{\langle \lambda_p, \alpha^\vee
\rangle s + \h \alpha^\vee -1} \Bigl( 1+o(1) \Bigr) \right]}
\endaligned
\]
as $|s| \to \infty$. 
Note that in the denominator, $1+o(1)$ comes from
\[
\prod_{\alpha \in  \Phi^+ \setminus
\Phi_p^+} \frac{(\langle \lambda_p, \alpha^\vee \rangle s + \h \alpha^\vee + \delta_{\alpha,w})(\langle \lambda_p, \alpha^\vee \rangle s + \h \alpha^\vee + \delta_{\alpha,w} -1)}
{ (\langle \lambda_p, \alpha^\vee \rangle s + \h \alpha^\vee + \delta_{\alpha,v}) (\langle \lambda_p, \alpha^\vee \rangle s + \h \alpha^\vee + \delta_{\alpha,v} -1) }
= 1+ O(|s|^{-1}).
\]

Concerning the numerator: 
If $|(w^{-1}\Delta)\setminus \Phi_p|=1$, we have 
(i) $\Phi^+\setminus \Phi_p^+ = (\Phi^+\setminus \Phi_p^+) \cap w^{-1}\Phi^+$ 
or 
(ii) $\Phi^+\setminus \Phi_p^+ = (\Phi^+\setminus \Phi_p^+) \cap w^{-1}\Phi^-$ 
by Lemma \ref{lem1001}. 
We have $w \in {\frak W}_p^\ddagger$ by Lemma \ref{lem_0603} in the case (i). 
It contradicts $w \not\in {\frak W}_p^\ddagger$. 
On the other hand, we have $w \in {\frak W}_p^-$ by Definition \ref{def_0506} in the case (ii). 
It contradicts $w \in {\frak W}_p \setminus {\frak W}_p^- = {\frak W}_p^+ \cup {\frak W}_p^0$. 
Hence the degree of $\prod_{\alpha \in (w^{-1} \Delta)\setminus \Phi_p}(\langle \lambda_p, \alpha^\vee \rangle s + \h \alpha^\vee -1)$ 
is at least two. 

Concerning the denominator, we have 
\[
\aligned 
\, & \sum_{v \in {\frak
W}_p^{\ddagger}} \left[ C_{p,v}D_{p,v} \prod_{\alpha \in (v^{-1}
\Delta)\setminus \Phi_p} \frac{1}{\langle \lambda_p, \alpha^\vee
\rangle s + \h \alpha^\vee -1} \right] \\
&= 
\sum_{{v \in {\frak W}_p^{\ddagger}}\atop{|(v^{-1}\Delta)\setminus \Phi_p|=1}} \left[ C_{p,v}D_{p,v} \frac{1}{\langle \lambda_p, \alpha_v^\vee
\rangle s + \h \alpha^\vee -1} \right] \\
&\quad + \sum_{{v \in {\frak W}_p^{\ddagger}}\atop{|(v^{-1}\Delta)\setminus \Phi_p|>1}} \left[ C_{p,v}D_{p,v} \prod_{\alpha \in (v^{-1}
\Delta)\setminus \Phi_p} \frac{1}{\langle \lambda_p, \alpha^\vee
\rangle s + \h \alpha^\vee -1} \right]
\endaligned
\]
and the first term of the right-hand side is not zero by Lemma \ref{lem1003}, 
where $\alpha_v \in (v^{-1}\Delta)\setminus \Phi_p$. 
Therefore, 
\[
\aligned 
\, & \sum_{v \in {\frak
W}_p^{\ddagger}} \left[ C_{p,v}D_{p,v} \prod_{\alpha \in (v^{-1}
\Delta)\setminus \Phi_p} \frac{1}{\langle \lambda_p, \alpha^\vee
\rangle s + \h \alpha^\vee -1} \right] \\
&= \left( 
\sum_{{v \in {\frak W}_p^{\ddagger}}\atop{|(v^{-1}\Delta)\setminus \Phi_p|=1}} \left[ C_{p,v}D_{p,v} \frac{1}{\langle \lambda_p, \alpha_v^\vee
\rangle s + \h \alpha^\vee -1} \right] \right)(1+o(1)) \quad \text{as $|s| \to \infty$.}
\endaligned
\]
Now we obtain 
\[
\aligned \frac{Q_{p,w}(s)}{Q_p^{\ddagger}(s)} & = 
\frac{1}{s^{r-1}}
\frac{\displaystyle C_{p,w}D_{p,w} \prod_{\alpha \in (w^{-1}\Delta)\setminus \Phi_p} \frac{1}{\langle \lambda_p, \alpha^\vee \rangle + (\h \alpha^\vee -1)/s} } 
{\displaystyle \sum_{{v \in {\frak W}_p^{\ddagger}}\atop{|(v^{-1}\Delta)\setminus \Phi_p|=1}} \left[ C_{p,v}D_{p,v} 
\frac{1}{\langle \lambda_p, \alpha_v^\vee \rangle + (\h \alpha^\vee -1)/s} \right]}  \Bigl( 1+o(1) \Bigr)
\endaligned
\]
as $|s| \to \infty$, where $r=|(w^{-1}\Delta)\setminus \Phi_p|>2$. 
Therefore $|Q_{p,w}(s)/Q_p^{\ddagger}(s)|=o(1)$
as $|s| \to \infty$ for every $w \in ({\frak W}_p^+ \cup {\frak
W}_p^0) \setminus {\frak W}_p^\ddagger$ by Lemma \ref{lem1003}. 
We complete the proof of Proposition \ref{prop_0605}.\hfill $\Box$

%
%
\section{Proof of Proposition \ref{prop_0609}} \label{section_11}
%
%
Write $s=\sigma+it$ ($\sigma,t\in\R$).
By Lemma \ref{lem_0603}, we have
\[
E_p(s)
=  X_{p}^{\ddagger}(s)[ Q_{p}^{\ddagger}(s) + V_p(s)]
\]
with
\[
V_p(s)  = \sum_{w \in {\frak W}_{p}^{+} \setminus {\frak W}_{p}^{\ddagger}}
Q_{p,w}(s) \cdot \frac{X_{p,w}(s)}{X_{p}^{\ddagger}(s)}
+ \quad \frac{1}{2} \sum_{w \in {\frak W}_{p}^{0}}
Q_{p,w}(s) \cdot \frac{X_{p,w}(s)}{X_{p}^{\ddagger}(s)}.
\]
Thus, it suffices to prove $Q_{p}^{\ddagger}(s) + V_p(s)\not=0$ in a
left half plane, since $X_{p}^{\ddagger}(s)$ is a finite product of
zeta functions.

>From the proof of Proposition \ref{prop_0606}, we have
\begin{equation*}
\frac{X_{p,w}(s)}{X_{p}^{\ddagger }(s)}
= \prod_{k=1}^{k_p} \prod_{m \in \Lambda_w(k)}
\frac{\xi(ks + h_{m,w}(k))}{\xi(ks + h_{m}(k) + 1)}.
\end{equation*}
Using the functional equation $\xi(s)=\xi(1-s)$, we have
\[
\aligned
& \frac{\xi(ks + a)}{\xi(ks + b + 1)}
 = \frac{\xi((1-ks) - a)}{\xi((1-ks) - b - 1)} \\
& \qquad = \pi^{\frac{a-b-1}{2}}
\frac{(ks + a)(ks + a - 1)}{(ks + b)(ks + b + 1)}
\frac{\Gamma((1-ks)/2 - a/2)}{\Gamma((1-ks)/2 - (b + 1)/2)}
\frac{\zeta((1-ks) - a)}{\zeta((1-ks) - b - 1)}.
\endaligned
\]
Because of $\zeta(s)^{-1}=\sum_{n=1}^{\infty}\mu(n)n^{-s}$ as $\sigma \to \infty$, we have
\[
\frac{\zeta((1-ks) - a)}{\zeta((1-ks) - b - 1)} = 1 + \sum_{n=2}^{\infty}\frac{{\frak C}_n(k;a,b)}{n^{-s}}.
\]
Using the Stirling formula
\[
\Gamma(z)=\sqrt{2\pi}\,z^{z-\frac{1}{2}}\,e^{-z}
\left(
1 + \frac{1}{12z}+\frac{1}{288z^2}+\cdots+\frac{1}{c_n z^n} +O\bigl(|z|^{-n-1}\bigr)
\right)
\]
for $|\arg z|<\pi-\epsilon$, we have
\[
\frac{\Gamma(z+\lambda)}{\Gamma(z)}=z^\lambda \Bigl(1+\frac{a_1(\lambda)}{z}+\cdots + \frac{a_n(\lambda )}{z^n} + O\bigl(|z|^{-n-1}\bigr) \Bigr)
\]
as $|z| \to \infty$ with $|\arg z|< \pi-\epsilon$, where $a_j(x) \in
{\R}[x]$ are polynomials of $x$ and the implied constant depends on
$\lambda$ and $\epsilon$. Using the above facts, we obtain
\[
\aligned
\frac{\xi(ks + a)}{\xi(ks + b + 1)}
&= (2\pi/k)^{\frac{a-b-1}{2}}
(-s)^{-\frac{a-b-1}{2}}
\Bigl(1 + \frac{c_1(a,b)}{s}+\cdots + \frac{c_n(a,b)}{s^n} + O\bigl(|s|^{-n-1}\bigr) \Bigr)  \\
& \quad \times \left(1 + \sum_{\nu=2}^{\infty}\frac{{\frak C}_\nu(k;a,b)}{\nu^{-s}}\right)
\endaligned
\]
as $\sigma \to -\infty$ for any fixed $n \geq 0$, where $c_j(a,b)$
are real numbers depending on real numbers $a$, $b$. Applying the
result to each term of $X_{p,w}(s)/X_{p}^{\ddagger}(s)$, we obtain
\begin{equation*}\label{}
\frac{X_{p,w}(s)}{X_{p}^{\ddagger }(s)}
= K_{p,w}(-s)^{\frac{1}{2}A_{p,w}(k)}
\Bigl(1 + \frac{c_1(w)}{s}+\cdots + \frac{c_n(w)}{s^n} + O\bigl(|s|^{-n-1}\bigr) \Bigr)
\left(1 + \sum_{\nu=2}^{\infty}\frac{{\frak C}_\nu(w)}{\nu^{-s}}\right)
\end{equation*}
for some real numbers $c_1(w), \cdots, c_n(w)$, where
\begin{equation}\label{1101}
A_{p,w} = \sum_{k=1}^{k_p} \sum_{m \in \Lambda_w(k)}(h_{m}(k)-h_{m,w}(k)+1),
\end{equation}
\[
K_{p,w}=(2\pi)^{-\frac{1}{2}A_{p,w}} \exp\left( \frac{1}{2} \sum_{k=1}^{k_p} \log k {\sum_{m \in \Lambda_w(k)}(h_{m}(k)-h_{m,w}(k)+1)} \right)~(>0)
\]
Note that $A_{p,w}$ are positive integers, since
$h_{m}(k)-h_{m,w}(k) \geq 0$ for every $1 \leq k \leq k_p$ and $m
\in \Lambda_w(k)$ by definition. Therefore, for sufficiently large
$n$, we have
\[
\aligned
V_p(s)
& = \sum_{w \in ({\frak W}_{p}^{+} \cup {\frak W}_{p}^{0}) \setminus {\frak W}_{p}^{\ddagger}}
K_{p,w}^\prime (-s)^{\frac{1}{2}A_{p,w}} Q_{p,w}(s)  \\
& \hspace{90pt} \times
\Bigl(1 + \frac{c_1(w)}{s}+\cdots + \frac{c_n(w)}{s^n} + O\bigl(|s|^{-n-1}\bigr) \Bigr)
\left(1 + \sum_{\nu=2}^{\infty}\frac{{\frak C}_\nu(w)}{\nu^{-s}}\right)
\endaligned
\]
as $\sigma \to -\infty$,
where $K_{p,w}^\prime=K_{p,w}$ for $w \in {\frak W}_{p}^{+} \setminus {\frak W}_{p}^{\ddagger}$
and $K_{p,w}^\prime=K_{p,w}/2$ for $w \in {\frak W}_{p}^{0}$.

We write $Q_{p}^{\ddagger}(s) + V_p(s)$ as
\[
\,Q_{p}^{\ddagger}(s) + V_p(s) = \sum_{\mu=1}^\infty \frac{1}{\mu^{-s}}{\mathcal Q}_\mu(s^{1/2})
\left(1 + \sum_{k=1}^{n} \frac{c_k(\mu)}{s^{k/2}} + O\bigl(|s|^{-(n+1)/2}) \right),
\]
where ${\mathcal Q}_\mu(s)$ $(m=1,2,3,\cdots)$ are polynomials
with $M=:\underset{\mu \geq 1}{\rm max} \, \deg {\mathcal Q}_\mu < \infty$.
Let $\mu_0$ be the smallest positive integer such that $\deg {\mathcal Q}_{m_0}=M$.
Then
\[
\aligned
\,Q_{p}^{\ddagger}(s) + V_p(s)
&=  \frac{1}{\mu_0^{-s}} Q_{\mu_0}(s^{1/2}) \left[ \left(1 + \sum_{k=1}^{n} \frac{c_k(\mu_0)}{s^{k/2}} + O\bigl(|s|^{-(n+1)/2}) \right)\right.\\
& + \sum_{\mu=1}^{\mu_0-1} \frac{1}{(\mu/\mu_0)^{-s}}\frac{{\mathcal Q}_\mu(s^{1/2})}{{\mathcal Q}_{\mu_0}(s^{1/2})}
\left(1 + \sum_{k=1}^{n} \frac{c_k(\mu)}{s^{k/2}} + O\bigl(|s|^{-(n+1)/2}) \right) \\
& \quad
\left. + \sum_{\mu=\mu_0+1}^{\infty} \frac{1}{(\mu/\mu_0)^{-s}}\frac{{\mathcal Q}_\mu(s^{1/2})}{{\mathcal Q}_{\mu_0}(s^{1/2})}
\left(1 + \sum_{k=1}^{n} \frac{c_k(\mu)}{s^{k/2}} + O\bigl(|s|^{-(n+1)/2}) \right) \right].
\endaligned
\]

We can take $\sigma_0>0$ such that
\[
\sum_{\mu=\mu_0+1}^{\infty} \frac{1}{(\mu/\mu_0)^{-\sigma}}\left| \frac{{\mathcal Q}_\mu(s^{1/2})}{{\mathcal Q}_{\mu_0}(s^{1/2})} \right|
\left|1 + \sum_{k=1}^{n} \frac{c_k(\mu)}{s^{k/2}} + O\bigl(|s|^{-(n+1)/2})\right| < \frac{1}{2}.
\]
holds for $\Re(s) < -\sigma_0$ and $|\Im(s)|>1$.
On the other hand
\[
\sum_{\mu=1}^{\mu_0-1} \frac{1}{(\mu/\mu_0)^{-\sigma}} \left| \frac{{\mathcal Q}_\mu(s^{1/2})}{{\mathcal Q}_{\mu_0}(s^{1/2})} \right|
\left| 1 + \sum_{k=1}^{n} \frac{c_k(\mu)}{s^{k/2}} + O\bigl(|s|^{-(n+1)/2})\right| = O(\mu_0^{-\sigma}|s|^{-1/2})
\]
as $|\Im(s)| \to \infty$ in $\Re(s) < -\sigma_0$  by the choice of $\mu_0$. Hence we have
\[
\aligned
\,Q_{p}^{\ddagger}(s) + V_p(s)
&=  \frac{{\mathcal Q}_{\mu_0}(s^{1/2}) }{\mu_0^{-s}} \left[ 1  + \Theta(1/2) +  O(\mu_0^{-\sigma}|s|^{-1/2})  \right]
\endaligned
\]
for $\Re(s) < -\sigma_0$ and $|\Im(s)|>R$, where $R$ is a large positive number and $\Theta(1/2)$ means a function
whose absolute value is bounded by $1/2$.
Therefore
\[
\varepsilon_p(s) = \frac{X_p^\ddagger(s)}{D_p(s)} \frac{1}{R_p(s)}
\Bigl( Q_{p}^{\ddagger}(s) + V_p(s) \Bigr)
 = \frac{X_p^\ddagger(s)}{D_p(s)} \frac{Q_{p}^{\ddagger}(s)}{R_p(s)}
\frac{{\mathcal Q}_{\mu_0}(s^{1/2}) }{Q_{p}^{\ddagger}(s) } \mu_0^{s}\left[ 1  +  g(s)  \right],
\]
and $|g(s)|<1$ as $|s| \to \infty$ with $\sigma < \kappa \log(|t|+10)$.
This formula implies Proposition \ref{prop_0609} by making $\kappa>0$ large if necessary.
\hfill $\Box$

%
%
\section{ Proof of Proposition \ref{prop_0701}} \label{section_12}
%
%
We need the following lemma.

\begin{lemma}\label{lem_1201} Let $\sigma_0>0$, $T>10$. Let $0 \leq \alpha<\beta<\sigma_{0}$, where $\alpha$ and $\beta$ can be depending on $T$.
Let $f(s)$ be an
analytic function, real for real, regular for
$\sigma\geqslant\alpha$, except at finitely many poles on the real
line; let
\[
|\Re f(\sigma+it)|\geqslant m>0
\]
and
\[|f(\sigma'+it')|\leqslant
M_{\sigma,t} \quad
(\sigma' \geq \sigma,~1\leqslant t' \leq t).
\]
Then, if $T$ is not the ordinate of a zero of $f(s)$
\[
|\arg
f(\sigma+iT)|\leqslant\frac{\pi}{\log(\sigma_0-\alpha)/(\sigma_0-\beta)}\left(\log
M_{\alpha,T+2}+\log\frac{1}{m}\right)+\frac{3}{2}\pi
\]
for $\sigma\geqslant\beta$.
\end{lemma}
\begin{proof}
We readily prove this by the similar method as in \cite[p. 213]{MR882550}.
\end{proof}

\begin{proof}[Proof of Proposition 7.1] It suffices to show the
proposition for large $T$, since we shall make $T$ large if necessary. 
>From the proof of Proposition \ref{prop_0609}, we have
\begin{equation*}
Q_p^{\ddag}(s)+V_p(s)={\mathcal Q}(s^{1/2})\, \mu_0^s \, (1+g(s))
\end{equation*}
for some positive integer $\mu_0$ and a polynomial ${\mathcal Q}(s)={\mathcal Q}_{\mu_0}(s)$ such that
\begin{equation}\label{21}
|g(-\sigma_{\rm L}+it)|<1 \quad \text{as} \quad |t|\to\infty
\end{equation}
for some fixed $\sigma_{\rm L}>0$, and
\begin{equation}\label{23}
|g(s)|<1 \quad \text{as} \quad |s|\to\infty \quad \text{with} \quad \Re(s) < -\kappa \log(|\Im(s)|+10).
\end{equation}
Thus, from this, we have
\begin{equation}
\varepsilon_p(s)
=\frac{X_p^{\ddag}(s)}{D_p(s)}\frac{Q_{p}^\ddagger(s)}{R_p(s)} \cdot \frac{{\mathcal Q}(s^{1/2})}{Q_{p}^\ddagger(s)} \, \mu_0^s \, (1+g(s))
=\frac{X_p^{\ddag}(s)}{D_p(s)}\frac{Q_{p}^\ddagger(s)}{R_p(s)} \cdot (1+r_p(s))
\end{equation}
where $r_p(s)$ is of \eqref{0602_1}. Recall the bound \eqref{0602} for $r_p(s)$.
Note that $Q_p^{\ddagger}(s)/R_p(s)$ is a polynomial by Definition \ref{def_0509} of $R_p(s)$.
\medskip

For $N(T;\sigma_{\rm L})$, we consider the rectangle $R_T$ with vertices at $-\sigma_{\rm L}+ci$,
$\sigma_{\rm R}+ci$, $\sigma_{\rm R}+iT$, $-\sigma_{\rm L}+iT$,
where $\sigma_{\rm R}$ is a positive real number
and $c>0$ is a positive constant.
We take $\sigma_{\rm R}$ and $c$ large positive real numbers
such that $\varepsilon_p(s)$ has no zeros on the line from  $-\sigma_{\rm L}+ci$ to
$\sigma_{\rm R}+ci$ and $|g(-\sigma_{\rm L}+it)|,~|r_p(\sigma_{\rm R}+it)|<1$ for $|t|\geqslant c$.
Also, we can assume that $\varepsilon_P(s)$ has no zeros on the rectangle.

We apply the standard method of the counting of zeros to the above rectangle with \eqref{0602} and \eqref{21}
(see \cite[p. 212]{MR882550}). We have
\begin{equation*}
\begin{split}
N(T;\sigma_{\rm L})=&\frac{1}{2\pi i}\int_{R_T} d\log \varepsilon_p(s) \\
=&\frac{1}{2\pi}\Delta_1\arg \varepsilon_p(s)+\frac{1}{2\pi}\Delta_2\arg
\varepsilon_p(s)+\frac{1}{2\pi}\Delta_3\arg \varepsilon_p(s)+\frac{1}{2\pi}\Delta_4\arg
\varepsilon_p(s),
\end{split}
\end{equation*}
where $\Delta_1,\,\Delta_2,\,\Delta_3,\,\Delta_4$ denote the
variations from $\sigma_{\rm R}+ci$ to $\sigma_{\rm R}+iT$, from $\sigma_{\rm R}+iT$
to $-\sigma_{\rm L}+iT$, from $-\sigma_{\rm L}+iT$ to $-\sigma_{\rm L}+ic$,
from $-\sigma_{\rm L}+ic$ to $\sigma_{\rm R}+ci$, respectively.
Clearly, $\Delta_4\arg \varepsilon_p(s)=O(1)$. Using  \eqref{0602} and \eqref{21}, we can readily
compute $\Delta_1\arg \varepsilon_p(s)$ and $\Delta_3\arg \varepsilon_p(s)$.
For $\Delta_2\arg \varepsilon_p(s)$, referring to \eqref{0903}, we define
\[
\aligned
\Gamma^*(s)
&= \prod_{k=1}^\infty \prod_{h> (kc_p+1)/2} \gamma(ks+h)^{N_p(k,h-1)-N_p(k,h)}; \\
L(s) &= \frac{Q_p^{\ddagger}(s)}{R_p(s)}\prod_{k=1}^\infty \prod_{h> (kc_p+1)/2} \zeta(ks+h)^{N_p(k,h-1)-N_p(k,h)};\\
L^*(s)& =L(s)\bigl(1  + r_p(s)\bigr); \\
\endaligned
\]
where $\gamma(s)=s(s-1)\pi^{-s/2}\Gamma(s/2)$.
Then
\[
\varepsilon_p(s)=\Gamma^*(s)L^*(s)
\]
by definition, and
\[
\Delta_2\arg \varepsilon_p(s)=\Delta_2\arg \Gamma^*(s)+\Delta_2\arg L^*(s).
\]
The valuation $\Delta_2\arg \Gamma^*(s)$ is computed easily by the Stirling formula.
On the other hand, we have
\begin{equation}\label{25}
\aligned
L(s)r_p(s)
& = \sum_{w \in {\frak W}_{p}^{+} \cup {\frak W}_{p}^{0} \setminus {\frak W}_{p}^{\ddagger}}
\frac{\tilde{Q}_{p,w}(s)}{R_p(s)} \prod_{\alpha \in (\Phi^+ \setminus \Phi_p^+) \cap w^{-1}\Phi^-}
\frac{\gamma(\langle \lambda_p, \alpha^\vee \rangle s + \h\alpha^\vee)}
{\gamma(\langle \lambda_p, \alpha^\vee \rangle s + \h\alpha^\vee + 1)}\\
& \qquad \times
\prod_{\alpha \in (\Phi^+ \setminus \Phi_p^+) \cap w^{-1}\Phi^-}
\zeta(\langle \lambda_p, \alpha^\vee \rangle s + \h\alpha^\vee),
\endaligned
\end{equation}
where $\tilde{Q}_{p,w}(s)/R_p(s)$ are polynomials by definition of $R_p(s)$.
Therefore we see that
\[L^*(s) \ll T^M \qquad (-\sigma_{\rm L} \leq \Re(s) \leq \sigma_{\rm R})
\]
for some positive $M$ if $T$ is sufficiently large (\cite[Chap.V]{MR882550}).
Hence we have
\[
\Delta_2\arg L^*(s)=O(\log T).
\]
by Lemma \ref{lem_1201}.
Thus, we get the formula for $N(T;\sigma_{\rm L})$.

For $N(T;+\infty)$, it suffices to consider $N(T;-\kappa \log T)$ by Proposition \ref{prop_0609}.
We form the rectangle $-\kappa\log T+c_1i$, $\sigma_{\rm R}+c_1i$,
$\sigma_{\rm R}+iT$, $-\kappa\log T+iT$, where $c_1>0$ is taken such that $\varepsilon_p(s)$ has no
zeros on the boundary of this rectangle. We follow the method as
above. We similarly have
\[N(T;\infty)=\Delta_1^*\arg \varepsilon_p(s)+\Delta_2^*\arg
\varepsilon_p(s)+\Delta_3^*\arg \varepsilon_p(s)+\Delta_4^*\arg \varepsilon_p(s)
\]
and $\Delta_2^*\arg \varepsilon_p(s)=\Delta_2^*\arg \Gamma^*(s)+\Delta_2^*\arg
L^*(s) $, where where
$\Delta_1^*,\,\Delta_2^*,\,\Delta_3^*,\,\Delta_4^*$ denote the
variations from $\sigma_{\rm R}+c_1i$ to $\sigma_{\rm R}+iT$, from $\sigma_{\rm R}+iT$
to $-\kappa\log T+iT$, from $-\kappa\log T+iT$ to $-\kappa\log
T+ic_1$, from $-\sigma_{\rm L}+ic_1$ to $\sigma_{\rm R}+ic_1$, respectively.
Using \eqref{0602} and \eqref{23}, we can readily compute $\Delta_1^*\arg
\varepsilon_p(s)$, $\Delta_3^*\arg \varepsilon_p(s)$ and $\Delta_4^*\arg\Gamma^*(s)$.
We see that
\[
L^*(s) \ll T^{\kappa^* \log T}\qquad (\sigma=-\kappa\log T \leq \Re(s) \leq \sigma_{\rm R}),
\]
by \eqref{25} (and \cite[Chap.V]{MR882550}), where $\kappa^*>\kappa$ is a constant depending on $\kappa$.
Thus we obtain
\[
\Delta_2^*\arg L^*(s)=O(\log^2 T)
\]
by applying Lemma \ref{lem_1201}
to $\alpha=-2\kappa\log T$, $\beta=-\kappa\log T$
and $\sigma_0=\sigma_{\rm R}$.
It is not hard to see that $\Delta_4^*\arg L^*(s)=O(\log T)$ by Lemma \ref{lem_1201}.
Thus we obtain the formula for $N(T;+\infty)$.
\end{proof}
%
%
\section{ Proof of Proposition \ref{prop_0702}} \label{section_13}
%
%
The function $\varepsilon_p(s)$ is a linear combination of products
of several zeta functions $\xi(ks+h)$ with coefficients of rational
functions. Therefore $W_p(s)$ is an entire function of the order at
most one, since $\xi(s)$ is an entire function of order one (and
maximal type). Note that if $\rho$ is a zero of $W_p(z)$,
$-\bar{\rho}$ is so, since $\varepsilon_p(s)$ is real for real $s$.
Thus, Propositions \ref{prop_0608} and \ref{prop_0609} imply the
product factorization of $W_p(s)$ of the form
\[
W_p(z) = \omega \, e^{\beta z} \, V(z) \, w_1(z)w_2(z)
\]
with
\[
\aligned
w_1(z)&=\prod_{n=1}^{\infty}
\left[\left(1-\frac{z}{\rho_n}\right)\left(1+\frac{z}{\bar{\rho}_n}\right) \right]
\exp\left( \frac{z}{\rho_n} - \frac{z}{\bar{\rho}_n} \right), \\
w_2(z)&=\prod_{n=1}^{\infty}
\left[\left(1-\frac{z}{\eta_n}\right)\left(1+\frac{z}{\bar{\eta}_n}\right) \right]
\exp\left( \frac{z}{\eta_n} - \frac{z}{\bar{\eta}_n} \right),
\endaligned
\]
where $\omega$ is a nonzero real number, $\beta$ is a complex
number, $V(z)$ is a nonzero polynomial having no zeros in $\Im(z)>0$ except for purely imaginary zeros, $\Re(\rho_n) > 0$,
$\Re(\eta_n) > 0$, $\Im(\eta_n) \geq 0$, and $0<\delta(t) <
\Im(\rho_n) <\sigma_L+1< \Im(\eta_n)<\kappa \log(\Re(\eta_n)+10)$
for every $n\geq 1$. Here $\delta(t)$ is the function of Proposition
\ref{prop_0608}, and $\kappa$ is the positive number of Proposition
\ref{prop_0609}, and $\sigma_L$ is the positive number of
Proposition \ref{prop_0701}. The products of the right-hand sides
converge uniformly on every compact subset in $\C$. Write $\rho_n =
a_n + i b_n$ ($b_n >0$), and $\eta_n= c_n + id_n$ ($d_n >0$). Then
\[
\sum_{n=1}^\infty \left| \frac{1}{\rho_n} - \frac{1}{\bar{\rho}_n} \right|
\leq  2 \, \sum_{n=1}^\infty \frac{b_n}{|\rho_n|^2} \leq 2(\sigma_L+1) \, \sum_{n=1}^\infty \frac{1}{|\rho_n|^2},
\]
and
\[
\sum_{n=1}^\infty \left| \frac{1}{\eta_n} - \frac{1}{\bar{\eta}_n} \right|
\leq  2 \, \sum_{n=1}^\infty \frac{d_n}{|\eta_n|^2} \leq 2\kappa \, \sum_{n=1}^\infty \frac{\log(c_n+10)}{|\eta_n|^2}
\ll_\epsilon  \sum_{n=1}^\infty \frac{1}{|\eta_n|^{2-\epsilon}} .
\]
Because the sum on the right-hand side is finite, we can take
factors $\sum_{n=1}^\infty \exp\left( \frac{z}{\rho_n} -
\frac{z}{\bar{\rho}_n} \right)$ and $\sum_{n=1}^\infty \exp\left(
\frac{z}{\eta_n} - \frac{z}{\bar{\eta}_n} \right)$ out of the
infinite product. Hence we obtain the desired product formula except
for the requirement for the exponent
\[
\alpha := \beta
+ \sum_{n=1}^\infty \exp\left( \frac{1}{\rho_n} - \frac{1}{\bar{\rho}_n} \right)
+ \sum_{n=1}^\infty \exp\left( \frac{1}{\eta_n} - \frac{1}{\bar{\eta}_n} \right)
\]
of the new exponential factor.

Now we prove $\Im(\alpha)=0$ to complete the proof of Proposition \ref{prop_0702}.
We have
\[
\aligned
\log\left|
\frac{W_p(-iy)}{W_p(iy)}
\right|
& = 2 y \cdot \Im(\alpha) +
\log\left| \frac{V(-iy)}{V(iy)} \right|
+
\sum_{n=1}^{\infty}
\log \left| \frac{iy+\rho_n}{iy-\rho_n} \right|^2
+
\sum_{n=1}^{\infty}
\log \left| \frac{iy+\eta_n}{iy-\eta_n} \right|^2\\
& = 2 y \cdot \Im(\alpha)
+ O(\log y) \quad (y \to +\infty).
\endaligned
\]
In fact $\log\left| V(-iy)/V(iy) \right| = o(1)$,
\[
\sum_{n=1}^{\infty}
\log \frac{a_n^2 + (y + b_n)^2}{a_n^2 +  (y - b_n)^2}
\leq
\sum_{n=1}^{\infty} \frac{ 4 y b_n}{a_n^2 +  (y - b_n)^2}
= O\left( \sum_{\rho=1}^\infty \frac{y \log \rho}{\rho^2+y^2} \right)=O(\log y)
\]
as $y \to +\infty$ by $0<b_n<\sigma_L+1$ and Proposition \ref{prop_0701},
and
\[
\sum_{n=1}^{\infty}
\log \frac{c_n^2 + (y + d_n)^2}{c_n^2 +  (y - d_n)^2}
\leq
\sum_{n=1}^{\infty} \frac{ 4 y d_n}{c_n^2 +  (y - d_n)^2}
= O\left( \sum_{\rho=1}^\infty \frac{y \log \rho}{(\rho^2+y^2)^{1-\epsilon}} \right)=O(\log y)
\]
as $y \to +\infty$ by $0<d_n<\kappa \log(c_n+10)$ and Proposition
\ref{prop_0701}. Hence, in order to prove $\Im(\alpha)=0$, it
suffices to show
\[
\frac{\varepsilon_p(-c_p/2+y)}{\varepsilon_p(-c_p/2-y)} = A y^m(1+o(1))
\quad (y \to +\infty)
\]
for some $A \not=0$ and $m \geq 0$. Because of $D_p(-c_p-s)=D_p(s)$ and $R_p(-c_p-s)=R_p(s)$,
\[
\frac{\varepsilon_p(-c_p/2+y)}{\varepsilon_p(-c_p/2-y)}
=\frac{E_p(-c_p/2+y)}{E(-c_p/2-y)}\frac{R_p(-c_p -y)D_p(-c_p/2-y)}{R_p(-c_p+y)D_p(-c_p/2+y)}
=\frac{E_p(-c_p/2+y)}{E(-c_p/2-y)}.
\]
We need to show
\[
\frac{E_p(-c_p/2-y)}{E_p(-c_p/2+y)}
= \frac{\sum_{w} Q_{p,w}^\prime(-c_p/2-y)X_{p,w}(-c_p/2-y)}
{\sum_{w} Q_{p,w}^\prime(-c_p/2+y)X_{p,w}(-c_p/2+y)}
= A^{-1} y^{-m}(1+o(1))
\]
as $y \to +\infty$ for some $A \not=0$ and $m \geq 0$,
where the summation $\sum_w$ taken over all $w \in {\frak W}_{p}^{+} \cup {\frak W}_p^0$, $Q_{p,w}^\prime(s)=Q_{p,w}(s)$ for $w \in {\frak W}_p^+$
and $Q_{p,w}^\prime(s)=Q_{p,w}(s)/2$ for $w \in {\frak W}_p^0$.

By Propositions \ref{prop_0605} and \ref{prop_0606}, we have
\[
E_p(-c_p/2+y)=Q_{p}^{\ddagger}(-c_p/2+y)X_{p}^{\ddagger}(-c_p/2+y)(1+o(1))
\quad (y \to +\infty).
\]
Therefore it suffices to show that
\begin{equation}\label{1201}
\aligned
\frac{E_p(-c_p/2-y)}{Q_{p}^{\ddagger}(-c_p/2+y)X_{p}^{\ddagger}(-c_p/2+y)}
&= \sum_{w}
\frac{ Q_{p,w}^\prime(-c_p/2-y)}{Q_{p}^{\ddagger}(-c_p/2+y)}
\frac{ X_{p,w}(-c_p/2-y)}{ X_{p}^{\ddagger}(-c_p/2+y)} \\
&= A^{-1} y^{-m}(1+o(1))
\endaligned
\end{equation}
as $y \to +\infty$ for some $A \not=0$ and $m>0$. We have
\[
\frac{ Q_{p,w}^\prime(-c_p/2-y)}{Q_{p}^{\ddagger}(-c_p/2+y)}
= \delta_{w}^\prime + \frac{a_1(w)}{y} + \cdots + \frac{a_n(w)}{y^n} + O(y^{-n-1})
\quad (y \to +\infty),
\]
for some real numbers $a_1(w), \cdots, c_n(w)$, where $\delta_w=1$
if $w \in {\frak W}_p^\ddagger$ and $|w^{-1}\Delta \setminus \Phi_p|
=1$, and $\delta_w=0$ otherwise. We write
\[
\frac{ X_{p,w}(-c_p/2-y)}{ X_{p}^{\ddagger}(-c_p/2+y)}
 =
\frac{ X_{p,w}(-c_p/2+y)}{ X_{p}^{\ddagger}(-c_p/2+y)}
\frac{ X_{p,w}(-c_p/2-y)}{ X_{p,w}(-c_p/2+y)}.
\]
Here we find that
\[
\frac{ X_{p,w}(s)}{ X_{p}^{\ddagger}(s)}
= B_{p,w} \, s^{-\frac{A_{p,w}}{2}}
\Bigl(1 + \frac{b_1(w)}{s} + \cdots + \frac{b_n(w)}{s^n} + O\bigl(|s|^{-n-1}\bigr) \Bigr)
(1+O(2^{-\sigma}))
\]
as $\sigma \to +\infty $ for some real numbers $B_{p,w}>0$, $b_1(w), \cdots, b_n(w)$, and positive integers $A_{p,w}$ of \eqref{1101}
by a way similar to the proof of Proposition \ref{prop_0609}.
On the other hand,
\[
\aligned
\, & \frac{\xi(k(-c_p-s) + h + \delta)}
{\xi(ks + h + \delta)}
 = \frac{\xi(ks + kc_p - h -\delta +1)}
{\xi(ks + h + \delta)}  \\
&  =
(2\pi/k)^{-\frac{kc_p-2h-2\delta+1}{2}}
s^{\frac{kc_p-2h-2\delta+1}{2}}
\Bigl( 1 + \frac{c_1(k,h)}{s} + \cdots + \frac{c_n(k,h)}{s^n} + O\bigl(|s|^{-n-1}\bigr) \Bigr)
(1+O(2^{-\sigma}))
\endaligned
\]
as $\sigma \to +\infty $.  Therefore we obtain
\[
\aligned
\, & \frac{ X_{p,w}(-c_p-s)}{ X_{p,w}(s)}
= \prod_{\alpha \in \Phi^+ \setminus \Phi_p^+}
\frac{\xi(\langle \lambda_p, \alpha^\vee \rangle(-c_p-s) + + \h\alpha^\vee + \delta_{\alpha,w})}
{\xi(\langle \lambda_p, \alpha^\vee \rangle s + \h\alpha^\vee + \delta_{\alpha,w})}
\\
& = \prod_{k=1}^{k_p} \prod_{h=1}^{kc_p + \hbar_k^+}
\left( \frac{\xi(k(-c_p-s) + h)}{\xi(ks + h)} \right)^{N_{p,w}(k,h)}
\left( \frac{\xi(k(-c_p-s) + h + 1)}{\xi(ks + h + 1)} \right)^{N_p(k,h)-N_{p,w}(k,h)}
\\
& = B_{p,w}^\prime \,
s^{- \frac{A_{p,w}^\prime}{2}}
\Bigl( 1 + \frac{b_1^\prime(w)}{s} + \cdots + \frac{b_n^\prime(w)}{s^n} + O\bigl(|s|^{-n-1}\bigr)\Bigr)
(1+O(2^{-\sigma}))
\endaligned
\]
as $\sigma \to +\infty $ for some real numbers $B_{p,w}^\prime >0$
and integers $A_{p,w}^\prime$ given by
\[
- A_{p,w}^\prime
 = \sum_{k=1}^{k_p} \sum_{h=k + \hbar_k^- }^{k + \hbar_k^+}
\Bigl(
N_{p,w}(k,h)(kc_p-2h+1) + (N_{p}(k,h)-N_{p,w}(k,h))(kc_p-2h-1)
\Bigr).
\]
Now we calculate $A_{p,w}^\prime$. We have
\[
\aligned
- A_{p,w}^\prime
& = \sum_{k=1}^{k_p} \sum_{h=k + \hbar_k^- }^{k + \hbar_k^+}
\Bigl( N_{p}(k,h)(kc_p-2h) + 2N_{p,w}(k,h) - N_{p}(k,h) \Bigr) \\
& =
\sum_{k=1}^{k_p} \sum_{h=k + \hbar_k^- }^{k + \hbar_k^+} N_{p}(k,h)(kc_p-2h)
- \sum_{k=1}^{k_p} \sum_{h=k + \hbar_k^-}^{kc_p + \hbar_k^+} \Bigl( N_{p}(k,h) - 2N_{p,w}(k,h)  \Bigr).
\endaligned
\]
The first sum on the right-hand side is zero, since
$N_{p}(k,h)=N_{p}(k,kc_p-h)$ and $\hbar_k^-+\hbar_k^+=k(c_p-2)$ by
Lemmas \ref{lem_0803} and \ref{lem_0814}. Therefore, we get
\[
\aligned
 A_{p,w}^\prime
& = \sum_{k=1}^{k_p} \sum_{h=k + \hbar_k^-}^{kc_p + \hbar_k^+} \Bigl( N_{p}(k,h) - 2N_{p,w}(k,h) \Bigr)  \\
& = \sum_{k=1}^{k_p} \sum_{h=k + \hbar_k^-}^{kc_p + \hbar_k^+} 2 \Bigl( N_{p}(k,h) - N_{p,w}(k,h) \Bigr)
- \sum_{k=1}^{k_p} \sum_{h=k + \hbar_k^-}^{kc_p + \hbar_k^+}N_{p}(k,h) \\
& = 2 \left(|\Phi^+ \setminus \Phi_p^+| - l_p(w) \right) - |\Phi^+ \setminus \Phi_p^+|
=  |\Phi^+ \setminus \Phi_p^+| - 2\,l_p(w) .
\endaligned
\]
Hence $A_{p,w}^\prime>0$ for $w \in {\frak W}_{p}^+$
and $A_{p,w}^\prime=0$ for $w \in {\frak W}_p^0$.

Combining the above facts, we obtain
\[
\aligned
\, & \frac{E_p(-c_p/2-y)}{Q_{p}^{\ddagger}(-c_p/2+y)X_{p}^{\ddagger}(-c_p/2+y)} \\
& \qquad =  \sum_{w \in {\frak W}_p^+ \cup {\frak W}_p^0}
B_{p,w}^\ast \, y^{- \frac{A_{p,w}+A_{p,w}^\prime}{2}} \,
\Bigl( \delta_w^\prime  + \frac{b_1^\ast(w)}{y} + \cdots + \frac{b_n^\ast(w)}{y^n} + O\bigl(y^{-n-1}\bigr) \Bigr)(1+O(2^{-y}))
\endaligned
\]
as $y \to + \infty$ for some positive integers $B_{p,w}^\ast$ and
real numbers $b_1^\ast(w), \cdots, b_n^\ast(w)$. Now we obtain
\eqref{1201}, and complete the proof of Proposition \ref{prop_0702}.
\hfill $\Box$

%
%
\section{ Proof of Theorem \ref{thm_0704}} \label{section_14}
%
%
In order to prove that all but finitely many zeros of $\xi_p(s)$ are simple,
it suffices to show that $\theta(t)=\arg \varepsilon_p(-c_p/2+it)$ is strictly
increasing as $t \to +\infty$,
since we already know that all but finitely many zeros of $\xi_p(s)$
lie on the line $\Re(s)=-c_p/2$, and
\[
\xi_p(-c_p/2+it)=|\varepsilon_p(-c_p/2+it)|\left( e^{i \theta(t)} \pm e^{-i\theta(t)} \right)
\]
holds for a suitable choice of sign.
We have
\[
E_p(s) = Q_p^\ddagger (s) X_p^{\ddagger}(s)(1 + o(1))
\]
as $|s| \to \infty$ on $\Re (s) = -c_p/2$ by Propositions
\ref{prop_0605} and \ref{prop_0606}, and hence
\[
\varepsilon_p(s) = \frac{Q_p^\ddagger (s)}{R_p(s)}\frac{X_p^{\ddagger}(s)}{D_p(s)}(1 + o(1))
\]
as $|s| \to \infty$ on $\Re (s) = -c_p/2$. Note that $Q_p^\ddagger
(s)/R_p(s)$ and $X_p^{\ddagger}(s)/D_p(s)$ are entire by Definitions
\ref{def_0509} and \ref{def_0510}.
Using the above asymptotic formula and \eqref{0903}
of $X_p^{\ddagger}(s)/D_p(s)$, we can see that
\[
\frac{d \arg \varepsilon_p(-c_p/2+it)}{dt}
= \kappa \log t \, (1 + o(1)) \quad (t \to +\infty)
\]
for some positive constant $\kappa$.
Here the factor $\log t$ comes
from the gamma factors of $X_p^{\ddagger}(s)/D_p(s)$ (see \eqref{0903}),
and  we need to use the fact that
\[
\frac{\zeta^\prime}{\zeta}(1+it) = O\left( \frac{\log t}{\log \log t} \right) = o(\log t)
\quad (t \to +\infty)
\]
(see Theorem 5.7 of \cite{MR882550}, for example),
if $kc_p$ is even and $N_p(k,h-1)-N_p(k,h)>0$ for some $1 \leq k \leq k_p$ with $h=1+kc_p/2$.
In any case  $\theta(t)=\arg \varepsilon_p(-c_p/2+it)$ is strictly increasing as $t \to \infty$.
Hence we obtain the desired result.
\hfill $\Box$
\medskip

The worst case of the proof could occur.
In fact the zeta function of $A_1$ is
\begin{equation*}\begin{split}
\hat{\zeta}_{1}^{A_1}(s)
= \frac{\hat{\zeta}(s+2)}{s}-\frac{\hat{\zeta}(s+1)}{s+2}.
\end{split}\end{equation*}
In this case, we have the factor $\zeta(1+it)$
in $\varepsilon_1(s)=\xi(s+2)$ on $\Re(s)=-1$ ($c_p=2$).

%
%
\section{Appendix 1: Decomposition of $\Sigma_p(1)$} \label{section_17}
%
%

We use the same numbering as in \cite[p.53]{Kac}.
In the following, we abbreviate $L_j^\vee(1)$ of Lemma \ref{lem_0815} as $L_j^\vee$, and $\alpha^\vee=\sum_{i=1}^r a_i\alpha_
i^\vee$ as the sequence $a_1\cdots a_r$ with $a_p=1$.
We will give explicit forms of the sets $L_j^\vee$ according to the type of the root system $\Phi^\vee$.

\subsection{$A_n$ case}
Due to the symmetry of the root system, it is sufficient to consider the case $1\leq p\leq [n/2+1]$.
Let $\alpha_{i,j}^\vee=\sum_{k=i}^j\alpha_k^\vee$.
Then
for $1\leq j\leq p$,
\begin{equation*}
  L_j^\vee=\{\alpha_{j,n}^\vee,\alpha_{j,n-1}^\vee,\ldots,\alpha_{j,p+j-1}^\vee,\alpha_{j+1,p+j-1}^\vee,\ldots,\alpha_{p,p+j-1}^\vee\}.
\end{equation*} 

For example, in the $A_6$ case, we have
\begin{itemize}
\item $p=1$ case:\\
$L_1^\vee$: 111111 111110 111100 111000 110000 100000
\item $p=2$ case:\\
$L_1^\vee$: 111111 111110 111100 111000 110000 010000 \\
$L_2^\vee$: 011111 011110 011100 011000
\item $p=3$ case:\\
$L_1^\vee$: 111111 111110 111100 111000 011000 001000 \\
$L_2^\vee$: 011111 011110 011100 001100 \\
$L_3^\vee$: 001111 001110
\item $p=4$ case:\\
$L_1^\vee$: 111111 111110 111100 011100 001100 000100 \\
$L_2^\vee$: 011111 011110 001110 000110 \\
$L_3^\vee$: 001111 000111
\item $p=5$ case:\\
$L_1^\vee$: 111111 111110 011110 001110 000110 000010 \\
$L_2^\vee$: 011111 001111 000111 000011
\item $p=6$ case:\\
$L_1^\vee$: 111111 011111 001111 000111 000011 000001
\end{itemize}

\subsection{$B_n$ case}
Let $\alpha_{i,j}^\vee=\sum_{l=i}^j\alpha_l^\vee$ and
$\beta_{i,j,k}^\vee=\sum_{l=i}^j\alpha_l^\vee+2\sum_{l=j+1}^k\alpha_l^\vee$.
Then
for $1\leq j\leq \min\{p,n-p+1\}$,
\begin{equation*}
  L_j^\vee=\{\beta_{j,p,n}^\vee,\beta_{j,p+1,n}^\vee,\ldots, \beta_{j,n-1,n}^\vee,
  \alpha_{j,n}^\vee,  \alpha_{j,n-1}^\vee,  \ldots,  \alpha_{j,p+j-1}^\vee,
  \alpha_{j+1,p+j-1}^\vee,  \ldots,  \alpha_{p,p+j-1}^\vee  \}.
\end{equation*} 
Further if $n-p+1<p\leq n$, then
for $n-p+1<j\leq \min\{p,2n-2p+1\}$,
\begin{equation*}
  L_j^\vee=\{\beta_{j,p,n}^\vee,\beta_{j,p+1,n}^\vee,\ldots,
  \beta_{j,2n-p-j+1,n}^\vee,
  \beta_{j+1,2n-p-j+1,n}^\vee,
  \ldots,
  \beta_{p,2n-p-j+1,n}^\vee\}.
\end{equation*}

For example, in the $B_6$ case, we have
\begin{itemize}
\item $p=1$ case:\\
$L_1^\vee$: 122222 112222 111222 111122 111112 111111 111110 111100 111000 110000 100000
\item $p=2$ case:\\
$L_1^\vee$: 112222 111222 111122 111112 111111 111110 111100 111000 110000 010000 \\
$L_2^\vee$: 012222 011222 011122 011112 011111 011110 011100 011000
\item $p=3$ case:\\
$L_1^\vee$: 111222 111122 111112 111111 111110 111100 111000 011000 001000 \\
$L_2^\vee$: 011222 011122 011112 011111 011110 011100 001100 \\
$L_3^\vee$: 001222 001122 001112 001111 001110
\item $p=4$ case:\\
$L_1^\vee$: 111122 111112 111111 111110 111100 011100 001100 000100 \\
$L_2^\vee$: 011122 011112 011111 011110 001110 000110 \\
$L_3^\vee$: 001122 001112 001111 000111 \\
$L_4^\vee$: 000122 000112
\item $p=5$ case:\\
$L_1^\vee$: 111112 111111 111110 011110 001110 000110 000010 \\
$L_2^\vee$: 011112 011111 001111 000111 000011 \\
$L_3^\vee$: 001112 000112 000012
\item $p=6$ case:\\
$L_1^\vee$: 111111 011111 001111 000111 000011 000001
\end{itemize}

\subsection{$C_n$ case}
Let $\alpha_{i,j}^\vee=\sum_{l=i}^j\alpha_l^\vee$ and
$\beta_{i,j}^\vee=\sum_{l=i}^j\alpha_l^\vee+2\sum_{l=j+1}^{n-1}\alpha_l^\vee+\alpha_n^\vee$.
Then
for $1\leq j\leq \min\{p,n-p+1\}$,
\begin{equation*}
  L_j^\vee=\{\beta_{j,p}^\vee,\beta_{j,p+1}^\vee,\ldots,
  \beta_{j,n-2}^\vee,
  \alpha_{j,n}^\vee,
  \alpha_{j,n-1}^\vee,
  \ldots,
  \alpha_{j,p+j-1}^\vee,
  \alpha_{j+1,p+j-1}^\vee,
  \ldots,
  \alpha_{p,p+j-1}^\vee
  \}.
\end{equation*} 
Further if $n-p+1<p\leq n-1$, then
for $n-p+1<j\leq \min\{p,2n-2p\}$,
\begin{equation*}
  L_j^\vee=\{\beta_{j,p}^\vee,\beta_{j,p+1}^\vee,\ldots,
  \beta_{j,2n-j-p}^\vee,
  \beta_{j+1,2n-j-p}^\vee,
  \ldots,
  \beta_{p,2n-j-p}^\vee
  \}.
\end{equation*} 

In the case $p=n$, for $1\leq j\leq \frac{n+1}{2}$, 
\begin{equation*}
  L_j^\vee=\{\beta_{j,j-1}^\vee,\beta_{j,j}^\vee,\beta_{j,j+1}^\vee,\ldots,\beta_{j,n-j}^\vee,
  \beta_{j+1,n-j}^\vee,
  \ldots,
  \beta_{n-j+1,n-j}^\vee
  \}.
\end{equation*}

For example, in the $C_6$ case, we have
\begin{itemize}
\item $p=1$ case:\\
$L_1^\vee$: 122221 112221 111221 111121 111111 111110 111100 111000 110000 100000
\item $p=2$ case:\\
$L_1^\vee$: 112221 111221 111121 111111 111110 111100 111000 110000 010000 \\
$L_2^\vee$: 012221 011221 011121 011111 011110 011100 011000
\item $p=3$ case:\\
$L_1^\vee$: 111221 111121 111111 111110 111100 111000 011000 001000 \\
$L_2^\vee$: 011221 011121 011111 011110 011100 001100 \\
$L_3^\vee$: 001221 001121 001111 001110
\item $p=4$ case:\\
$L_1^\vee$: 111121 111111 111110 111100 011100 001100 000100 \\
$L_2^\vee$: 011121 011111 011110 001110 000110 \\
$L_3^\vee$: 001121 001111 000111 \\
$L_4^\vee$: 000121
\item $p=5$ case:\\
$L_1^\vee$: 111111 111110 011110 001110 000110 000010 \\
$L_2^\vee$: 011111 001111 000111 000011
\item $p=6$ case:\\
$L_1^\vee$: 222221 122221 112221 111221 111121 111111 011111 001111 000111 000011 000001 \\
$L_2^\vee$: 022221 012221 011221 011121 001121 000121 000021 \\
$L_3^\vee$: 002221 001221 000221
\end{itemize}

\subsection{$D_n$ case}
Let $\alpha_{i,j}^\vee=\sum_{l=i}^j\alpha_l^\vee$,
$\beta_{i,j}^\vee=\sum_{l=i}^j\alpha_l^\vee+2\sum_{l=j+1}^{n-2}\alpha_l^\vee+\alpha_{n-1}^\vee+\alpha_n^\vee$ and
$\gamma_{i}^\vee=\sum_{l=i}^{n-2}\alpha_l^\vee+\alpha_n^\vee$.
Then in the cases $1\leq p\leq n-2$,
for $1\leq j\leq \min\{p,n-p+1\}$,
\begin{equation*}
  L_j^\vee=\{\beta_{j,p}^\vee,\beta_{j,p+1}^\vee,\ldots,
  \beta_{j,n-2}^\vee,
  \alpha_{j,n-1}^\vee,
  \ldots,
  \alpha_{j,p+j-1}^\vee,
  \alpha_{j+1,p+j-1}^\vee,
  \ldots,
  \alpha_{p,p+j-1}^\vee
  \}.
\end{equation*} 
Further if $n-p+1<p\leq n-2$, then
for $n-p+1<j\leq \min\{p,2n-2p-1\}$,
\begin{equation*}
  L_j^\vee=\{\beta_{j,p}^\vee,\beta_{j,p+1}^\vee,\ldots,
  \beta_{j,2n-j-p-1}^\vee,
  \beta_{j+1,2n-j-p-1}^\vee,
  \ldots,
  \beta_{p,2n-j-p-1}^\vee
  \}.
\end{equation*} 
In addition, for $j=\min\{p,2n-2p-1\}+1$,
\begin{equation*}
  L_j^\vee=\{\gamma_{1}^\vee,\gamma_{2}^\vee,\ldots,
  \gamma_{p}^\vee\}.
\end{equation*} 

In the case $p=n-1$,
for $j=1$,
\begin{equation*}
  L_1^\vee=\{\beta_{1,1}^\vee,\beta_{1,2}^\vee,\ldots,\beta_{1,n-2}^\vee,
  \alpha_{1,n-1}^\vee,   \alpha_{2,n-1}^\vee,
  \ldots,
  \alpha_{n-1,n-1}^\vee
  \},
\end{equation*} 
and
for $2\leq j\leq \frac{n}{2}$,
\begin{equation*}
  L_j^\vee=\{\beta_{j,j}^\vee,\beta_{j,j+1}^\vee,\ldots,\beta_{j,n-j}^\vee,
  \beta_{j+1,n-j}^\vee,
  \ldots,
  \beta_{n-j,n-j}^\vee
  \},
\end{equation*} 

In the case $p=n$, for $1\leq j\leq\frac{n}{2}$,
$L_j^\vee$ is the same as in the case $p=n-1$ with the roles of $\alpha_{n-1}$ and $\alpha_{n}$ exchanged.

For example, in the $D_6$ case, we have
\begin{itemize}
\item $p=1$ case:\\
$L_1^\vee$: 122211 112211 111211 111111 111110 111100 111000 110000 100000 \\
$L_2^\vee$: 111101
\item $p=2$ case:\\
$L_1^\vee$: 112211 111211 111111 111110 111100 111000 110000 010000 \\
$L_2^\vee$:        012211 011211 011111 011110 011100 011000 \\
$L_3^\vee$:                      111101 011101
\item $p=3$ case:\\
$L_1^\vee$: 111211 111111 111110 111100 111000 011000 001000 \\
$L_2^\vee$:        011211 011111 011110 011100 001100 \\
$L_3^\vee$:               001211 001111 001110 \\
$L_4^\vee$:               111101 011101 001101
\item $p=4$ case:\\
$L_1^\vee$: 111111 111110 111100 011100 001100 000100 \\
$L_2^\vee$:        011111 011110 001110 000110 \\
$L_3^\vee$:               001111 000111 \\
$L_4^\vee$:        111101 011101 001101 000101
\item $p=5$ case:\\
$L_1^\vee$: 122211 112211 111211 111111 111110 011110 001110 000110 000010 \\
$L_2^\vee$: 012211 011211 011111 001111 000111 \\
$L_3^\vee$: 001211
\item $p=6$ case:\\
$L_1^\vee$: 122211 112211 111211 111111 111101 011101 001101 000101 000001 \\
$L_2^\vee$: 012211 011211 011111 001111 000111 \\
$L_3^\vee$: 001211
\end{itemize}

\subsection{$E_6$ case}
\begin{itemize}
\item $p=1$ case:\\
$L_1^\vee$: 123212 123211 122211 122111 122101 112101 111101 111001 111000 110000 100000\\
$L_2^\vee$:    112211 112111 111111 111110 111100
\item $p=2$ case:\\
$L_1^\vee$: 112211 112111 111111 111110 111100 111000 110000 010000\\
$L_2^\vee$:  012211 012111 012101 011101 011100 011000 \\
$L_3^\vee$:   112101 111101 111001 011001  \\
$L_4^\vee$:    011111 011110
\item $p=3$ case:\\
$L_1^\vee$: 111111 111110 111100 111000 011000 001000\\
$L_2^\vee$:  111101 111001 011001 001001 \\
$L_3^\vee$:  011111 011110 011100 001100 \\
$L_4^\vee$:   011101 001101  \\
$L_5^\vee$:   001111 001110
\item $p=4$ case:\\
$L_1^\vee$: 122111 122101 112101 111101 111100 011100 001100 000100\\
$L_2^\vee$:  112111 111111 111110 011110 001110 000110 \\
$L_3^\vee$:   012111 012101 011101 001101  \\
$L_4^\vee$:    011111 001111
\item $p=5$ case:\\
$L_1^\vee$: 123212 123211 122211 122111 112111 111111 111110 011110 001110 000110 000010\\
$L_2^\vee$:    112211 012211 012111 011111 001111
\item $p=6$ case:\\
$L_1^\vee$: 123211 122211 122111 122101 112101 111101 111001 011001 001001 000001\\
$L_2^\vee$:   112211 112111 111111 011111 001111 001101  \\
$L_3^\vee$:    012211 012111 012101 011101
\end{itemize}

\subsection{$E_7$ case}
\begin{itemize}
\item $p=1$ case:\\
$L_1^\vee$: 1343212 1243212 1233212 1233211 1232211 1222211 1122211 1122111 1121111 1111111 1111110 1111100 1111000 1110000 1100000 1000000\\
$L_2^\vee$:    1232212 1232112 1232102 1232101 1222101 1122101 1121101 1111101 1111001 1110001   \\
$L_3^\vee$:      1232111 1222111 1221111 1221101 1221001 1121001
\item $p=2$ case:\\
$L_1^\vee$: 1122211 1122111 1121111 1111111 1111110 1111100 1111000 1110000 1100000 0100000\\
$L_2^\vee$:  0122211 0122111 0121111 0111111 0111110 0111100 0111000 0110000 \\
$L_3^\vee$:   1122101 0122101 0121101 0111101 0111001 0110001  \\
$L_4^\vee$:    1121101 1111101 1111001 1110001   \\
$L_5^\vee$:     1121001 0121001
\item $p=3$ case:\\
$L_1^\vee$: 1111111 1111110 1111100 1111000 1110000 0110000 0010000\\
$L_2^\vee$:  1111101 1111001 1110001 0110001 0010001 \\
$L_3^\vee$:  0111111 0111110 0111100 0111000 0011000 \\
$L_4^\vee$:   0111101 0111001 0011001  \\
$L_5^\vee$:   0011111 0011110 0011100  \\
$L_6^\vee$:    0011101
\item $p=4$ case:\\
$L_1^\vee$: 1221111 1221101 1221001 1121001 1111001 1111000 0111000 0011000 0001000\\
$L_2^\vee$:  1121111 1121101 1111101 0111101 0111100 0011100 0001100 \\
$L_3^\vee$:   1111111 1111110 0111110 0011110 0001110  \\
$L_4^\vee$:   0121111 0121101 0121001 0111001 0011001  \\
$L_5^\vee$:    0111111 0011111 0011101   \\
$L_6^\vee$:     1111100
\item $p=5$ case:\\
$L_1^\vee$: 1232112 1232102 1232101 1222101 1221101 1121101 1111101 1111100 0111100 0011100 0001100 0000100\\
$L_2^\vee$:  1232111 1222111 1221111 1121111 1111111 1111110 0111110 0011110 0001110 0000110 \\
$L_3^\vee$:    1122111 1122101 0122101 0121101 0111101 0011101   \\
$L_4^\vee$:     0122111 0121111 0111111 0011111
\item $p=6$ case:\\
$L_1^\vee$: 2343212 1343212 1243212 1233212 1232212 1232112 1232111 1222111 1221111 1121111 1111111 1111110 0111110 0011110 0001110 0000110 0000010\\
$L_2^\vee$:     1233211 1232211 1222211 1122211 1122111 0122111 0121111 0111111 0011111    \\
$L_3^\vee$:         0122211
\item $p=7$ case:\\
$L_1^\vee$: 1233211 1232211 1232111 1232101 1222101 1122101 1121101 1111101 1111001 1110001 0110001 0010001 0000001\\
$L_2^\vee$:   1222211 1222111 1221111 1221101 1221001 1121001 0121001 0111001 0011001  \\
$L_3^\vee$:    1122211 1122111 1121111 1111111 0111111 0011111 0011101   \\
$L_4^\vee$:     0122211 0122111 0122101 0121101 0111101    \\
$L_5^\vee$:       0121111
\end{itemize}

\subsection{$E_8$ case}
\begin{itemize}
\item $p=1$ case:\\
$L_1^\vee$: 13456423 12456423 12356423 12346423 12345423 12345422 12345322 12344322 12334322 12334312 12234312 12234212 12233212 12223212 12223211 12222211 12222111 12222101 11222101 11122101 11112101 11111101 11111001 11111000 11110000 11100000 11000000 10000000\\
$L_2^\vee$:      12345323 12345313 12345312 12344312 12344212 12334212 12333212 12333211 12233211 11233211 11223211 11222211 11222111 11122111 11112111 11111111 11111110 11111100     \\
$L_3^\vee$:          12234322 11234322 11234312 11234212 11233212 11223212 11123212 11123211 11122211 11112211
\item $p=2$ case:\\
$L_1^\vee$: 11234322 11234312 11234212 11233212 11223212 11223211 11222211 11222111 11222101 11122101 11112101 11111101 11111001 11111000 11110000 11100000 11000000 01000000\\
$L_2^\vee$:  01234322 01234312 01234212 01233212 01223212 01123212 01123211 01122211 01112211 01112111 01111111 01111110 01111100 01111000 01110000 01100000 \\
$L_3^\vee$:     11233211 01233211 01223211 01222211 01222111 01222101 01122101 01112101 01111101 01111001    \\
$L_4^\vee$:      11123212 11123211 11122211 11112211 11112111 11111111 11111110 11111100     \\
$L_5^\vee$:         11122111 01122111
\item $p=3$ case:\\
$L_1^\vee$: 11123212 11123211 11122211 11122111 11122101 11112101 11111101 11111001 11111000 11110000 11100000 01100000 00100000\\
$L_2^\vee$:  01123212 00123212 00123211 00122211 00112211 00112111 00111111 00111101 00111001 00111000 00110000 \\
$L_3^\vee$:   01123211 01122211 01122111 01122101 01112101 01111101 01111001 01111000 01110000  \\
$L_4^\vee$:    11112211 11112111 11111111 11111110 11111100 01111100 00111100   \\
$L_5^\vee$:     01112211 01112111 01111111 01111110 00111110    \\
$L_6^\vee$:      00122111 00122101 00112101
\item $p=4$ case:\\
$L_1^\vee$: 11112211 11112111 11112101 11111101 11111001 11111000 11110000 01110000 00110000 00010000\\
$L_2^\vee$:  01112211 00112211 00012211 00012111 00012101 00011101 00011001 00011000 \\
$L_3^\vee$:   11111111 11111110 11111100 01111100 01111000 00111000  \\
$L_4^\vee$:   01112111 01111111 01111110 00111110 00111100 00011100  \\
$L_5^\vee$:    01112101 00112101 00111101 00111001   \\
$L_6^\vee$:    00112111 00111111 00011111 00011110   \\
$L_7^\vee$:     01111101 01111001
\item $p=5$ case:\\
$L_1^\vee$: 11111111 11111110 11111100 11111000 01111000 00111000 00011000 00001000\\
$L_2^\vee$:  11111101 11111001 01111001 00111001 00011001 00001001 \\
$L_3^\vee$:  01111111 01111110 01111100 00111100 00011100 00001100 \\
$L_4^\vee$:   00111111 00111110 00011110 00001110  \\
$L_5^\vee$:   01111101 00111101 00011101 00001101  \\
$L_6^\vee$:    00011111 00001111
\item $p=6$ case:\\
$L_1^\vee$: 12222111 12222101 11222101 11122101 11112101 11111101 11111100 01111100 00111100 00011100 00001100 00000100\\
$L_2^\vee$:  11222111 11122111 11112111 11111111 11111110 01111110 00111110 00011110 00001110 00000110 \\
$L_3^\vee$:   01222111 01222101 01122101 00122101 00112101 00012101 00011101 00001101  \\
$L_4^\vee$:    01122111 00122111 00112111 00012111 00011111 00001111   \\
$L_5^\vee$:     01112111 01112101 01111101 00111101    \\
$L_6^\vee$:      01111111 00111111
\item $p=7$ case:\\
$L_1^\vee$: 12345313 12345312 12344312 12344212 12334212 12333212 12333211 12233211 12223211 12222211 12222111 11222111 11122111 11112111 11111111 11111110 01111110 00111110 00011110 00001110 00000110 00000010\\
$L_2^\vee$:    12334312 12234312 12234212 12233212 12223212 11223212 11123212 11123211 11122211 11112211 01112211 01112111 01111111 00111111 00011111 00001111   \\
$L_3^\vee$:      11234312 11234212 11233212 11233211 11223211 11222211 01222211 01222111 01122111 00122111 00112111 00012111     \\
$L_4^\vee$:       01234312 01234212 01233212 01233211 01223211 01123211 01122211 00122211 00112211 00012211      \\
$L_5^\vee$:          01223212 01123212 00123212 00123211
\item $p=8$ case:\\
$L_1^\vee$: 12333211 12233211 12223211 12222211 12222111 12222101 11222101 11122101 11112101 11111101 11111001 01111001 00111001 00011001 00001001 00000001\\
$L_2^\vee$:   11233211 01233211 01223211 01123211 00123211 00122211 00112211 00012211 00012111 00011111 00001111 00001101  \\
$L_3^\vee$:    11223211 11222211 11222111 11122111 11112111 11111111 01111111 00111111 00111101 00011101   \\
$L_4^\vee$:     11123211 11122211 11112211 01112211 01112111 00112111 00112101 00012101    \\
$L_5^\vee$:      01222211 01222111 01222101 01122101 01112101 01111101     \\
$L_6^\vee$:       01122211 01122111 00122111 00122101
\end{itemize}

\subsection{$F_4$ case}
\begin{itemize}
\item $p=1$ case:\\
$L_1^\vee$: 1342 1242 1232 1231 1221 1220 1120 1110 1100 1000\\
$L_2^\vee$:    1222 1122 1121 1111
\item $p=2$ case:\\
$L_1^\vee$: 1122 1121 1120 1110 1100 0100\\
$L_2^\vee$:  0122 0121 0120 0110 \\
$L_3^\vee$:   1111 0111
\item $p=3$ case:\\
$L_1^\vee$: 1111 1110 0110 0010\\
$L_2^\vee$:  0111 0011
\item $p=4$ case:\\
$L_1^\vee$: 1231 1221 1121 1111 0111 0011 0001\\
$L_2^\vee$:    0121
\end{itemize}

\subsection{$G_2$ case}
\begin{itemize}
\item $p=1$ case:\\
$L_1^\vee$: 13 12 11 10
\item $p=2$ case:\\
$L_1^\vee$: 11 01
\end{itemize}

%
%
\section{Appendix 2: Table of numbers $c_p$} \label{section_15}
%
%

We use the numbering of simple roots as in Appendix 1.

\begin{enumerate}
\item $A_r$ case ($r \geq 1$): $c_p=r+1$ (independent of $p$).
\item $B_r$ case ($r \geq 2$): $c_p=2r-p$ ($p\not=r$), $c_{r}=2r$.
\item $C_r$ case ($r \geq 3$): $c_p=2r-p+1$.
\item $D_r$ case ($r \geq 4$): $c_p=2r-p-1$ ($p\not=r-1$, $r$), $c_{r-1}=c_{r}=2r-2$.
\item $E_6$ case: \\[5pt]
\begin{tabular}{|c|c|c|c|c|c|c|} \hline
$p$    & 1   & 2 & 3 & 4 & 5 & 6 \\ \hline
$c_p$  & 12  & 9 & 7 & 9 & 12 & 11 \\ \hline
\end{tabular} \\
\item $E_7$ case: \\[5pt]
\begin{tabular}{|c|c|c|c|c|c|c|c|c|} \hline
$p$    & 1    & 2  & 3 & 4  & 5  & 6  & 7\\ \hline
$c_p$  & 17  & 11 & 8 & 10 & 13 & 18  & 14 \\ \hline
\end{tabular} \\
\item $E_8$ case: \\[5pt]
\begin{tabular}{|c|c|c|c|c|c|c|c|c|c|} \hline
$p$    & 1 & 2 & 3 & 4 & 5 & 6  & 7  & 8           \\ \hline
$c_p$  & 29 & 19 & 14 & 11  & 9 & 13 & 23 & 17         \\ \hline
\end{tabular} \\
\item $F_4$ case: $c_1=11$, $c_2=7$, $c_3=5$, $c_4=8$.
\item $G_2$ case: $c_1=5$, $c_2=3$.
\end{enumerate}

%
%
\section{Appendix 3: Homogeneous vector bundles} \label{section_16}
%
%
We present basic facts on homogeneous vector bundles according to ~\cite{Snow}.
Let $G$ be a semisimple complex Lie group of type $\Phi$ with a maximal torus $T$.
Let $B$ be a Borel subgroup containing $T$,
and let $P$ be a maximal parabolic subgroup of $G$ corresponding a simple root $\alpha_p$
in the fundamental system $\Delta$ attached to $B$.

Then $X_p=G/P$ is a connected compact complex manifold with the Picard number one,
and the homogeneous line bundles on $X_p$ are in one-to-one correspondence with the set of weights
$\Lambda_p=\Z \lambda_p$ and ${\rm Pic}(X_p) \simeq \Lambda_p$,
where $\lambda_p$ is the fundamental weight corresponding to $\alpha_p$.

Let $L$ be a holomorphic line bundle on $X_p$,
and let $c_1(L) \in H^2(X_p,{\R})$ be the first Chern class of $L$.
The first Chern class of a line bundle $L$ is identified with its associated weight $\lambda \in \Lambda_p \simeq {\rm Pic}(X_p)$
as follows:
\[
c_1(L) = \frac{i}{2\pi} \sum_{\alpha \in \Phi^+ \setminus \Phi_p^+} \langle \lambda, \alpha^\vee \rangle \, dx_\alpha \wedge d\bar{x}_\alpha,
\]
and hence we shall just write $c_1(L)=\lambda$.
Let $T_{X_p}$ be the tangent bundle of $X_p$. Then we find that
\[
c_1(T_{X_p}) = \sum_{\alpha \in \Phi^+ \setminus \Phi_p^+} \alpha = c_p \lambda_p,
\]
where $c_p$ is the number of \eqref{0202}.

The index ${\frak i}_p$ of $X_p$ is defined by the identity $-K_{X_p} = {\frak i}_p L_{\lambda_p}$
for the canonical bundle $K_{X_p}$ and the (ample) line bundle $L_{\lambda_p}$ associated with the fundamental weight.
For the first Chern class of $K_{X_p}$, we have
\[
c_1(K_{X_p}) = - \sum_{\alpha \in \Phi^+ \setminus \Phi_p^+} \alpha = - c_1(T_{X_p}) = - c_p \lambda_p.
\]
This derives the identity
\[
-K_{X_p} = c_p L_{\lambda_p} \quad ({\frak i}_p=c_p).
\]
Now the functional equation of the zeta function $\hat{\zeta}_p(s)$
associated with $(\Phi,\Delta,p)$ can be written as
\[
\hat{\zeta}_p(\langle c_1(K_{X_p}), \alpha_p^\vee \rangle -s) = \hat{\zeta}_p(s).
\]

\bibliographystyle{amsplain}
\bibliography{GPzeta_180614_revision}

%
\end{document}